\documentclass[11pt, reqno]{amsart}
\usepackage{amsfonts,amssymb,amscd,amsthm,amsbsy,bbm,epsf,calc}
\usepackage{hyperref}
\hypersetup{
    colorlinks = true,
	linkcolor = {blue},
	citecolor = {blue},
}
\usepackage{dsfont} 
\usepackage{color, float}
\usepackage{datetime}
\usepackage{bm}
\usepackage[pdftex]{graphicx}
	
\makeatletter
\def\l@section{\@tocline{1}{0pt}{1pc}{}{}}
\def\l@subsection{\@tocline{2}{0pt}{1pc}{4.6em}{}}
\def\l@subsubsection{\@tocline{3}{0pt}{1pc}{7.6em}{}}
\renewcommand{\tocsection}[3]{%
  \indentlabel{\@ifnotempty{#2}{\makebox[2.3em][l]{%
    \ignorespaces#1 #2.\hfill}}}#3}
\renewcommand{\tocsubsection}[3]{%
  \indentlabel{\@ifnotempty{#2}{\hspace*{2.3em}\makebox[2.3em][l]{%
    \ignorespaces#1 #2.\hfill}}}#3}
\renewcommand{\tocsubsubsection}[3]{%
  \indentlabel{\@ifnotempty{#2}{\hspace*{4.6em}\makebox[3em][l]{%
    \ignorespaces#1 #2.\hfill}}}#3}
\makeatother

\numberwithin{equation}{section}

\newcounter{hours}\newcounter{minutes}

\theoremstyle{definition}
\newtheorem{thm}{Theorem}[section]

\newtheorem{lem}[thm]{Lemma}
\newtheorem{cor}[thm]{Corollary}
\newtheorem{prop}[thm]{Proposition}

\newtheorem*{THM}{Theorem}

\theoremstyle{remark}                  
\newtheorem{rem}[thm]{Remark}
\newtheorem*{REM}{Remark}

\theoremstyle{definition}
\newtheorem{DEF}[thm]{Definition}

\def\real{{\mathbb R}}

\def\diam{\textnormal{diam}}

\def\det{\textnormal{det}}

\def\union{\bigcup}

\newcommand{\norm}[1]{\lvert#1\rvert}
\newcommand{\Norm}[2][]{\lVert#2\rVert_{#1}}
\newcommand{\inner}[2]{\langle #1,#2\rangle}

\def\arbitrary{A}
\def\sublevelset{S}
\def\ellipsoid{E}
\def\normal{N}

\def\innerdom{\Omega_0}
\def\outerdom{\Omega}
\def\innertarget{\Omegabar_0}
\def\outertarget{\Omegabar}
\newcommand{\direction}[2][]{\omega^{#1}_{#2}}

\DeclareMathOperator{\spt}{spt}
\DeclareMathOperator{\cl}{cl}
\DeclareMathOperator{\interior}{int}

\newcommand{\dist}[2]{d{\left(#1, #2\right)}}
\newcommand{\gdist}[2]{d_{\g{}}{\left(#1, #2\right)}}
\newcommand{\gbardist}[2]{d_{\gbar{}}{\left(#1, #2\right)}}

\newcommand{\Leb}[1]{\left\vert{#1}\right\vert_{\mathcal{L}}} 
\def\defin{:=}
\newcommand{\ball}[2]{B_{#1}{\left(#2\right)}}
\newcommand{\euclidinner}[2]{\left (#1, #2\right)}

\def\M{M}
\def\Mbar{\bar{M}}
\newcommand{\dVol}[1][]{d \textnormal{Vol}_{#1}}
\newcommand{\tansp}[2]{T_{#1}{#2}}
\newcommand{\cotansp}[2]{T^*_{#1}{#2}}
\newcommand{\cotanspM}[1]{\cotansp{#1}{M}}
\newcommand{\tanspM}[1]{\tansp{#1}{M}}
\newcommand{\cotanspMbar}[1]{\cotansp{#1}{\Mbar}}
\newcommand{\tanspMbar}[1]{\tansp{#1}{\Mbar}}
\newcommand\g[1]{g_{#1}}
\newcommand\gbar[1]{\bar{g}_{#1}}
\newcommand{\gnorm}[2][]{\lvert #2\rvert_{\g{#1}}}
\newcommand{\gbarnorm}[2][]{\lvert #2\rvert_{\gbar{#1}}}
\newcommand{\innerg}[3][x_0, \xbar_0]{\g{#1}\left(#2, #3\right)}
\newcommand{\innergbar}[3][\xbar_0]{\gbar{#1}\left(#2, #3\right)}
\newcommand{\raisecovect}[1]{{#1}^{\sharp}}

\DeclareMathOperator{\ch}{conv}

\newcommand{\plane}[2]{\Pi^{#2}_{#1}}
\newcommand{\projop}[1]{\pi_{#1}}
\newcommand{\proj}[2]{\projop{#1}{\left(#2\right)}}
\newcommand{\cone}[2]{I_{\pbar_0}(#1, #2)}
\newcommand{\polardual}[2]{{#1}^{\ast}_{#2}}

\def\Omegabar{\bar{\Omega}}
\def\arbitrarybar{\bar{\arbitrary}}
\def\Dbar{\bar{D}}
\def\xbar{\bar{x}}
\def\pbar{\bar{p}}

\def\qbar{\bar{q}}
\def\Vbar{\bar{V}}

\def\zhat{\widehat{z}}
\def\xbarhat{\widehat{\bar{x}}}

\def\xdot{\dot{x}}
\def\xddot{\ddot{x}}
\def\xbardot{\dot{\xbar}}
\def\xbarddot{\ddot{\xbar}}

\def\ttil{\tilde{t}}

\def\pmax{p_{\max}}

\def\xmax{x_{\max}}

\newcommand{\coord}[2]{[#1]_{#2}} 
\newcommand{\outerdomcoord}[1]{\coord{\outerdom}{#1}}
\newcommand{\innerdomcoord}[1]{\coord{\innerdom}{#1}}
\newcommand{\innertargetcoord}[1]{\coord{\innertarget}{#1}}
\newcommand{\outertargetcoord}[1]{\coord{\outertarget}{#1}}
\newcommand{\sublevelsetcoord}[1]{\coord{\sublevelset}{#1}}
\def\sectioncoord{\sublevelsetcoord{\xbar_0}}
\newcommand{\cExp}[2]{exp^c_{#1}({#2})}

\def\mountain{m}
\def\mountainhat{{m}}

\newcommand{\Gseg}[3]{\left[#1, #2\right]_{#3}} 

\def\subzero{\sublevelset_0}

\def\subzerocoord{\coord{\subzero}{\xbar_0, \z_0}}

\def\contactcoord{\coord{\contact}{\xbar_0}}
\def\contact{\sublevelset_0}
\newcommand{\linear}[3][]{l^{#2}_{#1}(#3)}

\newcommand{\e}[2][]{e^{#1}_{#2}}

\def\mountainhat{\widehat{\mountain}}

\newcommand{\w}[2][]{w^{#1}_{#2}}

\newcommand{\subcoord}[1][\xbar]{\coord{\sublevelset}{#1}}
\newcommand{\supsegment}[2]{ l(#1,#2) } 

\newcommand{\paren}[1]{\left(#1\right)}
\newcommand{\brackets}[1]{\left[#1\right]}
\newcommand{\curly}[1]{\left\{#1\right\}}

\def\G{G}
\def\Gstar{{G^\ast}}
\def\H{H}
\def\x{x}
\def\y{y}
\def\u{u}
\def\p{p}
\def\q{q}
\def\intervalG{I_{\G}}
\def\intervalH{I_{\H}}
\def\z{z}
\newcommand{\Gsubdiff}[3][\G]{\partial_{#1}#2\paren{#3}}

\def\Dx{D}
\def\Dxx{D^2}
\def\Gz{\G_{\z}}
\def\Hu{\H_{\u}}
\def\Gzz{\G_{\z\z}}

\def\Xop{\textnormal{exp}^{\G}}
\def\Xbarop{\textnormal{exp}^{\G}}
\def\Zop{Z^{\G}}
\def\Uop{U^{\G}}
\newcommand{\X}[3]{\Xop_{#1,#2}(#3)}
\newcommand{\Xbar}[3]{\Xbarop_{#1,#2}(#3)}
\newcommand{\Z}[3]{\Zop_{#1}(#2, #3)}
\newcommand{\U}[3]{\Uop_{#1} (#2, #3)}
\newcommand{\A}[2]{A_{#1#2}}
\newcommand{\Adual}[2]{A^{\ast}_{#1#2}}
\def\V{V}
\def\Vbar{\bar{\V}}
\def\etabar{\bar{\eta}}
\def\ybar{\bar y}
\def\pbarmapop{\bar{p}}
\def\pmapop{p}
\newcommand{\pbarmap}[3]{\pbarmapop_{#1,#2}(#3)}
\newcommand{\pmap}[3]{\pmapop_{#1,#2}(#3)}
\newcommand{\nondegmatrix}[4][]{E_{#1}(#2, #3, #4)}
\newcommand{\nondegmatrixentries}[1][]{E_{#1}}
\newcommand{\nondegmatrixinv}[4][]{E^{-1}_{#1}(#2, #3, #4)}
\newcommand{\dualnondegmatrix}[4][]{\bar{E}_{#1}(#2, #3, #4)}
\newcommand{\dualnondegmatrixentries}[1][]{\bar{E}_{#1}}
\newcommand{\dualnondegmatrixinv}[4][]{\bar{E}^{-1}_{#1}(#2, #3, #4)}
\DeclareMathOperator{\tensorop}{T}
\newcommand{\tensor}[5][]{\tensorop_{#1}\paren{#2, #3, #4, #5}}
\DeclareMathOperator{\dualtensorop}{T^{\ast}}
\newcommand{\dualtensor}[5][]{\dualtensorop_{#1}\paren{#2, #3, #4, #5}}
\newcommand{\pdiff}[2][]{\frac{\partial #1}{\partial #2}}
\newcommand{\twicepdiff}[3][]{\frac{\partial^2 #1}{{\partial #2}{\partial #3}}}
\newcommand{\twicepdiffsamevar}[2][]{\frac{\partial^2 #1}{{\partial #2^2}}}
\newcommand{\diff}[2][]{\frac{d#1}{d#2}}

\def\zdot{\dot{z}}

\def\perturb{\delta}
\def\perturbparam{\tau_{\perturb}}
\def\perturbparamtilde{\tilde{\tau}_{\perturb}}
\newcommand\nbhdof[2]{\mathcal{N}_{#1}\left(#2\right)}

\newcommand{\vertbar}[1]{\left. #1\right\vert}

\def\Gfiveupper{\overline{u}}
\def\Gfivelower{\underline{u}}
\def\lipbound{K_0}

\def\bdry{\partial}

\def\zmax{\z_{\max}}

\def\gendom{\mathfrak{g}}
\def\gendomdual{\mathfrak{h}}

\newcommand{\Gdual}[2]{{#1}^{\G}_{#2}} 

\def\basesection{S}
\def\basemountain{\mountain}
\def\basexbar{\xbar}
\def\baseq{\qbar}
\def\basez{\z}
\def\cm{x_{cm}}
\def\pcm{p_{cm}}
\def\supheight{\lambda_{\sup}}
\def\dilationconst{K}
\def\infbound{C_{\inf}}
\def\supbound{C_{\sup}}

\def\firsttime{t_0}
\def\sublevelsethat{\widehat{\sublevelset}}

\def\sptmountain{\widehat {\mountain}}
\def\sptpoint{\widehat x}
\def\sptxbar{\widehat \xbar}
\def\sptparam{\widehat s}
\def\sptseg{\widehat x}
\def\bdrysptpoint{\sptpoint^\partial}
\def\sptmountainshift{\tilde{\sptmountain}}

\def\maxpoint{x_{\max}}
\def\bdrymaxpoint{\maxpoint^\partial}
\def\maxparam{s_{\max}}
\def\maxseg{x_{\max}}

\def\Gdualybar{\ybar}
\def\polarheightconst{C_1}

\def\QQConvlower{\underline{u}_{\rm Q}}
\def\QQConvupper{\overline{u}_{\rm Q}}

\def\Nicelower{\underline{u}_{\rm N}}
\def\Niceupper{\overline{u}_{\rm N}}

\newcommand{\subdiff}[2]{\partial {#1}{(#2)}}

\def\S{\mathbb{S}}

\newcommand{\Gcone}[1]{K^{\G}_{#1}}
\def\Niceinterval{[\Nicelower, \Niceupper]}
\def\basexbarA{\xbar}
\def\basezA{\z}
\def\basemountainA{\mountain}
\def\contactpoint{\x_{\omega}}
\def\contactpointcoord{\pmap{\basexbarA}{\basezA}{\contactpoint}}
\def\sectioncoord{\sublevelsetcoord{\basexbarA, \basezA}}
\def\parallelpoint{\x_1}
\def\parallelpointcoord{\p_1}
\def\Gconez{{\z}_{\direction{}}}

\def\origvect{v}
\def\newvect{w}

\def\turnedfocus{\xbar_{\direction{}}}
\def\turnedfocuscoord{\pbar_{\direction{}}}
\def\turnedseg{\xbar_{\direction{}}}
\def\turnedz{\z_{\direction{}}}
\def\Gtil{\tilde{G}}
\def\gendomtil{\tilde{\gendom}}

\def\ftil{\tilde{f}}
\def\fhat{\hat{f}}
\newcommand{\clarke}[2]{\partial^C#1(#2)}

\begin{document}
\title{Pointwise estimates and regularity in geometric optics and other Generated Jacobian equations}

\author{Nestor Guillen}
\address[Nestor Guillen]{Department of Mathematics, University of Massachusetts at Amherst}
\email{nguillen@math.umass.edu} 

\author{Jun Kitagawa}
\address[Jun Kitagawa]{Department of Mathematics, Michigan State University}
\thanks{N. Guillen is partially supported by a National Science Foundation grant DMS-1201413. This material is based upon work supported by the National Science Foundation under Grant No. 0932078 000, while J. Kitagawa was in residence at the Mathematical Sciences Research Institute in Berkeley, California, during the fall semester of 2013. The authors would also like to thank the Fields Institute for Research in Mathematical Sciences, where they resided as part of the Thematic Program on Variational Problems in Physics, Economics and Geometry (Fall 2014).}
\email{kitagawa@math.msu.edu}

\begin{abstract}
  The study of reflector surfaces in geometric optics necessitates the analysis of certain nonlinear equations of Monge-Amp\`ere type known as generated Jacobian equations. This class of equations, whose general existence theory has been recently developed by Trudinger, goes beyond the framework of optimal transport. We obtain pointwise estimates for weak solutions of such equations under minimal structural and regularity assumptions, covering situations analogous to that of costs satisfying the A3-weak condition introduced by Ma, Trudinger and Wang in optimal transport. These estimates are used to develop a $C^{1,\alpha}$ regularity theory for weak solutions of Aleksandrov type. The results are new even for all known near-field reflector/refractor models, including the point source and  parallel beam reflectors and are applicable to problems in other areas of geometry, such as the generalized Minkowski problem.
\end{abstract}

\date{\today}

\date{}
\maketitle
\markboth{N. Guillen and J. Kitagawa}{Pointwise estimates and regularity in geometric optics and other generated Jacobian equations}

\tableofcontents 

\section{Introduction}\label{section: introduction}

\subsection{Overview.} This paper is concerned with the regularity theory of a broad class of Monge-Amp\`ere type equations spanning optimal transport and geometric optics. These may sometimes lie outside the scope of optimal transport but always have a Jacobian structure, namely
\begin{align}
  \det(D[T(x,Du,u) ]) = \psi(x,D u,u) \label{eqn: intro generated Jacobian equation},
\end{align}
for some $T:dom(T)\subseteq\Omega \times \mathbb{R}^d\times \mathbb{R} \to \mathbb{R}^d$ (see below). Admissible $u$ are ``convex'', i.e.
\begin{align*}
  D[T(x,Du,u)]\geq 0,
\end{align*}
which is necessary for \eqref{eqn: intro generated Jacobian equation} to be a degenerate elliptic PDE. To appreciate the generality of \eqref{eqn: intro generated Jacobian equation}, note it covers the real Monge-Amp\'ere equation, the $c$-Monge-Amp\'ere equation from optimal transport with cost $c$, the point source near-field reflector problem from geometric optics, several variations of the Minkowski problem and principal-agent problems in economics when dealing with a non-quasilinear utility function (references to the relevant literature are given below, see also Section \ref{section: examples}). Some of the corresponding $T$'s are
\begin{align*}
   T(x,\pbar,u) & = \pbar & \textnormal{(real Monge-Amp\'ere equation)},\\
  (D_xc)(x,T(x,\pbar, u)) & = -\pbar & \textnormal{(Optimal transport with cost $c(\cdot,\cdot)$)},\\
   T(x,\pbar,u) & = \frac{\pbar}{|\pbar|^2-(u-\pbar\cdot x)^2 } & \textnormal{(Point source near-field reflector)}.
\end{align*}
These and other examples will be discussed further in Section \ref{section: examples}. The mappings $T(x,\pbar,u)$ considered here will always be given by a \emph{generating function}. This means there is a function $G:dom(G)\subseteq \M\times\Mbar\times \mathbb{R}\to\mathbb{R}$ and associated ``exponential mappings'' $\Xbar{\x}{\u}{\cdot},\Z{\x}{\cdot}{\cdot}$ (see Section \ref{section: elements of generating functions} for definitions) such that
\begin{align*}
  T(x,\pbar,u) = \Xbar{\x}{\u}{\pbar}.
\end{align*}
For such $T$'s, \eqref{eqn: intro generated Jacobian equation} takes the form
\begin{align}
  \det(D^2u+(D^2_{\x}\G)(x,\Xbar{\x}{\u}{Du},\Z{\x}{Du}{u})) = \psi_G(\x,D \u,u) \label{eqn: generated Jacobian equation}. \tag{\textnormal{GJE}}
\end{align}
The corresponding convexity condition for $\u$ asks that it be of the form 
\begin{align*}
  \u(x) = \sup \limits_{\xbar} \limits\G(x,\xbar,\z_{\xbar})
\end{align*}
for some function $\xbar\to \z_{\xbar}\in \mathbb{R}$, $\xbar\in\Mbar$. Following work of Trudinger \cite{Tru14}, where the general framework for these equations is proposed, equation \eqref{eqn: generated Jacobian equation} will be called a ``Generated Jacobian Equation.'' The distinguishing feature of \eqref{eqn: generated Jacobian equation} is the dependence of the mapping $T$ on the values of the solution, which is not present in the case of optimal transport. Recall that in optimal transport, one has 
\begin{align*}
  \G(\x,\xbar,z) & =-c(\x,\xbar)+z,\\
  T(x,\pbar,u) & = T(x,\pbar) = \cExp{x}{\pbar}
\end{align*}
where $c(x,\xbar)$ denotes the cost function. In general, changing the ``height'' parameter $\z$ in $\G(\x,\xbar,\z)$ will result in a change in the shape of the function, and not merely a vertical shift; and the choice of coordinate systems must now take this into account. 

The aim of this work is to determine the differentiability of weak solutions to \eqref{eqn: generated Jacobian equation} under minimal on assumptions on the data (including the generating function $\G$, and the function $\psi_\G$). Specifically, we focus on weak solutions to \eqref{eqn: generated Jacobian equation}  in the ``Aleksandrov sense'', we also require the right hand side of the equation to bounded away from zero and infinity. The notion of Aleksandrov solution originated in the study of the real Monge-Amp\`ere equation and has also played a key role in optimal transport, see Definition \ref{def:Aleksandrov_solutions} for the setting of generated Jacobian equations. Our results are new even for the case of near-field reflector/refractor problems, covering situations where the condition \eqref{G3s} fails but \eqref{G3w} still holds.  The \eqref{G3w} condition was introduced by Trudinger in \cite{Tru14}, it generalizes (A3w) condition for the  Ma-Trudinger-Wang tensor \cite{MTW05}. Both the MTW tensor and the (A3w) condition play a central role in the regularity theory of optimal transport. 

This general framework makes our results applicable to problems beyond geometric optics. Roughly speaking, these results are in the same vein as Caffarelli's localization and differentiability estimate for the real Monge-Amp\'ere equation \cite{Caf90}; Figalli, Kim, and McCann's regularity theory for optimal transport maps under the (A3w) condition \cite{FKM13}, as well as work by V\'etois \cite{Vetois2015}, and by the authors on the strict $c$-convexity of $c$-convex potentials \cite{GK14}.

In the spirit of \cite{GK14}, the most important assumption on $\G$ is a synthetic version of \eqref{G3w}, which is roughly a ``quantitative quasiconvexity'' condition along $\G$-segments \eqref{QQConv}. This condition follows from \eqref{G3w} when $\G$ is smooth enough, it generalizes the (QQConv) condition introduced in \cite{GK14} for optimal transport (and in that case, it refines Loeper's maximum principle).

Our main results can be broadly separated in two parts. The first part consists of pointwise inequalities, Theorems \ref{thm: G-aleksandrov estimate} and \ref{thm: Sharp growth}, for \emph{$\G$-convex} functions $\u$ (see Definition~\ref{def: G-functions}). These are obtained under natural assumptions on $G$ and $\u$, one of the key conditions being \eqref{QQConv} mentioned above. The pointwise inequalities may be thought of as nonlinear analogues of the Blaschke-Santal\'o inequalities for the Mahler volume (see discussion in Section \ref{section: Mahler volume})

The second part comprises Theorems \ref{thm:strict_convexity} and \ref{thm:C1alpha_regularity}, in which we prove  \emph{strict} $\G$-convexity and interior $C^{1, \alpha}$ differentiability respectively of weak solutions $\u$ of \eqref{eqn: intro generated Jacobian equation}. This part relies on the pointwise inequalities in Theorems \ref{thm: G-aleksandrov estimate} and \ref{thm: Sharp growth} to show solutions satisfy a localization property (which leads to strict convexity) and an engulfing property (which leads to interior $C^{1,\alpha}$ estimates). Finally, we show that for $\G$ smooth enough, condition \eqref{QQConv} is implied by \eqref{G3w} (Theorem \ref{thm: G3w implies QQconv}). Precise statements for these results are given in Section \ref{section: Main results}.\\

\noindent \textbf{Acknowledgements.} The authors would like to thank Neil Trudinger for helpful correspondence and also for graciously sharing some of his unpublished work. The authors would also like to express their gratitude to the anonymous referee for their prompt and thorough reading of the manuscript, as well as suggested improvements.

\subsection{Strategy: Mahler volume and Monge-Amp\'ere equations.}\label{section: Mahler volume} In order to motivate the main results (Section \ref{section: Main results}) it will be convenient to recall several facts about the Mahler volume and relate it to the regularity theory for the real Monge-Amp\'ere equation. Let $\sublevelset$ be a convex set with non-empty interior and whose center of mass is $0$. The Mahler volume of $\sublevelset$, $m(\sublevelset)$, is defined as
\begin{align*}
  m(\sublevelset) := \Leb{\sublevelset}\Leb{\sublevelset^*},
\end{align*}	
where the set $\sublevelset^*$ is the polar dual of $\sublevelset$,
\begin{align*}
  \sublevelset^* := \{ y \in \mathbb{R}^n \mid (x',y)\leq 1,\;\;\forall\;x'\in \sublevelset \}.
\end{align*}
Then, the celebrated Blashke-Santal\'o and reverse Santal\'o inequalities together say that
\begin{align}\label{eqn:Blashke-Santalo}
  c_n^{-1}\; \leq \Leb{\sublevelset}\Leb{\sublevelset^*} \leq c_n.
\end{align}
These geometric inequalities imply (and are in fact equivalent) to certain pointwise inequalities for convex functions. Suppose a convex function $u:\mathbb{R}^n\to\mathbb{R}$ and an affine function $l$ are such that the set $\sublevelset:=\{u< l\}$ is nonempty, bounded, and with center of mass at $0$. Then, one can use \eqref{eqn:Blashke-Santalo} to prove the bounds
\begin{align}
  & \sup\limits_{\sublevelset }(l-u)^n\geq C_n^{-1}\Leb{\sublevelset}\Leb{\subdiff{u}{\tfrac{1}{2}\sublevelset}},\label{eqn: classical Sharp Growth}\\
  & (l(x)-u(x))^n \leq C_n \Leb{\sublevelset}\Leb{\subdiff{u}{\sublevelset}}, \;\;\forall\;\x\in\sublevelset\label{eqn: classical Aleksandrov}.
\end{align}
These two estimates are crucial in the theory of the Monge-Amp\`ere equation, in fact, they are the basis of Caffarelli's theorem on the strict convexity and differentiability of Aleksandrov solutions of the real Monge-Amp\'ere equation \cite{Caf90} (see the discussion in \cite[Part 2]{Caf92lectures}, and the discussion in \cite[Section 1.3]{GK14}). For now let us explain informally how these estimates may be used to obtain regularity to solutions of the Monge-Amp\`ere equation, namely $C^{1,1}$ regularity and strong convexity. Note first that a convex function $u$ is $C^{1,1}$ and strongly convex if and only if there is a $C>0$ such that for any supporting affine function $l_0$ and every small enough $h>0$ we have the inclusions
\begin{align}
  B_{C^{-1}\sqrt{h}} \subset  \{ u\leq l_0+h\} \subset B_{C\sqrt{h}}.\label{eqn: sublevel sets regular solution}
\end{align}
In other words, the level sets of $u$ are comparable to those of a paraboloid. Now, let $u$ be an Aleksandrov solution to $\det(D^2u) = f$, with  $\lambda \leq f \leq \Lambda$ (see Definition \ref{def:Aleksandrov_solutions}). Let us see to what extent would something like \eqref{eqn: sublevel sets regular solution} hold for $u$. Since $u$ is an Aleksandrov solution,  we have  $\Leb{\subdiff{u}{\sublevelset} }\sim \Leb{\sublevelset}$, in which case estimates \eqref{eqn: classical Sharp Growth}-\eqref{eqn: classical Aleksandrov} imply
\begin{align}
  \Leb{B_{C^{-1}\sqrt{h}}} \leq \Leb{\{ u\leq l_0+h\}} \leq \Leb{B_{C\sqrt{h}}}\label{eqn: sublevel sets measure}
\end{align}
for some $C>0$. This is a weaker assertion than \eqref{eqn: sublevel sets regular solution}, since we only compare the measures of the sublevel sets. The approach introduced by Caffarelli in the context of the real Monge-Amp\'ere equation \cite{Caf90, Caf91, Caf92}  shows how to go beyond \eqref{eqn: sublevel sets measure}  and obtain $C^{1,\alpha}$ regularity and strict convexity for $u$, and even \eqref{eqn: sublevel sets regular solution} and higher regularity if $f$ is assumed to be regular. See Guti\'errez's book \cite{Gut01} for a comprehensive exposition of these ideas.

For the $c$-Monge-Amp\'ere equation arising in optimal transport, Figalli, Kim, and McCann obtained \cite{FKM13} analogues of \eqref{eqn: classical Sharp Growth}-\eqref{eqn: classical Aleksandrov} under the (A3w) assumption of Ma, Trudinger and Wang, from where they obtained $C^{1,\alpha}$ and strict $c$-convexity estimates. In \cite{GK14} the authors introduced a condition on costs, ``quantitative quasiconvexity'' (QQConv), and used it to derive analogues of \eqref{eqn: classical Sharp Growth}-\eqref{eqn: classical Aleksandrov}. This (QQConv) condition is a refinement of Loeper's ``maximum principle'' \cite{Loe09} but at least for $C^4$ costs turns out to be equivalent to (A3w) (and thus to Loeper's condition itself).

Beyond $C^{1,\alpha}$ and $C^{2,\alpha}$ estimates, these inequalities are also an important tool in deriving $W^{2,p}$ estimates \cite{Caf90b} under extra assumptions on $f$, and more recently $W^{2,1+\epsilon}$ estimates under minimal assumptions \cite{DePFig2013,DePFigSav2013}. See the survey by Figalli and De Philippis \cite{DePFig2014} for a thorough discussion of recent optimal transport literature (see also Section \ref{section: examples}).

\subsection{Notation.} Before continuing with the introduction, let us set up some notational conventions used within the paper: $\inner{\cdot}{\cdot}$ will denote the evaluation pairing between an element of a vector space and an element of its dual space. $(\M,\g{}),(\Mbar,\gbar{})$ will denote $n$-dimensional complete Riemannian manifolds. Points in $\M$ will be denoted with $\x,\y...$ points in $\Mbar$ will be denoted with $\xbar,\ybar,...$ while $\Leb{\cdot}$ will denote the Riemannian volume on $(\M, \g{})$, $(\Mbar, \gbar{})$, or the associated Riemannian volumes on a tangent or cotangent space. $\gnorm[x]{\cdot}$ and $\gbarnorm[\xbar]{\cdot}$ will denote the length of tangent or cotangent vectors, with respect to the inner products $\g{x}$ and $\gbar{\xbar}$, and $\gdist{\cdot}{\cdot}$, $\gbardist{\cdot}{\cdot}$ will refer to the geodesic distances induced by the respective metrics. Also, we will use $A^{\interior}$, $A^{\cl}$, and $A^\bdry$ to refer to the interior, closure, and boundary of a set $A$ respectively.\\

Here is a summary of several other symbols, together with their definition number.\\

\begin{tabular}{lll}
Notation / Condition & Name & Definition location\\
\hline
$\G$, $\H$ & Generating function, dual function & Section \ref{subsection: basic definitions}\\
$\gendom,\gendomdual$ &  & Section \ref{subsection: basic definitions}\\
\eqref{G5}, \eqref{Lip}, $\lipbound$ &  & Definition~\ref{definition: G5}\\
\eqref{Twist}, \eqref{DualTwist} &  & Definition~\ref{definition: Bi-twist}\\
\eqref{Nondeg} &  & Definition~\ref{def: nondeg}\\
$\nondegmatrix{\x}{\xbar}{\z}$, $\dualnondegmatrix{\x}{\xbar}{\z}$ &  & Definition~\ref{def: nondeg}\\
$\pmap{\xbar}{\z}{x}$, $\pbarmap{x}{\u}{\xbar}$  &  & Definition~\ref{DEF: G-coordinates}\\
$\coord{\bar \arbitrary}{x,\u}$, $\coord{\arbitrary}{\xbar,\z}$ &  & Definition~\ref{DEF: G-coordinates}\\
$\Gseg{\x_0}{\x_1}{\xbar, \z}$ & $\G$-segment  & Remark~\ref{rem: G-segment notation}\\
$\X{\xbar}{\z}{\cdot}$, $\Xbar{\x}{\u}{\cdot}$, $\Z{\x}{\cdot}{\cdot}$ & $\G$-exponential mappings & Definition~\ref{DEF: exponential mappings}\\
\eqref{DualDomConv}, \eqref{DomConv} &  & Definition~\ref{def: DomConv}\\
\eqref{QQConv}, \eqref{DualQQConv} &  & Definition~\ref{definition: QQConv}\\
$\mountain,\mountain_0,\ldots$ & $\G$-affine functions  & Definition~\ref{def: G-functions}\\
$\u,\u_0,\ldots$ &  $\G$-convex functions & Definition~\ref{def: G-functions}\\
$\partial_{\G}u$ & $\G$-subdifferential  & Definition~\ref{DEF: G-subdifferentials}\\
& \emph{nice}, \emph{very nice} functions  & Definition~\ref{def: very nice}\\
&\emph{very nice} constant& Remark~\ref{rem: universal constants}\\
$\arbitrary^*_{p,q,\lambda}$ & Polar dual  & Definition~\ref{DEF: polar dual}\\
$\arbitrary^{\G}_{\x,m,\lambda}$ & $\G$-dual  & Definition~\ref{DEF: G-dual}\\
$K^{\G}_{\x,\sublevelset}(\cdot)$ & $\G$-cone  & Definition~\ref{DEF: G-cone}\\
$\plane{\arbitrary}{\direction{}}$ & Supporting hyperplane  & Definition~\ref{DEF:notation_supporting_plane}\\
\end{tabular}

\section{Statement of main results}\label{section: Main results}

In this section we state the exact form of our main results. The precise statement themselves involve a great deal of notation that will not be introduced until Section \ref{section: elements of generating functions}, however, for the sake of having all the main results stated in one section, we choose to present them here. Thus, the reader is advised to skim through this section on a first reading and return to it after reading the elements of generating functions in Section \ref{section: elements of generating functions}.\\
	
\textbf{Structural assumptions}. All of the theorems below require a number of structural assumptions on $\G$ and its domain of definition. In many important subclasses of examples (i.e. optimal transport, near field problems in optics) each of these structural assumptions are known to be necessary conditions for the regularity of solutions. 

Then, we are given $n$-dimensional Riemannian manifolds $M,\Mbar$, a \emph{generating function} which is a function $\G:\M\times \Mbar\times \real\to\mathbb{R}$; we are also given domains $\outerdom\subset M$, $\outertarget\subset \Mbar$, and $\gendom \subset \M\times \Mbar\times \real$. We assume these objects have the following properties (see Section \ref{section: elements of generating functions} for details)
\begin{align*}
  \begin{array}{l}
     \textnormal{(I)}\; \G(\x,\xbar,\z) \textnormal{ is } C^2 \textnormal{ in the sense that all purely mixed second derivatives exist and are }\\
	 \text{continuous in all of } \M\times \Mbar\times \real.\;\textnormal{ Moreover, } \G_z<0.\\
     \\
	 \textnormal{(II)}\;\; \textnormal{There are constants } -\infty\leq \Gfivelower<\Gfiveupper\leq \infty  \textnormal{ and }\lipbound>0\textnormal{ such that } \outerdom, \outertarget, \textnormal{ and } \gendom \textnormal{ satisfy}\\
	 \textnormal{\eqref{G5}, \eqref{Lip}, \eqref{DomConv}, and \eqref{DualDomConv}} \textnormal{ with respect to the interval } (\Gfivelower, \Gfiveupper).\\
	 \\	 
	 \textnormal{(III)}\; \G \textnormal{ satisfies } \eqref{Twist}, \eqref{DualTwist}, \eqref{Nondeg}, \eqref{QQConv} \textnormal{ and } \eqref{DualQQConv}.\\
  \end{array}
\end{align*}
We are also given a function $u$ (eventually, the solution to \eqref{eqn: generated Jacobian equation}), assumed to satisfy the following (see Definition~\ref{def: very nice} and Remarks~\ref{rem: very nice interval} and \ref{rem: universal constants})
\begin{align*}
  \textnormal{(IV) } u:\outerdom\to\mathbb{R} \textnormal{ is a \emph{very nice} } \G \textnormal{-convex function with an associated} \;\;\;\; \quad\quad\quad\quad\quad\quad\quad \\
   \textnormal{\emph{very nice} interval } \Niceinterval.  \quad\quad\quad\quad\quad\quad\quad\quad\quad\quad\quad\quad\quad\quad\quad\quad \quad\quad\quad\quad\quad\quad\quad\quad\quad\quad	
\end{align*}	
\begin{REM} The notion of \emph{very nice} for a $\G$-convex function is explained in Definition \ref{def: very nice}, this notion is irrelevant in optimal transport, where all $\G$-convex functions are automatically \emph{very nice}. The necessity for this notion for general Generated Jacobian equations is illustrated by phenomena present in the near field problem (see  Karakhanyan and Wang \cite[Theorem A,B]{KarWang10}). This is discussed in detail at the end of Section \ref{section: literature reflector problems}.
\end{REM}

Finally, in all what follows $M\geq 1$ will denote the constant associated to $\G$ by \eqref{QQConv} and \eqref{DualQQConv} with the interval $\Niceinterval$.\\

The first result is an Aleksandrov type estimate, which will play the role that \eqref{eqn: classical Aleksandrov} plays for the standard Monge-Amp\'ere theory.
\begin{thm}\label{thm: G-aleksandrov estimate}
  Suppose $\basemountainA(\cdot)\defin\G(\cdot, \basexbarA, \basezA)$ for some $(\basexbar, \basez)\in\outertarget\times\real$ is a \emph{nice} $\G$-affine function and $\sublevelset\defin\curly{u\leq \basemountainA}$. Also assume that $\sectioncoord\subset B\subset 3B\subset \coord{\outerdom}{\basexbarA, \basezA}$ for some ball $B$ in $\cotanspMbar{\basexbarA}$ (which may be of any radius). Then, there exist \emph{very nice} constants $\epsilon$, $C>0$ such that for any $\direction{1}\in\S^{n-1}\subset \cotanspMbar{\basexbarA}$ and $\x_0\in\sublevelset^{\interior}$, if $\diam{(\sublevelset)}<\epsilon$ then 
  \begin{align*}
    (\basemountainA(\x_0)-u(\x_0))^n\leq \frac{C\dist{\p_0}{\plane{\sectioncoord}{\direction{1}}}}{\supsegment{\sectioncoord}{\direction{1}}}\Leb{\sublevelset}\Leb{\partial_{\G}u(\sublevelset)},
  \end{align*}
  where $\p_0\defin \pmap{\basexbarA}{\basezA}{\x_0}$ and $\supsegment{\sectioncoord}{\direction{1}}$ is defined as the maximum length among all line segments parallel to $\direction{1}$ and contained in $\sectioncoord$.
\end{thm}

The second result gives a generalization of estimate \eqref{eqn: classical Sharp Growth}.

\begin{thm}\label{thm: Sharp growth} Suppose $\basemountain(\cdot):=\G(\cdot, \basexbar, \basez)$ for some $(\basexbar, \basez)\in\outertarget\times\real$ is a $\G$-affine function such that $\Nicelower\leq \basemountain \leq\Niceupper$ on $\outerdom^{\cl}$. Writing $\sublevelset:=\curly{\x\in\outerdom\mid u(x)\leq \basemountain(x)}$, there exist \emph{very nice} constants $C$, $\dilationconst>0$ such that for any $\arbitrary\subset \outerdom$ with $\coord{\arbitrary}{\basexbar, \basez}$ connected, satisfying 
\begin{align}
\dilationconst M\coord{\arbitrary}{\basexbar, \basez}&\subset \coord{\sublevelset}{\basexbar, \basez},\label{eqn: dilation condition}\\
\sup_{\arbitrary}{\basemountain}+ \sup_{\arbitrary}{(\basemountain-u)}&< \Gfiveupper,\label{eqn: section not too deep}
\end{align}
 we have
\begin{align*}
 \sup_{\arbitrary}{(\basemountain-u)}^n\geq C\Leb{\arbitrary}\Leb{\Gsubdiff{u}{\arbitrary}}.
\end{align*}
Here $\dilationconst M\coord{\arbitrary}{\basexbar, \basez}$ is the dilation of $\coord{\arbitrary}{\basexbar, \basez}$ with respect to its center of mass $\pmap{\basexbar}{\basez}{\cm}$.
\end{thm}

Our next two results concern weak solutions $u$ to \eqref{eqn: generated Jacobian equation},  in the sense of Aleksandrov (see Definition~\ref{def:Aleksandrov_solutions}). We use the notation $\innerdom$ for the support of the Radon measure $\Leb{\Gsubdiff{u}{\cdot}}$ and $\innertarget:=\Gsubdiff{u}{\innerdom}$. The first of the two theorems deals with the strict $\G$-convexity of $u$.
\begin{thm}\label{thm:strict_convexity}
Suppose $u$ is a \emph{very nice} Aleksandrov solution of \eqref{eqn: generated Jacobian equation}. If $\innerdom^{\cl}\subset \outerdom^{\interior}$ and $\innertarget^{\cl}\subset \outertarget^{\interior}$, and $\coord{\innertarget}{\x_0, u(\x_0)}$ is convex for some $\x_0\in \innerdom^{\interior}$, then $u$ is strictly $\G$-convex at $\x_0$, i.e. if $\xbar_0\in \Gsubdiff{u}{\x_0}$, then the set $\curly{x\in\outerdom\mid u(\x)=\G(\x,\xbar_0, \H(\x_0, \xbar_0, u(\x_0))}$ is the singleton $\curly{\x_0}$.
\end{thm}

We prove (interior) $C^{1,\alpha}$ regularity of weak solutions (provided they are \emph{very nice}). The proof relies on the previous theorems as well as extensions of the \emph{engulfing property} of sublevelsets of solutions for the real Monge-Amp\'ere Equation (see \cite{FM04},\cite[Section 9]{FKM13}).
	
\begin{thm}\label{thm:C1alpha_regularity}
Suppose in addition to the assumptions of Theorem~\ref{thm:strict_convexity} above, that $\G$ is a $C^{1, \alpha}$ function in the $\x$ variable for some $\alpha\in(0, 1]$, uniformly in the $(\xbar, \z)$ variables. Then there exists an $\beta\in (0, 1]$ such that $u\in C_{loc}^{1, \beta}(\innerdom^{\interior})$.
\end{thm}	

Our final result connects the \eqref{G3w} condition introduced by Trudinger \cite{Tru14} with the conditions \eqref{QQConv} and \eqref{DualQQConv}.
\begin{thm}\label{thm: G3w implies QQconv}
Assume there are $-\infty\leq \Gfivelower<\Gfiveupper\leq \infty$ such that $\outerdom$, $\outertarget$, and $\gendom$ satisfy \eqref{G5}, \eqref{DomConv}, and \eqref{DualDomConv} with respect to $(\Gfivelower, \Gfiveupper)$. Also assume $\G$ is $C^4$, by which we mean all derivatives of up to order $4$ total, with at most two derivatives ever falling on one variable $\x$, $\xbar$, or $\z$ at once, exist and are continuous and $\G$ satisfies \eqref{Twist}, \eqref{DualTwist}, \eqref{Nondeg}, and \eqref{G3w}. Then $\G$ also satisfies both \eqref{QQConv} and \eqref{DualQQConv}.
\end{thm}

\subsection{Overview of the rest of the paper.} A detailed discussion of examples of \eqref{eqn: generated Jacobian equation} covered by our results is carried out in Section~\ref{section: examples}, examples discussed include the near-field reflector problem and the generalized Minkowski problem. In Section~\ref{section: elements of generating functions} we review the elements of generating functions and the associated Jacobian equations \eqref{eqn: generated Jacobian equation} (following to a great extent the ideas in \cite{Tru14}), we also introduce the \eqref{QQConv} and \eqref{DualQQConv} conditions on $\G$.

In Section~\ref{section: Aleksandrov estimate} we show how \eqref{QQConv} and \eqref{DualQQConv} lead to the Aleksandrov-type estimate, Theorem \ref{thm: G-aleksandrov estimate}. In Section~\ref{section: sharp growth estimate} we prove the sharp growth estimate, Theorem \ref{thm: Sharp growth}. In Section \ref{section: localization} we use the pointwise estimates to prove a localization property for weak solutions, and their strict convexity (Theorem \ref{thm:strict_convexity}). The work of all previous sections are combined in Section \ref{section: engulfing} to prove solutions are $C^{1,\alpha}$ (Theorem \ref{thm:C1alpha_regularity}).

Finally, in Section~\ref{section: G3w implies QQConv} we prove (Theorem \ref{thm: G3w implies QQconv}) that the condition \eqref{G3w} (defined by Trudinger in \cite{Tru14} to obtain classical regularity in generated Jacobian equations), implies conditions \eqref{QQConv} and \eqref{DualQQConv}. 

\section{Examples}\label{section: examples}

\subsection{Point source, near-field reflector}\label{section: literature reflector problems} 
For our first example, we spend some time discussing the \emph{near-field reflector problem}, as it is a well-studied problem that gives rise to a generated Jacobian equation \eqref{eqn: generated Jacobian equation} which does \emph{not} arise from an optimal transport problem, and as such displays many subtle difficulties not seen in the optimal transport case.

The engineering literature on reflector design is too large to review in detail here, but let us point out the reader to a few references, such as Oliker \cite{Oli89} Kochengin and Oliver \cite{KO98} and Janssen and Maes \cite{JanssenMaes92} for the case of cylindrical reflectors. For more on the literature and the exposition to follow, the reader is directed to the survey article \cite{Oli03} by Oliker, the discussion in Karakhanyan and Wang \cite{Kar14}. See also the classical monograph by Rusch and Potter \cite{RuschPotter1970} for a broader introduction to the engineering of antennas.
\begin{figure}[H] 
  \centering
    \includegraphics[height=.40\textwidth]{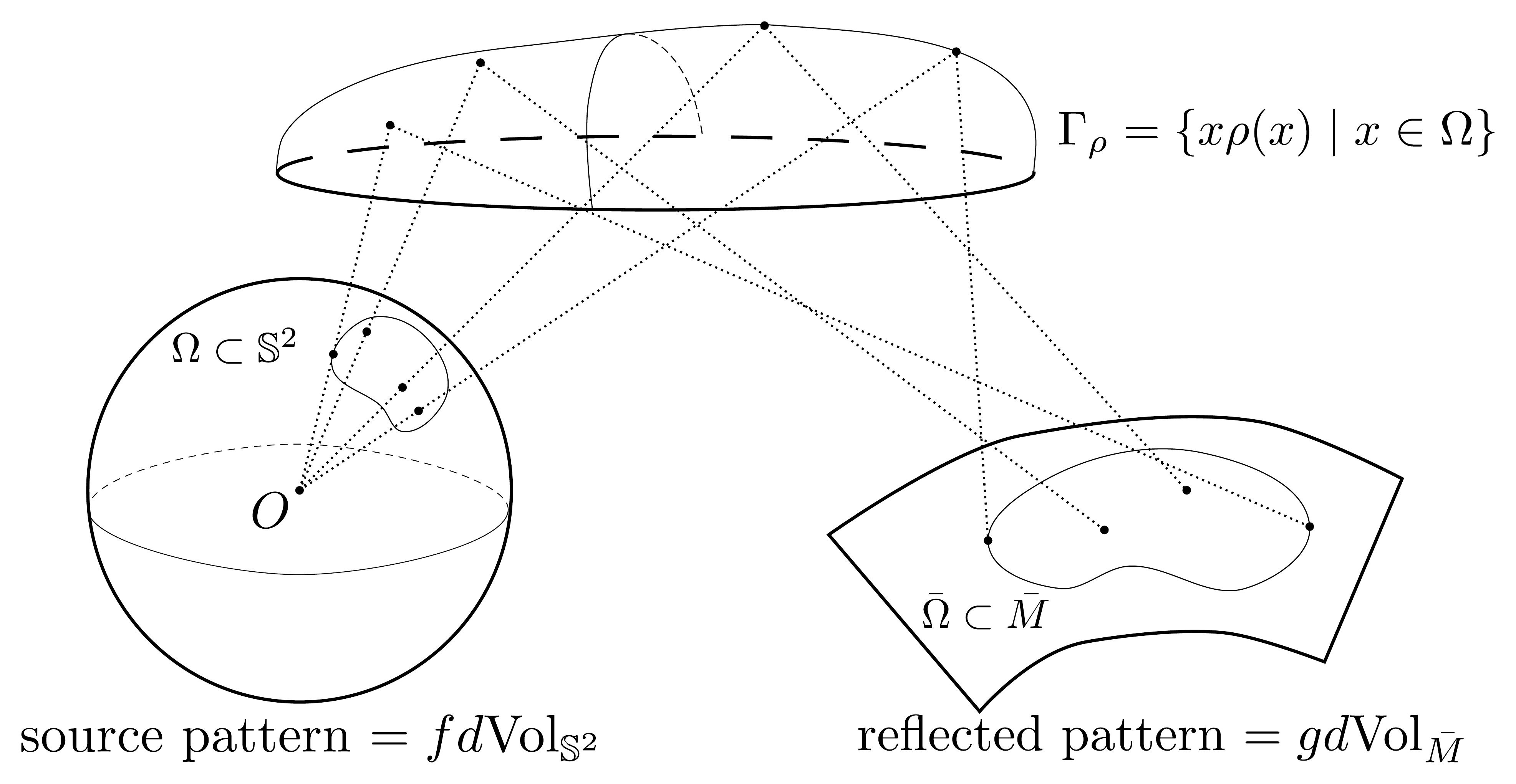}
     \caption{}\label{figure: Sec1fig1}
\end{figure}
We are given a light source at some point $O\in\mathbb{R}^3$, that shines through a ``source region'' $\outerdom\subset \S^2$ and a ``target region'' to be illuminated, which is a region $\outertarget$ contained within some codimension one surface $\Mbar \subset \mathbb{R}^3$. Moreover, the light source may not have a uniform intensity, instead it radiates energy through $\outerdom$ modeled by some absolutely continuous measure $f \dVol[\S^2]$.

The goal is now to build a \emph{reflector}: a (perfectly reflective) surface $\Gamma_\rho\subset\mathbb{R}^3$ given by the radial graph of some function $\rho:\outerdom\to\mathbb{R}$ with the property that light emanating from $O$ according to the distribution $f$ is reflected off to arrive in $\outertarget$. This problem is severely underdetermined, thus we also assume that we are given an absolutely continuous measure $g\dVol[\Mbar]$ supported on $\outertarget$, and the reflector is required to recreate this measure as the resulting illumination pattern. The assumption of perfect reflection implies that the total masses of $f$ and $g$ must be equal. The usual plan of attack for this problem is to first assume the \emph{geometric optics approximation}, in which light rays are treated like particles, completely ignoring any wave-like behavior that may be present.

To motivate an elementary method of constructing such a desired reflector, consider the case where the target measure is not absolutely continuous, but a Dirac delta concentrated at a point $\xbar \in \outertarget$. Then the reflector can be taken as any ellipsoid of revolution with foci $O$ and $\xbar$. For $2a>\norm{\xbar}$ there is a unique ellipsoid of revolution with foci $O$ and $\xbar$ whose major axis has length equal to $2a$. A straightforward computation shows that such an ellipsoid can be written as the radial graph of a function $e(\cdot,\xbar,a):\mathbb{S}^2\to\mathbb{R}_+$ defined by
\begin{align*}
  e(\x,\xbar,a) = \frac{a^2-\tfrac{1}{4}|\xbar|^2}{a-\tfrac{1}{2}\euclidinner{\x}{\xbar}}
\end{align*}
where $\euclidinner{\x}{\xbar}$ is the Euclidean inner product in $\real^{3}$. We can view $a$ here as a scalar parameter controlling the eccentricity of the ellipse, in particular we see there is a one parameter family of reflectors that solve our problem.
\begin{figure}[H]
  \centering
    \includegraphics[height=.40\textwidth]{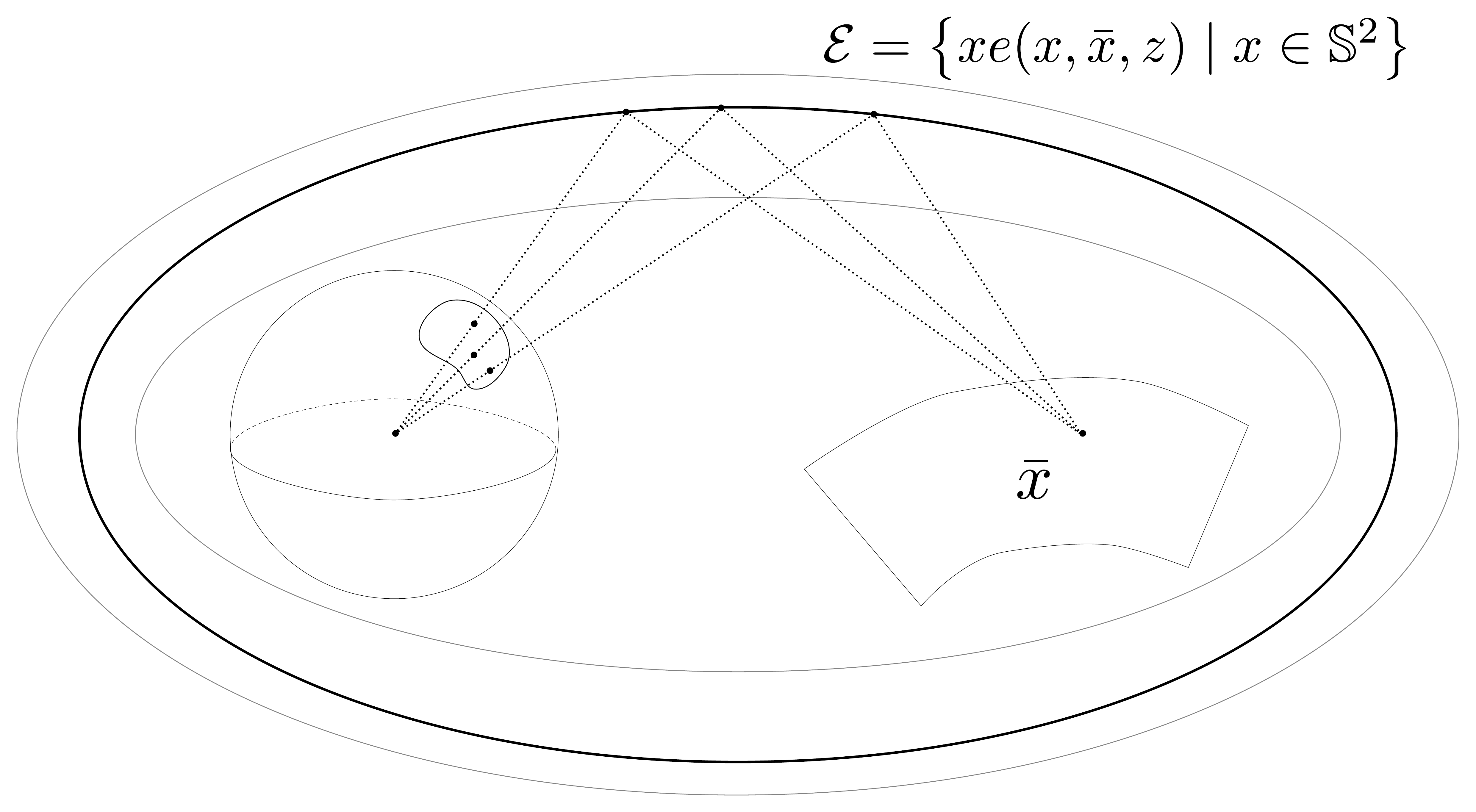}
     \caption{}\label{figure: Sec1fig2}
\end{figure}

If the target measure is now a finite sum of weighted Dirac deltas, we can take the reflector to be the boundary of the intersection of the same number of ellipsoids, each with one focus at $O$ and the other at a point where the sum of deltas is supported. By adjusting the scalar parameters, we can ensure each point in the target receives the correct amount of energy.

One can then approximate the absolutely continuous target measure by a sequence of such finite sums of Dirac deltas, and rigorously justify a limiting process to obtain a reflector that is the boundary of an intersection of (an infinite) family of ellipsoids, or in terms of $\rho$: 
\begin{align}\label{eqn:intro_ellipsoid_envelope}
  \rho(x) = \inf_{(\xbar, a)\in \mathcal{A}}e(x, \xbar, a)
\end{align}
for some appropriate collection $\mathcal{A}$. 
\begin{figure}[H]
  \centering
    \includegraphics[height=.40\textwidth]{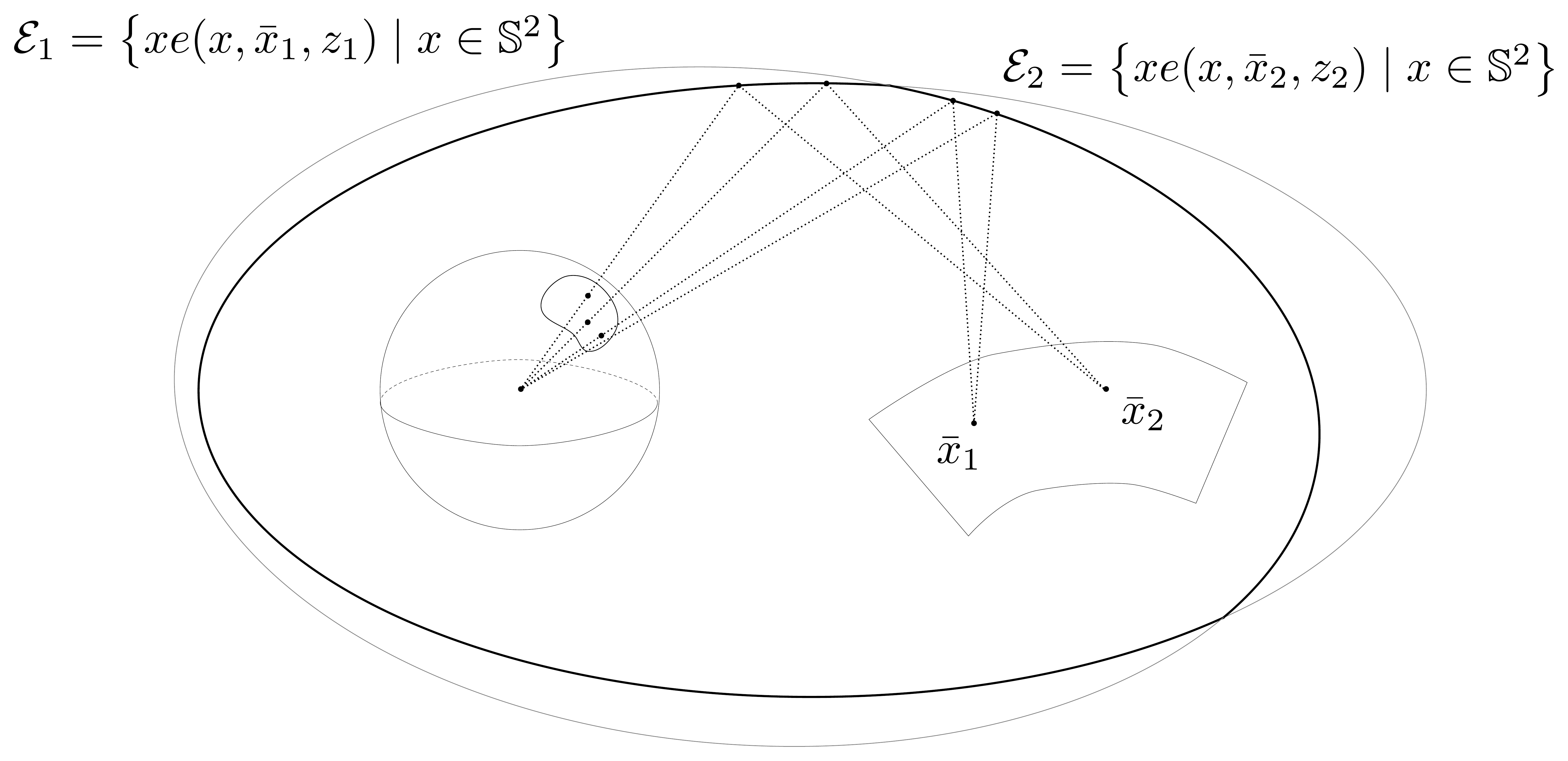}
     \caption{}\label{figure: Sec1fig3}
\end{figure}
This representation of $\rho$ can be interpreted as a form of ``concavity'' of $\rho$, where instead of hyperplanes as in the usual case of a concave function, $\rho$ is supported from above by graphs of ellipsoids which serve as some sort of ``fundamental shape''. Indeed, if we take 
\begin{align*}
 \G(x, \xbar, \z):=\frac{1}{e(x, \xbar, \z^{-1})},
\end{align*}
defined in
\begin{align*}
  \gendom = \curly{(\x,\xbar,\z) \in \mathbb{S}^2\times \Mbar \times\mathbb{R}_+\mid \tfrac{1}{2}\z\norm{\xbar}<1},
\end{align*}
then $\rho$ will exactly be a \emph{$\G$-convex} function as in Definition~\ref{def: G-functions}. 

When $\Mbar$ can be written as the graph of a function over a portion of $\real^{2}$, it can easily be verified that our choice of $\G$ coincides with that of Trudinger in \cite[(4.15)]{Tru14}. In the particular case when $\Mbar$ is contained in a hyperplane parallel to $\real^2$ lying below $\real^2$, from the formulae in \cite[Section 4]{Tru14} it can be seen that $\G$ satisfies conditions \eqref{Twist}, \eqref{DualTwist}, \eqref{Nondeg}, and \eqref{G3w}, and \eqref{G5} with $(\Gfivelower, \Gfiveupper)=(0, \infty)$. The main difference here from the usual case of convexity / concavity (or indeed, from the optimal transport case known in the literature as \emph{$c$-convexity}), is that when the scalar parameter $\z$ is changed in any of the functions $e$ forming the infimum in \eqref{eqn:intro_ellipsoid_envelope}, there is a nontrivial change in the shape that goes beyond a simple translation or dilation.

Next one can consider what is known as the \emph{ray-tracing map}, a map $T_\rho: \outerdom\to\outertarget$ that simply gives the location that a beam originating through $x$ ends up after reflecting off of $\Gamma_\rho$.

It can be seen that to obtain the desired illumination property, it is sufficient to impose a \emph{prescribed Jacobian equation} of the form $f(x)\det{DT_\rho(x)}=g(T_\rho(x))$. From the form \eqref{eqn:intro_ellipsoid_envelope} of $\rho$, and a calculation of $T_\rho$ in terms of the derivative of $\rho$, this equation can be re-written as a generated Jacobian equation of the form \eqref{eqn: generated Jacobian equation}. In fact, the choice of $u=\rho^{-1}$ will be a solution of \eqref{eqn: generated Jacobian equation} with our above choice of $\G$, and a certain $\psi_\G$ involving the densities $f$ and $g$.

\begin{figure}[H]
  \centering
    \includegraphics[height=.40\textwidth]{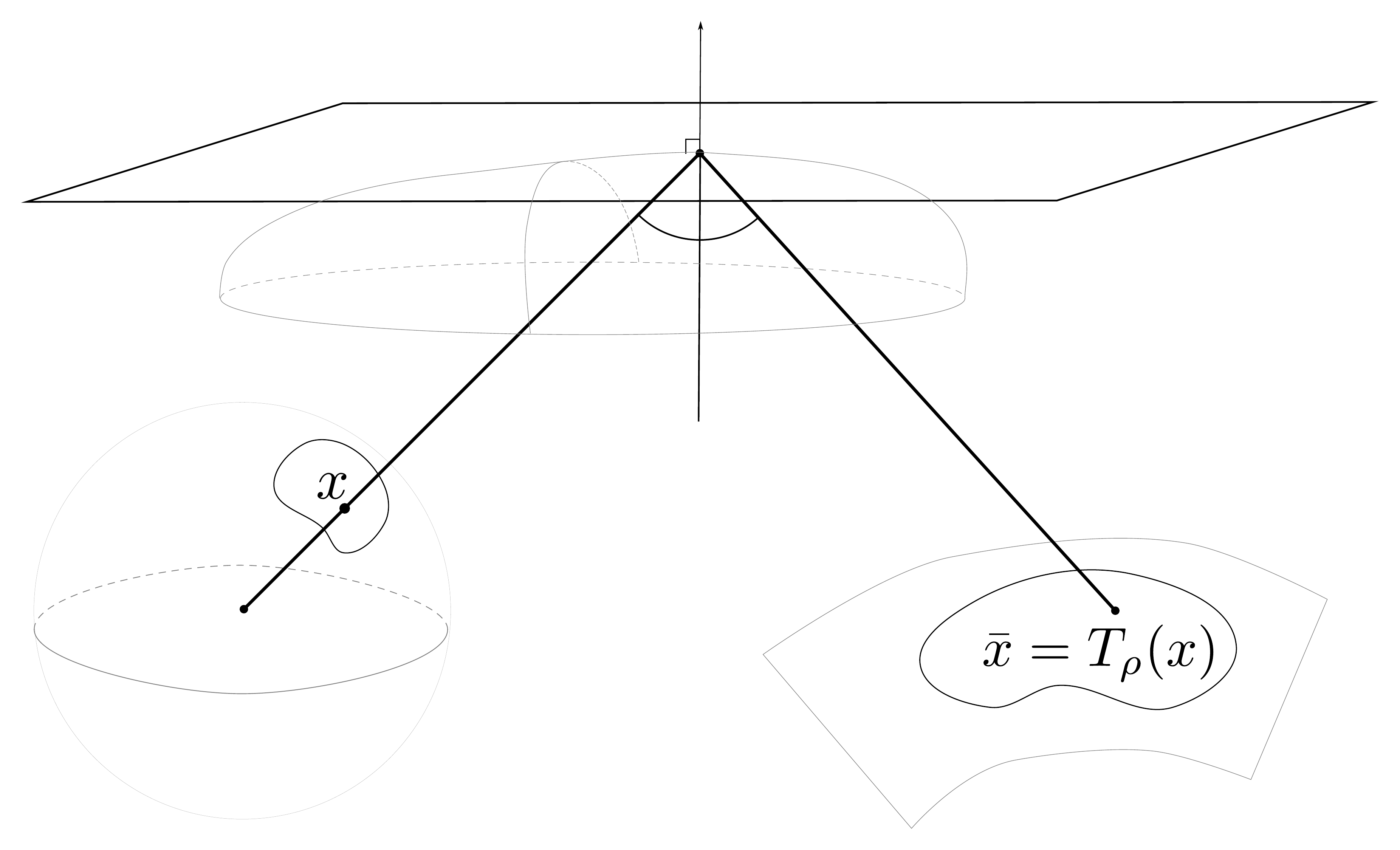}
     \caption{}\label{figure: Sec1fig4}
\end{figure}

It should be noted that a question of deep physical interest now is regularity of the reflector. Indeed, non differentiability of a reflector would cause diffraction phenomena, which may not be accurately modeled by the geometric optics approximation. In the case of refraction problems which also give rise to generated Jacobian equations, singularities can lead to chromatic aberrations, which also lie outside the realm of geometric optics. 

Recent work of Karakhanyan and Wang \cite{KarWang10} guarantees regularity ($C^{2,\alpha}$) for reflectors. Their main result illustrates some of the complexities that arise once we leave the optimal transport framework (see in particular Remark \ref{rem: different regularity at different heights} below). 

\begin{THM}\cite[see Theorems A-B]{KarWang10} Suppose that
	
	\begin{enumerate}
	  \item $\outerdom$, $V\subset \mathbb{S}^{n-1}$, $\outerdom\cap V=\emptyset$, $\outerdom$ has Lipschitz boundary.
	  \item $\outertarget$ is a region in a convex hypersurface $\Mbar$, given by a radial graph of some smooth function over $V$.
	  \item $f:\outerdom\to\mathbb{R},$ $g:\outertarget\to\mathbb{R}$ smooth, strictly positive functions with the same total mass.
	  \item $\partial \outertarget$ is ``$R$-convex.''
	\end{enumerate}
	
	Then, there is a reflector that is contained in a region close to $O$ that is smooth.
\end{THM}
The authors continue on to give finer conditions to obtain regularity (see \cite[Theorem C]{KarWang10}). In particular, they provide a condition on the second fundamental form of the target hypersurface $\Mbar$ corresponding to the \eqref{G3s} condition; they demonstrate regularity under this condition, and that if a version of the condition corresponding to \eqref{G3w} fails then there are smooth, positive $f$ and $g$ for which the reflector is \emph{not even $C^1$}.

\begin{rem}\label{rem: different regularity at different heights} Another important difference with the regularity theory of optimal transport is that two solutions for the same data $f$ and $g$ may exhibit different regularity. In fact, the existence of such examples can be proven, see the discussion on page 567 of \cite{KarWang10}. This difficulty is what requires us to have to consider the notion of \emph{very nice} solutions, see Definition~\ref{def: very nice} and the remarks that follow it.
	
\end{rem}	

We point out that our method of proof is entirely different from those of \cite{KarWang10}, as their method relies on uniform a priori estimates, while in this paper we rely on pointwise estimates of the solution. In particular, we are able to handle the borderline case corresponding to the \eqref{G3w} condition. However, it should also be noted that the results of \cite{KarWang10} (as those of \cite{GutTour14}, see below) are finer than ours in the sense that they are ``local'' in nature: their result can characterize and separate regions of regularity and nonregularity of solutions, while ours are ``global'': we can only find a solution to be regular on its whole domain, or not.

\bigskip
\subsection{Other geometric optics problems}
There are a number of other geometric optics problems that also result in generated Jacobian equations of the form \eqref{eqn: generated Jacobian equation}, which do not fall within the optimal transport problem. Some of this we mention briefly (even though they each deserve as lengthy a discussion as the previous). One can, for example, consider problems of \emph{refraction} instead of reflection with a point light source, as considered in works by Guti\'errez and Huang \cite{GutHuang14}; and Oliker, Rubinstein, and Wolansky \cite{OliRubWol15}. In another direction, one can change the light source to be a parallel beam instead of a point source (see \cite{Kar14}), or consider multiple optical instruments instead of just one (see work of Glimm and Oliker\cite{GlimmOliker04}, and Oliker \cite{Oli11}). Another interesting family of problems are models with nonperfect energy transmission, as studied by Guti\'errez and Mawi \cite{GutMawi13} and Guti\'errez and Sabra \cite{GutSabra14}. 

There are regularity results available for several of these problems, under assumption \eqref{G3s}. We highlight recent work of Guti\'errez and Tournier \cite{GutTour14} dealing with the (near field) parallel beam reflection and refraction problems. Their results include  $C^{1,\alpha}$ estimates without any smoothness assumptions on the source and target measures. Moreover, unlike our results, the results in \cite{GutTour14} only require local assumptions regarding the ``niceness'' of the solutions.

To give a concrete example, let us write down the generating function for the parallel beam, near-field reflector problem. Let $\Phi$ be a smooth function on some compact region of $\real^2$ (whose graph represents the target surface to be illuminated) and for $(\x,\xbar,\z)\in \mathbb{R}^2\times\mathbb{R}^2\times \mathbb{R}_+$ let
  \begin{equation*}
    \G(\x,\xbar,\z) := \frac{1}{2}\paren{\frac{1}{\z}-\z\norm{\x-\xbar}^2} +\Phi(\xbar).
  \end{equation*}
Then a solution of \eqref{eqn: generated Jacobian equation} with this $\G$ will solve the reflector problem, with an appropriate choice of right hand side depending on the input and output light patterns. There is a detailed verification of conditions \eqref{G5}, \eqref{Twist}, \eqref{DualTwist}, \eqref{Nondeg}, and \eqref{G3w} contained in  \cite[Section 4.2]{JiangTru14} for this choice of $\G$, with $(\Gfivelower, \Gfiveupper)=(0, \infty)$.

\subsection{Optimal transport.}\label{section: OT} Fix any two domains $\outerdom\subset\M$ and $\outertarget\subset\Mbar$ in Riemannian manifolds, suppose we have a measurable \emph{cost function} $c: \outerdom^{\cl}\times\outertarget^{\cl}\to\real$, and probability measures $\mu$ and $\nu$ with supports in $\outerdom$ and $\outertarget$ respectively. The \emph{optimal transport (Monge-Kantorovich)} problem is to find a measurable mapping $T:\spt{\mu}\to\spt{\nu}$ defined $\mu$-a.e. with $T_{\#}\mu=\nu$, minimizing
\begin{align*}
 \int_\outerdom c(\x, S(\x))\mu(dx)
\end{align*}
over all measurable $S:\spt{\mu}\to\spt{\nu}$ with $S_{\#}\mu=\nu$. 

The connection of the optimal transport problem with generated Jacobian equations is through defining
\begin{align*}
 \G(\x, \xbar, \z):=-c(\x, \xbar)-\z.
\end{align*}
With this definition, various structural conditions reduce to well-known conditions, for example in the notation of \cite[Section 2]{GK14}: \eqref{Twist} and \eqref{DualTwist} to (Twist), \eqref{Nondeg} to (Nondeg), \eqref{DomConv} and \eqref{DualDomConv} to (DomConv) there, \eqref{QQConv} and \eqref{DualQQConv} to (QQConv), and \eqref{G3w} and \eqref{dual-G3w} to (A3w) (also known as the Ma-Trudinger-Wang or (MTW) condition). If (Twist), (Nondeg), and (A3w) hold on $\outerdom^{\cl}\times\outertarget^{\cl}$, note that $\gendom=\gendomdual=\outerdom^{\cl}\times\outertarget^{\cl}\times\real$ hence \eqref{G5} is satisfied with $(\Gfivelower, \Gfiveupper)=\real$. Also with these conditions, if $\mu$, $\nu\ll\dVol[\M]$ it is known that a solution of the optimal transport problem can be obtained from a $\G$-convex potential function $u$ satisfying \eqref{eqn: generated Jacobian equation}, by the expression $T(x):=\Xbar{x}{\u_0}{Du(x)}$, for any choice of $\u_0\in\real$ (see \cite{Bre91, GM96, McC01, MTW05}). There is also a regularity theory based on conditions (A3w) and (QQConv), see Section~\ref{section: G3w implies QQConv} for more comments and references.

We also point out a connection of optimal transport to the near-field reflector example in Section~\ref{section: literature reflector problems}. If the target surface $\Mbar$ is very far from the source $O$, then any point being illuminated is approximately determined by the direction of the beam after reflection. Relatedly, if the focus $\xbar$ is far away, then the corresponding ellipsoids are close to being a paraboloid of revolution. Thus taking a limit as the target object goes out to infinity, one obtains the \emph{far-field reflector problem}, which can be viewed as a problem where both domains are the sphere, and reflectors are constructed as envelopes of paraboloids of revolution. Mathematical study of the far-field reflector problem itself stretches back several decades (\cite{HasanisKoutroufiotis85, CafOli08, CafGutHuang08}). The realization that this problem was equivalent to an optimal transport problem for the cost
\begin{align*}
 c(x, \xbar):=-\log{(1-\euclidinner{x}{\xbar})}
\end{align*}
on $\S^2\times \S^2$ (see Glimm and Oliker (\cite{GlimmOliker03}), X.-J. Wang (\cite{Wan96, Wan04}), Guan and Wang (\cite{GuanWang98})) was very fruitful and served as motivation much work in both the mathematics of reflectors and optimal transport. 

\subsection{Generalized Minkowski problem.} A different kind of generated Jacobian equation is given by the classical Minkowski problem. Recall that given a convex body $B \subset \mathbb{R}^n$, $O \in B$, its supporting function is a function $h:\mathbb{S}^{n-1}\to\mathbb{R}$ defined by
\begin{align*}
  h(x) = \sup \limits_{q\in B} (q,x),\;\;x\in\mathbb{S}^{n-1}.
\end{align*}	
It is well known that if $K(x)$ denotes the Gauss curvature of the boundary of $B$ at the point with outer normal $x$, then (see \cite{LutOli1995})
\begin{align*}
  \det\left (\nabla^2_{ij} h+ h g_{ij} \right )	 = \frac{\det( g_{ij})}{K(x)},	
\end{align*}	
where $g_{ij}$ denotes the standard metric of $\mathbb{S}^{n-1}$ and $\nabla^2_{ij}$ denotes the respective covariant derivative. The classical Minkowski problem consists in recovering $B$ from $K(x)$: given a function $K(x)$ on the sphere satisfying certain compatibility conditions, does there exists a smooth, strongly convex body $B$ whose Gauss curvature at the point with normal $x$ is equal to $K(x)$? The formula above shows that in terms of the support function of $B$, this problem falls within the scope of equation \eqref{eqn: generated Jacobian equation}. 

Motivated by questions stemming from the Brunn-Minkowski theory of mixed volumes, Lutwak and Oliker \cite{LutOli1995} considered the more general $p$-Minkowski problem ($p\geq 1$) which asks to find, for a given function $K:\mathbb{S}^{n-1}\to\mathbb{R}$, a convex set whose support function $h$ solves
\begin{align*}
  \det\left (\nabla^2_{ij} h+ h g_{ij} \right )	 = h^{p-1}\frac{\det( g_{ij})}{K(x)}.	
\end{align*}	
For $p\geq 1$, $p \neq n$ and $K(x)$ a positive, even function. When $p=1$ this gives back the original Minkowski problem.

Let us make a few comments about the validity of the various structural assumptions for this example when $p=1$ (see Section \ref{section: elements of generating functions} for definitions). First, the generating function is given by
\begin{align*}
  \G(\x,\xbar,\z) = \z (\x,\xbar),
\end{align*}
where
\begin{align*}
  \gendom = \{ (\x,\xbar,\z) \in \mathbb{S}^{n-1}\times \mathbb{S}^{n-1}\times \mathbb{R} \mid (\x,\xbar)>0,\;\;\z>0\}.	 
\end{align*}	
Then, a straightforward computation shows that
\begin{align*}
  D \G(\x,\xbar,\z) = z (\xbar-(\xbar,\x) \x),\;\;\; \Dbar \G(\x,\xbar,\z) = z (\x-(\x,\xbar) \xbar),\;\;\;\G_{\z}(\x,\xbar,\z) = (\x ,\xbar).
\end{align*}
From here it is not difficult to check the injectivity of $(D \G(\x,\xbar,\z),\G(\x,\xbar,\z)) $ as a function of $(\xbar,\z)$ (for any fixed $\x$) as well as the injectivity of
\begin{align*}
  -\frac{\Dbar \G(\x,\xbar,\z)}{\G_{\z}(\x,\xbar,\z)} = -\frac{\z}{\x\cdot \xbar}(\x-(\x\cdot\xbar) \xbar),
\end{align*}
as a function of $x$ (for any fixed $(\xbar,\z)$), therefore $\G$ verifies \eqref{Twist} and \eqref{DualTwist}. It is not hard to see that the above maps are local diffeomorphisms, and thus the condition \eqref{Nondeg} also holds (see also Remark \ref{rem: local coordinates}). The validity of condition \eqref{QQConv} remains to be determined for this particular $\G$.

\subsection{Stable matching problems with non-quasilinear utility functions} Finally, it is worthwhile to point out a recent preprint of Noldeke and Samuelson where a Generated Jacobian Equation arises in economics. In \cite{NoldekeSamuelson2013}, the authors consider stable matching problems and principal-agent problems where agents may have utility functions that are not \emph{quasilinear}. More concretely, in this setup, $x\in X$ represents all possible \emph{buyer types} while $y\in Y$ represents all possible \emph{seller types}, and $v\in \real$ is a \emph{monetary transfer} (i.e. the price of a product). One is given \emph{utility functions} $\phi(x, y, v)$ and $\psi(x,y,u)$, $\phi(x,y,v)$ being the intrinsic value that $x$ receives when purchasing from seller $y$ at a price of $v$, while $\psi(x,y,u)$ represents the utility $y$ obtains when making a transaction with $x$, by providing $x$ with a utility of $u$. Naturally, these functions satisfy the inverse relation $\phi(x, y, \psi(x, y, u))=u$. 

The \emph{stable matching problem} is then as follows: given two probability measures $\mu$ and $\nu$ on $X$ and $Y$, find a pair of \emph{utility profiles} $(u, v)$ which are measurable, real valued functions on $X$ and $Y$, and a bijective, measurable \emph{matching} $\bm{y}: X\to Y$ such that
\begin{align*}
 u(x)&\equiv\phi(x, \bm{y}(x), v(\bm{y}(x))),\\
 v(y)&\equiv\psi(\bm{y}^{-1}(y), y, u(\bm{y}^{-1}(y))),\\
 \bm{y}_\#\mu&=\nu,
\end{align*}
where it is asked that $(u, v, \bm{y})$ be \emph{stable}, meaning that
\begin{align*}
 u(x) \geq \phi(x, y, v(y)),\;\;\; v(y)\geq \psi(x, y, u(x)),\forall\;x\in X,\ y\in Y.
\end{align*}
In other words, each buyer and seller gets the most utility out of the particular matching $\bm{y}$, and so they have no incentive to pick different parties to deal with. Thus for a stable matching, the profile $u$ is a $\phi$-convex function, satisfying some weak version of equation \eqref{eqn: generated Jacobian equation} (with right hand side depending on the measures $\mu$ and $\nu$).

A utility function is said to be quasilinear if it has the form $\phi(x,y,v) = b(x,y)-v$. In this case, the stable matching problem reduces to an optimal transport problem, and is also related to a hedonic pricing problem. This direction has been explored by Ekeland (\cite{Ek05, Ek10}), and later by Chiappori, McCann, and Nesheim (\cite{CMN10}). Figalli, Kim, and McCann have also shown that the problem becomes a convex screening problem, under a strengthening of the (A3w) condition, often known as ``non-negative cross curvature'' in the optimal transport literature (see \cite{FKM11a}). Moreover, for this quasilinear case, the structural assumptions for $\G$ discussed in Section \ref{section: elements of generating functions} reduce to the standard ones for optimal transport (see Example \ref{section: OT}).

It is worth noting that in the terminology introduced at the beginning of Section \ref{section: elements of generating functions}, the function $\phi$ corresponds to $\G$, while $\psi$ corresponds to the dual generating function $H$. We do not know yet of any specific, multi-dimensional non-quasilinear utility functions for which our assumptions hold. It would be worthwhile to find concrete examples of such utility functions which are different from the generating functions in the previous examples, and to provide economic interpretations for these structural assumptions. 

\section{Elements of Generating Functions}\label{section: elements of generating functions}

\subsection{Basic definitions}\label{subsection: basic definitions} 
Suppose $(\M,\g{})$ and $(\Mbar,\gbar{})$ are $n$-dimensional Riemannian manifolds. We fix a real valued \emph{generating function} $\G(\cdot, \cdot, \cdot)$ defined on $\M\times\Mbar\times I$ for some open interval $I$; after a change of variables that will not affect any of the other conditions we pose on $\G$, it can be assumed $I=\real$ (which we will do for the remainder of the paper). 
We will use the notation $\Dx$ for derivatives in the $\x$ variable, and $\Dbar$ for derivatives in the $\xbar$ variable, while $\Gz$, $\Gzz$, etc. denote derivatives in the scalar $\z$ variable. We also assume that $\G$ is $C^2$ in the sense that any second order derivative in the variables $\x$, $\xbar$, and $\z$ which is mixed (i.e., $\Dbar \Dx \G$, or $\Dbar \Gz$, etc) is continuous and that $\Gz(\x, \xbar, \z)<0$ for all $(\x, \xbar, \z)$.

The inverse function theorem yields the existence of a unique function $\H(\x,\xbar,\u)$ such that
  \begin{align*}
    \G(\x, \xbar, \H(\x, \xbar, \u))=u.
  \end{align*}
 $\H(\x, \xbar, \cdot)$ is defined on some open interval (which may depend on $(\x, \xbar)$) with $\Hu<0$, and $\H$ is $C^2$ in the above sense. Whenever we write an expression of the form $\H(\x, \xbar, \u)$, it is with the understanding that $\u$ is in the range of $\G(\x, \xbar, \cdot)$.

  As in \cite{Tru14}, we require $\G$ to satisfy certain structural conditions. These assumptions will hold on a subset of the domain of $\G$, denoted $\mathfrak{g}$ (and fixed from now on), which has the form
  \begin{equation*}
    \gendom := \curly{ (\x,\xbar,\z) \in \M\times \Mbar\times \mathbb{R} \mid \z \in \intervalG(\x,\xbar) },
  \end{equation*}
  where for each $(\x,\xbar)\in \M\times \Mbar$ the set $\intervalG(\x,\xbar)$ is an open interval (possibly empty). Similarly, we will deal with the set
  \begin{equation*}
    \gendomdual := \{ (\x,\xbar,\u) \in \M\times \Mbar\times \mathbb{R} \mid \u \in \intervalH(\x,\xbar) \},\;\;\intervalH(\x,\xbar):=\G(\x,\xbar,\intervalG(\x,\xbar)).
  \end{equation*}
The following condition is a relaxation of the (G5) condition presented in \cite{Tru14}, and is also due to Trudinger \cite{Trudinger}.
\begin{DEF}\label{definition: G5}
  A generating function $\G$ and bounded, open domains $\outerdom\subset\M$, $\outertarget\subset\Mbar$ are said to satisfy \textbf{\underline{uniform admissibility}} if there are constants $-\infty\leq\Gfivelower<\Gfiveupper\leq \infty$ and $0<\lipbound<\infty$ for which, whenever $(\x, \xbar, \u)\in\outerdom^{\cl}\times\outertarget^{\cl}\times(\Gfivelower,\Gfiveupper)$, then 
  \begin{align}
(\x, \xbar, \H(\x, \xbar, \u))&\in\gendom,\tag{Unif}\label{G5}\\
      \gnorm[\x]{\Dx\G(\x, \xbar, \H(\x, \xbar, \u))}&\leq \lipbound.\label{Lip}\tag{$\textnormal{Lip}_{K_0}$}
  \end{align}
\end{DEF}

\begin{rem}\label{rem:Unif_consequence} One elementary but useful consequence of \eqref{G5} is that if $\G(\x, \xbar, \z)\in (\Gfivelower, \Gfiveupper)$, then we must have $(\x, \xbar, \z)\in \gendom$. Indeed this is immediate as if $\u:=\G(\x, \xbar, \z)$, by definition $\H(\x, \xbar, \u)=\z$. We will use this fact frequently.
\end{rem}

\begin{DEF}\label{definition: Bi-twist} 
  The function $\G$ is said to satisfy the \textbf{\underline{twist conditions}} if for any $(\x_0,\xbar_0,z_0)\in \M \times \Mbar \times \mathbb{R}$ we have the following 
  \begin{enumerate}
    \item The mapping
    \begin{align}
      (\xbar, \z) & \mapsto (\Dx \G(\x_0, \xbar, \z), \G(\x_0, \xbar, \z))\in \cotanspM{\x_0}\times \real,\tag{$\G$-Twist}\label{Twist}
    \end{align}
    is injective on the set 
    $\curly{(\xbar, \z)\in\Mbar\times\real \mid(\x_0, \xbar, \z)\in\gendom}$.
    \item The mapping
    \begin{align}
      x & \mapsto -\frac{\Dbar \G(\x, \xbar_0, \z_0)}{\Gz(\x, \xbar_0, \z_0)}\in\cotanspMbar{\xbar_0}\tag{$\Gstar$-Twist}\label{DualTwist}
    \end{align}
    is injective on
    $\curly{\x\in\M\mid (\x, \xbar_0, \z_0)\in\gendom}$.
  \end{enumerate}
\end{DEF}

 Although conditions \eqref{Twist} and \eqref{DualTwist} may seem quite different, they are actually symmetric in nature. See Remark~\ref{rem: twists are symmetric} for more details.
 
\begin{rem}
  For the sake of brevity, the arguments in expressions such as $(\Dx \G(\x_0, \xbar, \z),\G(\x_0, \xbar, \z))$ and $-\frac{\Dbar \G(\x, \xbar,\z)}{\Gz(\x, \xbar,\z)}$ will be written simply as $(\Dx \G, \G)(\x_0, \xbar, \z)$ and $-\frac{\Dbar \G}{\Gz}(\x, \xbar,\z)$.
\end{rem}

\begin{DEF}\label{def: nondeg}
  The function $\G$ is said to satisfy the \textbf{\underline{nondegeneracy condition}} if given any triplet $(\x, \xbar, \z) \in \mathfrak{g}$, the linear mapping $\nondegmatrix{\x}{\xbar}{\z}: \tanspMbar{\xbar}\to\cotanspM{\x}$ defined by 
  \begin{align}\tag{$\G$-Nondeg}\label{Nondeg}
    \nondegmatrix{\x}{\xbar}{\z}\Vbar := \Dbar \Dx\G(\x, \xbar, \z)\Vbar-\inner{\frac{\Dbar \G}{\Gz}(\x, \xbar, \z)}{\Vbar}\Dx \Gz(\x, \xbar, \z),\;\;\Vbar \in \tanspMbar{\xbar}
  \end{align}
  is invertible. The adjoint operator of $\nondegmatrix{\x}{\xbar}{\z}$ (which is also invertible under the assumption \eqref{Nondeg}), will be denoted by $\dualnondegmatrix{\x}{\xbar}{\z}: \tanspM{\x} \to \cotanspMbar{\xbar}$, so
  \begin{align*}
    \inner{\nondegmatrix{\x}{\xbar}{\z}\Vbar}{\V}=\inner{\Vbar}{\dualnondegmatrix{\x}{\xbar}{\z}\V},\qquad \forall\; \Vbar\in\tanspMbar{\xbar},\ \V\in\tanspM{\x}.
  \end{align*}	
\end{DEF}

\begin{DEF}\label{DEF: G-coordinates}
  We will use the notation 
  \begin{align*}
    \pbarmap{x}{\u}{\xbar}:&=\Dx \G (\x, \xbar, \H(\x, \xbar, \u)),\\
    \pmap{\xbar}{\z}{x}:&=-\frac{\Dbar \G}{\Gz}(\x, \xbar, \z).
  \end{align*}

  Also if $\arbitrary \subset \outerdom$ and $(\xbar,\z)$ are such that $(\x,\xbar,\z)\in\gendom$ for all $\x\in\arbitrary$, we will write
  \begin{align*}
    \coord{\arbitrary}{\xbar,\z} &:= \pmap{\xbar}{\z}{\arbitrary} \subset \cotanspMbar{\xbar}.
  \end{align*}
  Likewise, if $\bar\arbitrary$ and $(\x,\u)$ are such that $(\x,\xbar,\u) \in \gendomdual$ for all $\xbar\in\bar\arbitrary$, we will write
  \begin{align*}
    \coord{\bar \arbitrary}{x,\u} & := \pbarmap{x}{\u}{\bar \arbitrary} \subset \cotanspM{\x}.
  \end{align*}  
\end{DEF}

\begin{DEF}\label{DEF: exponential mappings}
  Due to \eqref{Twist} and \eqref{DualTwist} there are differentiable maps
  \begin{equation*}
    \X{\xbar}{\z}{\cdot},\;\;\Xbar{\x}{\u}{\cdot},\;\;\Z{\x}{\cdot}{\cdot}, 
  \end{equation*}
  respectively defined on subsets of $\cotanspMbar{\xbar},\cotanspM{\x}$, and $\cotanspM{\x}\times \mathbb{R}$, by the system of equations
  \begin{align*}
    (\Dx\G, \G)(\x,\Xbar{\x}{\u}{\pbar}, \Z{\x}{\pbar}{\u})&=(\pbar, \u), & \;\;\forall\;\; (\pbar, \u) \in (D\G,\G)(\left \{(\x,\xbar,\z) \mid (\x,\xbar,\z) \in \gendom \right \} )
  \end{align*}
  and 
  \begin{align*}
      -\frac{\Dbar \G}{\Gz}(\X{\xbar}{\z}{\p},\xbar,\z) =\p, \;\;\forall\; \p\in -\frac{\Dbar \G}{\Gz}( \curly{(\x,\xbar,\z) \mid (\x,\xbar,\z) \in \gendom  } ).
  \end{align*}
   Note that by \eqref{Twist},
  \begin{align*}
    \Z{\x}{\pbar}{\u}\equiv\H(\x, \Xbar{\x}{\u}{\pbar}, \u).
  \end{align*}  

\end{DEF}

\subsection{\G-convex geometry.} The coordinate systems given by $\pmap{\xbar}{\z}{\cdot}$ and $\pbarmap{\x}{\u}{\cdot}$ are of great relevance to the study of the generating function $\G$ (see also Lemma~\ref{lem: comparability}). Of special interest are those domains in $\outerdom$ (resp. $\outertarget$) that correspond to convex sets in at least one of these coordinate systems. The same can be said for curves in $\outerdom$ (resp. $\outertarget$) that correspond to a straight line segment in one of these coordinate systems. These ideas are recalled in detail below.

\begin{DEF}\label{DEF: G-segment}
  A differentiable curve $\x(s)$ in $\M$ ($s\in[0,1]$) is said to be a \emph{$\G$-segment} with respect to $(\xbar,\z)\in \Mbar \times \mathbb{R}$, if for all $s\in[0,1]$ we have that $(\x(s),\xbar, \z)\in\gendom$ and
  \begin{align*}
    \pmap{\xbar}{\z}{\x(s)} = (1-s)\pmap{\xbar}{\z}{\x(0)}+s\pmap{\xbar}{\z}{\x(1)}.
  \end{align*}
  Likewise, a curve $\xbar(t)$ in $\Mbar$ ($t\in[0,1]$) is said to be a \emph{$\G$-segment} with respect to $(\x, \u) \in \M\times \mathbb{R}$, if for all $t\in[0,1]$ we have that $(\x,\xbar(t), \u)\in\gendomdual$  and
  \begin{align*}
    \pbarmap{\x}{\u}{\xbar(t)} = (1-t)\pbarmap{\x}{\u}{\xbar(0)}+t\pbarmap{\x}{\u}{\xbar(1)}.
  \end{align*}
\end{DEF}
\begin{rem}\label{rem: G-segment notation}
 If $\x(s)$ is a $\G$-segment with respect to $(\xbar, \z)$ with $\x(0)=\x_0$, $\x(1)=\x_1$, we will use the notation $\Gseg{\x_0}{\x_1}{\xbar, \z}$ for the image $\x([0, 1])$. Moreover, given $\x_0$, $\x_1\in\M$, and $(\xbar, \z)\in\Mbar \times \mathbb{R}$, by an abuse of notation we will write \emph{$\x(s):=\Gseg{\x_0}{\x_1}{\xbar, \z}$} to signify that $\x(s)$ is the (unique) parametrization of a $\G$-segment given in the above definition with $\x(0)=\x_0$, $\x(1)=\x_1$. Additionally, when we say $\x(s)$ is \emph{well-defined} it specifically denotes that for all $s\in[0, 1]$, $(1-s)\pmap{\xbar}{\z}{\x_0}+s\pmap{\xbar}{\z}{\x_1}$ lies in the image $\curly{-\frac{\Dbar\G}{\Gz}(\x, \xbar, \z)\mid \x\in\M, (\x, \xbar, \z)\in\gendom}$. A similar remark holds for $\G$-segments $\xbar(t)$ in $\Mbar$.
\end{rem}
\begin{rem}\label{rem: local coordinates}
Fixing local coordinates in $\M$ and $\Mbar$, the matrix representation of $\nondegmatrix{\x}{\xbar}{\z}$ is 
\begin{align*}
  \nondegmatrixentries[ij]=\G_{x^i\xbar^j} -\frac{\G_{x^i \z}\G_{\xbar^j}}{\Gz}.
\end{align*}
A routine calculation then shows that the derivatives of the maps $\x\mapsto \pmap{\xbar}{\z}{\x}$ and $\xbar\mapsto \pbarmap{\x}{\u}{\xbar}$ are given by $-\frac{\dualnondegmatrix{\x}{\xbar}{\z}}{\Gz}$ and $\nondegmatrix{\x}{\xbar}{\H(\x, \xbar, \u)}$ respectively, hence these mappings are $C^1$-diffeomorphisms in a neighborhood of wherever \eqref{Nondeg} holds (i.e., near $\x$ such that $(\x, \xbar, \z)\in\gendom$ for $\pmap{\xbar}{\z}{\cdot}$ and near $\xbar$ such that $(\x, \xbar, \u)\in\gendomdual$ for $ \pbarmap{\x}{\u}{\cdot}$). In particular, this implies that $\G$-segments are differentiable as long as they are well-defined.
\end{rem}
We also make some convexity assumptions on the domains $\outerdom\subset \M$ and $\outertarget\subset \Mbar$.

\begin{DEF}\label{def: DomConv}  
  We will assume that for any $\x\in\outerdom^{\cl}$, 
  \begin{align}
    \u\in(\Gfivelower, \Gfiveupper)&\implies\coord{\outertarget}{\x, \u}\text{ is convex.}\label{DualDomConv}\tag{DomConv$^*$}
  \end{align}
  Also suppose $\x_0$, $\x_1\in\outerdom^{\cl}$, $\xbar\in\outertarget^{\cl}$, and $\z\in\real$ with $\G(\x_0, \xbar, \z)\in (\Gfivelower, \Gfiveupper)$. Then we assume that $\outerdom$ is path-connected and
  \begin{align}
    &(\x_0, \xbar, \z),\ (\x_1, \xbar, \z)\in\gendom\notag\\
    &\implies \x(s):=\Gseg{\x_0}{\x_1}{\xbar, \z}\text{ is well-defined and }\Gseg{\x_0}{\x_1}{\xbar, \z}\subset\outerdom^{\cl}.
    \label{DomConv}\tag{DomConv}
  \end{align}
\end{DEF}

The next proposition computes the velocity of a $\G$-segment in terms of the linear maps $\nondegmatrix{\x}{\xbar}{\z}$ and $\dualnondegmatrix{\x}{\xbar}{\z}$

\begin{prop}\label{prop: motivation for nondegeneracy}
  Let $\x(s)$ be a well-defined $\G$-segment with respect to some $(\xbar_0,\z_0)$, $\xbar(t)$ a well-defined $\G$-segment with respect to some $(\x_0,\u_0)$, and let 
  \begin{equation*}
    \z(t) := \H(\x_0,\xbar(t),\u_0).  
  \end{equation*}
  Then, using the notation $\p_s:=\pmap{\xbar_0}{\z_0}{\x(s)}$ and $\pbar_t:=\pbarmap{\x_0}{\u_0}{\xbar(t)}$, we have the expressions
  \smallskip
  \begin{align}
    \xdot(s) & = -\Gz(\x(s), \xbar_0, \z_0)\dualnondegmatrixinv{\x(s)}{\xbar_0}{\z_0}(\p_1-\p_0), \label{eqn: xdot representation}\\
    \xbardot(t) & = \nondegmatrixinv{\x_0}{\xbar(t)}{\z(t)}(\pbar_1-\pbar_0), \label{eqn: xbardot representation}\\
    \zdot(t) & = \inner{-\frac{\Dbar \G}{\Gz}(\x_0, \xbar(t), \z(t))}{\xbardot(t)}. \label{eqn: zdot representation}
  \end{align}
\end{prop}

\begin{proof}
  Differentiating the identity $-\frac{\Dbar \G}{\Gz}(x(s), \xbar_0, \z_0) = (1-s)\p_0+s\p_1$ in $s$ yields
  \begin{equation*}
    \left [\frac{-\Dx\Dbar \G}{\Gz}\right ] \xdot(s) +\frac{\inner{\Dx\Gz}{\xdot(s)}\Dbar\G}{\Gz^2} = \p_1-\p_0,
  \end{equation*}
  where all expressions are evaluated at $(\x(s), \xbar_0, \z_0)$. 
  Therefore, for an arbitrary $\Vbar \in \tanspMbar{\xbar_0}$ and $s\in[0,1]$,
  \begin{align*}
    \inner{\p_1-\p_0}{\Vbar}&=-\frac{1}{\Gz(x(s), \xbar_0, \z_0)}\paren{\inner{\brackets{\Dbar\Dx\G}\Vbar}{\xdot(s)}-\inner{\Dx\Gz}{\xdot(s)}\inner{\frac{\Dbar\G}{\Gz}}{\Vbar}}\\
    &= -\frac{1}{\Gz(x(s), \xbar_0, \z_0)}\inner{\nondegmatrix{\x(s)}{\xbar_0}{\z_0}\Vbar}{\xdot(s)}
  \end{align*}
  and \eqref{eqn: xdot representation} follows. Similarly, differentiating the identities
  \begin{align*}
    \Dx\G(\x_0, \xbar(t), \z(t))&=(1-t)\pbar_0+t\pbar_1,\\
    \G(\x_0, \xbar(t), \z(t))&=\u_0,
  \end{align*}
in $t$ we obtain 
\begin{align*}
\brackets{ \Dbar\Dx\G} \xbardot(t)+\Dx\Gz \zdot(t)&=\pbar_1-\pbar_0,\\
 \inner{\Dbar\G}{\xbardot(t)}+\Gz\zdot(t)&=0,
\end{align*}
where all expressions are evaluated at $(\x_0, \xbar(t), \z(t))$. Rearranging the second line above and using that $\Gz\neq 0$ yields \eqref{eqn: zdot representation}. We can then substitute \eqref{eqn: zdot representation} into the first line above to obtain
\begin{align*}
 \pbar_1-\pbar_0&= \brackets{\Dbar\Dx\G} \xbardot(t)-\Dx\Gz \inner{\frac{\Dbar \G}{\Gz}}{\xbardot(t)}\\
 &=\nondegmatrix{\x_0}{\xbar(t)}{\z(t)}\xbardot(t).
\end{align*}
Since $\nondegmatrix{\x_0}{\xbar(t)}{\z(t)}$ is invertible by \eqref{Nondeg}, the formula \eqref{eqn: xbardot representation} follows.
\end{proof}

The last two conditions on the generating function $\G$ are as follows.\\

\begin{DEF}\label{definition: QQConv}
We say \emph{$\G$ satisfies \eqref{QQConv}}, if for any compact subinterval $[\QQConvlower, \QQConvupper]\subset (\Gfivelower, \Gfiveupper)$, there is a constant $M\geq 1$ with the following property: take any $\x_0$, $\x_1\in\outerdom^{\cl}$, $\xbar_1$, $\xbar_0\in\outertarget^{\cl}$, $\z_0\in\real$ such that $\G(\x(s), \xbar_0, \z_0)\in[\QQConvlower, \QQConvupper]$ for all $s\in [0, 1]$ where $\x(s):=\Gseg{\x_0}{\x_1}{\xbar_0, \z_0}$. Then if $\z_1:=\H(\x_0, \xbar_1, \G(\x_0, \xbar_0, \z_0))$, it holds

\begin{align*}
  & \G(\x(s), \xbar_1, \z_1))-\G(\x(s), \xbar_0, \z_0) \label{QQConv}\tag{$\G$-QQConv}\\
  & \qquad\leq \frac{Ms}{1-s'}(\G(\x_1, \xbar_1, \H(\x(s'), \xbar_1, \G(\x(s'), \xbar_0, \z_0)))-\G(\x_1, \xbar_0, \z_0)),
\end{align*}
for any $s\in [0, 1]$ and $s'\in [0, 1)$.
  Likewise, \emph{$\G$ satisfies \eqref{DualQQConv}} if for any compact subinterval $[\QQConvlower, \QQConvupper]\subset (\Gfivelower, \Gfiveupper)$ there exists a constant $M\geq 1$ such that: whenever $\x_0\in\outerdom^{\cl}$, $\xbar_0$, $\xbar_1\in\outertarget^{\cl}$, $\u_0\in[\QQConvlower, \QQConvupper]$, and $\x_1\in\M$ with $(\x_1, \xbar(t), \H(\x_0, \xbar(t), \u_0))\in\gendom$ for all $t\in[0, 1]$, (where $\xbar(t):=\Gseg{\xbar_0}{\xbar_1}{\x_0, \u_0}$), it holds for any $t'\in [0, 1)$ that
  \begin{align}
    &\G(\x_1,\xbar(t), \H(\x_0,\xbar(t),\u_0))-\G(\x_1,\xbar_0,\H(\x_0,\xbar_0,\u_0)) 	\label{DualQQConv}\tag{$\Gstar$-QQConv}\\
    &\qquad \leq \frac{Mt}{1-t'}\;\brackets{\G(\x_1,\xbar_1,\H(\x_0,\xbar_1,\u_0))-\G(\x_1,\xbar(t'),\H(\x_0,\xbar(t'),\u_0))}_{+}.\notag
  \end{align}
If $\G$ satisfies both \eqref{QQConv} and \eqref{DualQQConv}, we say that $\G$ is \textbf{\underline{quantitatively quasiconvex}}.
\end{DEF}
We note that due to assumptions \eqref{G5}, \eqref{DomConv}, and \eqref{DualDomConv}, in the above definitions both $\G$-segments $\x(s)$ and $\xbar(t)$ are well-defined and remain in $\outerdom^{\cl}$, $\outertarget^{\cl}$ respectively for all $s$, $t\in[0, 1]$. 

\subsection{$\G$-convex functions.}

\begin{DEF}\label{def: G-functions}
  A real valued function $u$ defined on $\outerdom$ is said to be \emph{$\G$-convex} if for any $\x_0\in\outerdom$ there is a \emph{focus $(\xbar_0, \z_0) \in \outertarget\times\real$} such that $(\x_0, \xbar_0, \z_0)\in\gendom$ and 
  \begin{align*}
    u(\x_0)&=\G(\x_0, \xbar_0, \z_0),\\
    u(\x)&\geq \G(\x, \xbar_0, \z_0),\qquad\forall\; \x\in\outerdom.
  \end{align*}
Any function of the form $\G(\cdot, \xbar_0, \z_0)$ will be called a \emph{$\G$-affine function}, and if it satisfies the above conditions we say it is \emph{supporting to $u$ at $\x_0$}.
\end{DEF}
We remark here that by \eqref{Twist} it is clear that if $\G(\cdot, \xbar_0, \z_0)$ is supporting to $u$ at $\x_0$, we must have $\z_0=\H(\x_0, \xbar_0, \z_0)$. Also note that in the definition above, it is \emph{not} assumed that $(\x, \xbar_0, \z_0)\in\gendom$ for all $\x\in\outerdom$, but only for $\x_0$. This distinction will motivate further definitions below.
\begin{DEF}\label{DEF: G-subdifferentials}
Let $u$ be a $\G$-convex function and $x\in\outerdom$. We define the \emph{$\G$-subdifferential of $u$ at $x$} as the set-valued mapping
\begin{align*}
\Gsubdiff{u}{\x}\defin\curly{\xbar\in\outertarget\mid \exists \z\in\real\text{ s.t. } \G(\cdot, \xbar, \z) \text{ is supporting to }u\text{ at }\x}.
\end{align*}
For $\x\in\outerdom^{\bdry}$, we define 
\begin{align*}
\Gsubdiff{u}{\x}:=\curly{\lim_{k\to\infty}\xbar_k\mid \xbar_k\in\Gsubdiff{u}{\x_k},\ \outerdom\ni\x_k\xrightarrow[k\to\infty]{}\x}.
\end{align*}
Also, for any $\arbitrary\subset \outerdom^{\cl}$, we define
\begin{align*}
\Gsubdiff{u}{\arbitrary}:=\union_{x\in \arbitrary}\Gsubdiff{u}{x}.
\end{align*}
\end{DEF}
If $\x\in \outerdom^\bdry$ we will say $\G(\cdot, \xbar, \H(\x, \xbar, u(\x)))$ is \emph{supporting to $u$ at $\x$} only when $\xbar\in \Gsubdiff{u}{\x}$.

With this notion in hand, we are now able to define an appropriate \emph{weak} notion of solutions to the generated Jacobian equation \eqref{eqn: generated Jacobian equation}, which will allow for measure valued data.
  \begin{DEF}\label{def:Aleksandrov_solutions}
Let $\mu$ be a positive Borel measure defined on $\outerdom$. We say a $\G$-convex function $u$ on $\outerdom$ is an \emph{Aleksandrov weak solution of the generated Jacobian equation} if for any Borel measurable $A\subset \outerdom$ we have
\begin{align*}
 \Leb{\Gsubdiff{u}{A}}=\mu(A).
\end{align*}
  \end{DEF}
We recall that $\Leb{\Gsubdiff{u}{\cdot}}$ is a Radon measure (see \cite[Section 4]{Tru14}) known as the \emph{$\G$-Monge-Amp{\`e}re measure} of $u$.

\begin{rem}\label{rem: weak solutions with bounded RHS}
 In this paper, we are concerned with the specific case corresponding to equation \eqref{eqn: generated Jacobian equation} when the function $\psi_G$ on the right hand side is bounded away from zero and infinity. Thus, in the sequel we will say $\G$-convex function $u$ on $\outerdom$ is an \emph{Aleksandrov solution of \eqref{eqn: generated Jacobian equation} (with bounded right hand side)} to mean there exists a constant $\Lambda>0$ such that
	  \begin{align*}
   \Lambda^{-1}\Leb{A\cap \innerdom}\leq \Leb{\Gsubdiff{u}{A}}\leq \Lambda \Leb{A\cap \innerdom},\quad\textnormal{any Borel set } \arbitrary\subset\outerdom.
      \end{align*}
      Here $\innerdom$ is the support of $\Leb{\Gsubdiff{u}{\cdot}}$.
\end{rem}

\begin{DEF}\label{def: very nice}
 We say that a $\G$-convex function $u$ is \emph{nice (in $\outerdom$)} if $\Gfivelower<u <\Gfiveupper$ on $\outerdom^{\cl}$.
 
We also say a $\G$-convex function $u$ is \emph{very nice (in $\outerdom$} if every $\G$-affine function supporting to $u$ in $\outerdom^{\cl}$ is \emph{nice} (thus in particular, $u$ is also nice).
\end{DEF}

\begin{rem}\label{rem: nice functions are nice}
If $u$ is a \emph{nice} $\G$-convex function, \eqref{Lip} combined with a standard argument implies $u$ is locally bounded, and also locally Lipschitz (and in particular continuous) in $\outerdom^{\cl}$. As a result, a \emph{nice} $\G$-convex function is differentiable a.e. on $\outerdom$.

Indeed, fix any point $\x_0\in\outerdom$. Since $u$ is \emph{nice} then $u(\x_0)\in (\Gfivelower+\epsilon, \Gfiveupper-\epsilon)$ for some $\epsilon>0$, small enough that $\ball{\epsilon/\lipbound}{\x_0}\subset\outerdom$ (where $\lipbound$ is the constant in \eqref{Lip}). We first claim that 
\begin{align*}
u(\x_0)-\epsilon<\G(\y, \xbar, \H(\x_0, \xbar, u(\x_0)))<u(\x_0)+\epsilon
\end{align*}
for all $\xbar\in\outertarget^{\cl}$ and $\y\in \ball{\epsilon/\lipbound}{\x_0}$. Indeed, let $\y\in \ball{\epsilon/\lipbound}{\x_0}^{\cl}$ and write $\y_{\g{}}(s)$ for the unit speed minimal geodesic from $\x_0$ to $\y$ (we may first shrink $\epsilon$ to ensure such a minimal geodesic exists for every point within the boundary of the ball). Fix an arbitrary $\xbar\in\outertarget^{\cl}$ and define 
\begin{align*}
s^*:=\sup\curly{s^{**}\in [0, \gdist{\x_0}{\y}]\mid \G(\y_{\g{}}(s), \xbar, \H(\x_0, \xbar, u(\x_0)))\in (u(\x_0)-\epsilon, u(\x_0)+\epsilon),\ \forall\; s\in[0,  s^{**}]}.
\end{align*}
If $s^*=\gdist{\x_0}{\y}$, we are done. Otherwise, we must have $\G(\y_{\g{}}(s^*), \xbar, \H(\x_0, \xbar, u(\x_0)))$ equal to either $u(\x_0)-\epsilon$ or $u(\x_0)+\epsilon$, thus by \eqref{Lip} we can calculate
\begin{align*}
 \epsilon&=\norm{\G(\y_{\g{}}(s^*), \xbar, \H(\x_0, \xbar, u(\x_0)))-u(\x_0)}\\
 &\leq \int_0^{s^*}\inner{\Dx\G(\y_{\g{}}(s), \xbar, \H(\x_0, \xbar, u(\x_0)))}{\dot{\y}_{\g{}}(s)}ds\\
 &\leq \lipbound s^*<\lipbound \gdist{\x_0}{\y}
\end{align*}
which is a contradiction, thus we obtain our first claim.

Now take any $\y\in \ball{\epsilon/\lipbound}{\x_0}$ and let $\ybar\in\Gsubdiff{u}{\y}$, then we have $\G(\x_0, \ybar, \H(\y, \ybar, u(\y)))\leq u(\x_0)$, thus combined with the above bound
\begin{align*}
 u(\y)&=\G(\y, \ybar, \H(\y, \ybar, u(\y)))\\
 &= \G(\y, \ybar, \H(\x_0, \ybar, \G(\x_0, \ybar, \H(\y, \ybar, u(\y)))))\\
 &\leq \G(\y, \ybar, \H(\x_0, \ybar, u(\x_0)))<u(\x_0)+\epsilon.
\end{align*}
On the other hand, if $\xbar_0\in\Gsubdiff{u}{\x_0}$, we see that $u(\y)\geq \G(\y, \xbar_0, \H(\x_0, \xbar_0, u(\x_0)))>u(\x_0)-\epsilon$, thus we find that $u\in (u(\x_0)-\epsilon, u(\x_0)+\epsilon)\subset (\Gfivelower, \Gfiveupper)$ on $\ball{\epsilon/\lipbound}{\x_0}^{\cl}$, i.e. $u$ is locally bounded in $\outerdom$.

By following the same line of proof as above, we can see that for any $\y_1$, $\y\in\ball{\epsilon/2\lipbound}{\x_0}^{\cl}$ and $\ybar_1\in\Gsubdiff{u}{\y_1}$, we have $\G(\y, \ybar_1, \H(\y_1, \ybar_1, u(\y_1)))\in (\Gfivelower, \Gfiveupper)$. Then by choosing $\mathcal{N}$ to be a small enough geodesically convex neighborhood of $\x_0$ contained in $\ball{\epsilon/2\lipbound}{\x_0}^{\cl}$, by \eqref{Lip} we find for any $\y_1$, $\y_2\in\mathcal{N}$,
\begin{align*}
 u(\y_1)-u(\y_2)&\leq \G(\y_1, \ybar_1, \H(\y_1, \ybar_1, u(\y_1)))-\G(\y_2, \ybar_1, \H(\y_1, \ybar_1, u(\y_1)))\\
  &\leq \lipbound\gdist{\y_1}{\y_2}.
\end{align*}
By a symmetric argument, $u$ is locally Lipschitz in $\outerdom$.
\end{rem}

\begin{rem}\label{rem: very nice interval}
 If $u$ is a \emph{very nice} $\G$-convex function, there exists a compact subinterval $\Niceinterval\subset (\Gfivelower, \Gfiveupper)$ of such that $\Nicelower<\mountain<\Niceupper$ on $\outerdom^{\cl}$ for \emph{any} $\G$-affine function $\mountain$, supporting to $u$ in $\outerdom$. Indeed, note that
\begin{align*}
 &\sup\curly{\mountain(\y)\mid\y\in\outerdom^{\cl},\ \mountain\text{ is }G\text{-affine and supporting to }u \text{ in }\outerdom^{\cl}}\\
 &=\sup\curly{\G(\y, \xbar, \H(\x, \xbar, u(\x)))\mid \y, \x\in\outerdom^{\cl}, \xbar\in\Gsubdiff{u}{\x}}
\end{align*}
and as $u$ is \emph{very nice} (by Remark~\ref{rem: nice functions are nice} above, $u$ is continuous on $\outerdom^{\cl}$), the constraint set in the second line is clearly compact. A similar argument holds for the infimum. We will refer to this subinterval as a \emph{very nice interval} associated to $u$.
\end{rem}

\begin{rem}\label{rem: existence of very nice solutions}
One of our ultimate goals is to apply Theorems ~\ref{thm: G-aleksandrov estimate} and \ref{thm: Sharp growth} toward regularity of \emph{weak solutions} of \eqref{eqn: generated Jacobian equation} (see \cite[Section 4]{Tru14} for a definition and discussion). However, when $(\Gfivelower, \Gfiveupper)\neq \real$ in \eqref{G5}, we may only be able to apply our estimates Theorems ~\ref{thm: G-aleksandrov estimate} and \ref{thm: Sharp growth} to a \emph{very nice} $\G$-convex function $u$. This is to be expected as one feature of this case is that weak solutions of \eqref{eqn: generated Jacobian equation} with the same data may have differing regularity (see Sections \ref{section: literature reflector problems}-\ref{section: OT}).

The following adaptation of the condition (G5) in \cite{Tru14} due to Trudinger (also shared with us through personal communication \cite{Trudinger}) gives existence of weak solutions of \eqref{eqn: generated Jacobian equation} that are \emph{very nice}. Indeed, define 
\begin{align*}
 d_\outerdom(x_1, x_2):&=\inf\curly{L_g(\gamma)\mid \gamma\subset \outerdom^{\cl},\ \gamma(0)=x_1, \gamma(1)=x_2,\ \gamma \text{ piecewise }C^1},\\
 \diam_\outerdom{(\outerdom)}:&=\sup_{x_1,\ x_2\in\outerdom}d_\outerdom(x_1, x_2),
\end{align*}
(here $L_g(\gamma)$ above is the length of a piecewise $C^1$ curve in $(\M, \g{})$). 
Then assume that the constant $\lipbound$ in \eqref{G5} satisfies $\lipbound<\tfrac{\Gfiveupper-\Gfivelower}{2\diam_{\outerdom}{(\outerdom)}}$. Then writing $K_1:=K_0 \sup_{\x\in\outerdom}d_{\outerdom}(\x_0,\x)$, for any measurable, bounded data, $\x_0\in\outerdom$, and $\u_0\in(\Gfivelower+K_1, \Gfiveupper-K_1)$, there exists a \emph{nice} weak solution $u$ of \eqref{eqn: generated Jacobian equation} with $u(\x_0)=\u_0$ (see \cite[Theorem 4.2]{Tru14}). If $\u_0\in(\Gfivelower+3K_1, \Gfiveupper-3K_1)$, $u$ will be \emph{very nice}; the argument is similar to the one in Remark~\ref{rem: nice functions are nice}.
\end{rem}
The next notion is that of the \emph{$\G$-dual of a set} $\arbitrary\subset\outerdom^{\cl}$.
\begin{DEF}\label{DEF: G-dual}
Let $\arbitrary\subset\outerdom$, $\x\in\arbitrary^{\interior}$, $\lambda>0$, and $\mountain$ be a $\G$-affine function. We define the \emph{$\G$-dual of $\arbitrary$ with vertex $\x$, base $\mountain$, and height $\lambda$} by
  \begin{align*}
    \Gdual{\arbitrary}{\x, \mountain,\lambda}&\defin \curly{\xbar\in\outertarget^{\cl}\mid \G(\y, \xbar, \H(\x, \xbar, \mountain(\x)))\leq \mountain(\y)+\lambda,\ \forall\;\y\in\arbitrary}.
  \end{align*}
  In other words, $\xbar \in \Gdual{\arbitrary}{\x, \mountain,\lambda}$ if and only if there exists some $\z$ such that
  \begin{equation*}
    \G(\x, \xbar, \z)=\mountain(\x)\;\;\textnormal{ and } \G(\y, \xbar, \z)\leq \mountain(\y)+\lambda,\ \forall\;\y\in\arbitrary.
  \end{equation*}
\end{DEF}
The following Propositions~\ref{prop: local to global} and \ref{prop: sections are convex} make essential use of the conditions \eqref{DualQQConv} and \eqref{QQConv}.
\begin{prop}\label{prop: local to global}
 If $u$ is a \emph{nice} $\G$-convex function, then $\coord{\Gsubdiff{u}{\x}}{\x, u(\x)}$ is convex for any $\x\in\outerdom$.
 
 If $\arbitrary\subset\outerdom^{\cl}$ is connected and $\mountain$ is a $\G$-affine function with $\Gfivelower< \mountain<\Gfiveupper$ on $\arbitrary^{\cl}$, then $\coord{\Gdual{\arbitrary}{\x, \mountain, \lambda}}{\x, \mountain(\x)}$ is convex for any $0<\lambda$ such that $\sup_{\arbitrary}\mountain+\lambda<\Gfiveupper$ and $\x\in\arbitrary^{\interior}$.
\end{prop}
\begin{proof}
Begin by fixing $\x\in\outerdom$ and $\xbar_0$, $\xbar_1\in\Gsubdiff{u}{\x}$. We let 
\begin{align*}
\xbar(t):=\Gseg{\xbar_0}{\xbar_1}{\x, u(\x)},\quad \z(t):=\H(\x, \xbar(t), u(\x)),
\end{align*}
 and define 
 \begin{align*}
 \rho(\y):=\sup_{t\in[0, 1]}\G(\y, \xbar(t), \z(t))
 \end{align*}
  for any $\y\in\outerdom$. Note since $u$ is \emph{nice}, $\xbar(t)$ is well-defined and contained in $\outertarget^{\cl}$ by \eqref{DualDomConv}. Also as a result, by \eqref{Lip} we see $\rho$ is continuous on $\outerdom$. Now consider the set
\begin{align*}
 \outerdom^\prime:=\curly{\y\in\outerdom\mid \rho(\y)\leq u(\y)}.
\end{align*}
Clearly $\x\in\outerdom^\prime$, and $\outerdom^\prime$ is relatively closed as a subset of $\outerdom$. We now aim to show that $\outerdom^\prime$ is relatively open, then we would obtain $\outerdom^\prime=\outerdom$ since $\outerdom$ is connected by \eqref{DomConv}. Since $u(\x)=\G(\x, \xbar(t), \z(t))$ for all $t\in[0, 1]$ by construction and $u$ is nice, \eqref{G5} implies that $(\x, \xbar(t), \z(t))\in\gendom$ for all $t\in[0, 1]$. As a result, we would have $\Gseg{\xbar_0}{\xbar_1}{\x, u(\x)}\subset\Gsubdiff{u}{\x}$, proving the proposition.

Note that since $\outerdom^{\cl}$ is compact and $u$ is nice, there exists some $\epsilon>0$ such that $\Gfivelower+\epsilon\leq u\leq \Gfiveupper-\epsilon$ on $\outerdom^{\cl}$. 
Suppose that $\y_0\in\outerdom^\prime$; thus $\rho(\y_0)\leq u(\y_0)\leq\Gfiveupper-\epsilon$. By continuity of $\rho$, there exists $\delta>0$ such that $\rho(\y)\leq \Gfiveupper-\epsilon/2$ for all $\gdist{\y}{\y_0}<\delta$. Fix such a $\y$, we claim that $\rho(\y)\leq u(\y)$ as well. If $\rho(\y)< \Gfivelower+\epsilon$, the claim is immediate. Otherwise let $[t_0, t_1]\subset [0, 1]$ be the maximal subinterval on which $\G(\y, \xbar(\cdot), \z(\cdot))\geq \Gfivelower+\epsilon$ that also contains a value $t_\y\in(t_0, t_1)$ where $\G(\y, \xbar(\cdot), \z(\cdot))$ is maximized; by possibly reversing the parametrization of $\xbar(t)$ let us assume $\G(\y, \xbar(t_0), \z(t_0))\geq \G(\y, \xbar(t_1), \z(t_1))$. Thus for any $t\in [t_0, t_1]$ we have $\G(\y, \xbar(t), \z(t))\in(\Gfivelower, \Gfiveupper)$, and in turn by \eqref{G5}, $(\y, \xbar(t), \z(t))\in \gendom$. As a result we can apply \eqref{DualQQConv} to the reparametrized $\G$-segment $\xbarhat(t):=\xbar((1-t)t_0+tt_1)$ to obtain
\begin{align*}
 \G(\y, \xbar(t_\y), \z(t_\y))&\leq \G(\y, \xbar(t_0), \z(t_0))+\frac{M(t_y-t_0)}{t_1-t_0}\brackets{\G(\y, \xbar(t_1), \z(t_1))-\G(\y, \xbar(t_0), \z(t_0))}_+\\
 &= \G(\y, \xbar(t_0), \z(t_0))\leq u(\y)
\end{align*}
as desired (the constant $M$ here actually depends on the specific value of $u(x)$, but it clearly does not affect the final inequality). 
This last inequality is due to the fact that $\G(\cdot, \xbar(0), \z(0))$ and $\G(\cdot, \xbar(1), \z(1))$ are supporting to $u$ from below (in the case $t_0=0$), while $\Gfivelower+\epsilon\leq u(\y)$ (in the case $t_0>0$).

To obtain the second claim, repeat nearly the same proof with $\xbar_0$, $\xbar_1\in \Gdual{\arbitrary}{\x, \mountain, \lambda}$, using $\xbar(t):=\Gseg{\xbar_0}{\xbar_1}{\x, \mountain(\x)}$, $\z(t):=\H(\x, \xbar(t), \mountain(\x))$, and $ \outerdom^\prime:=\curly{\y\in\arbitrary\mid \rho(\y)\leq \mountain(\y)+\lambda}$.
\end{proof}

\begin{cor}\label{cor: local to global}
 Suppose $u$ is a \emph{nice} $\G$-convex function as above. If $\mountain(\cdot)=\G(\cdot, \xbar, \z)$ is a $\G$-affine function such that $\mountain(\x_0)=u(\x_0)$ and $\mountain\leq u$ in some neighborhood of $\x_0\in\outerdom$, then $\xbar\in\Gsubdiff{u}{\x_0}$.
\end{cor}
\begin{proof}
 Suppose $\mountain$ is such a $\G$-affine function, locally supporting from below at a point $\x_0\in\outerdom$. Recall the \emph{subdifferential of $u$ at $\x_0$}, 
 \begin{align*}
 \subdiff{u}{\x_0}=\curly{p\in\cotanspM{\x_0}\mid u(\exp_{\x_0}{v})\geq u(\x_0)+\inner{p}{v}+o(\gnorm[\x_0]{v}),\ v\to 0}
 \end{align*}
  is a closed convex subset of $\cotanspM{\x_0}$, compact since $u$ is nice; here $\exp_{\x_0}$ is the usual Riemannian exponential map. We pause to remark here that since $\G$ is not assumed to be $C^2$ in the $\x$ variable, $u$ may not be semi-convex; however since it is $\G$-convex, it is easy to see that $\subdiff{u}{\x}\neq\emptyset$ for any $\x\in\outerdom$. By our current assumptions, $\Dx \mountain(\x_0)\in\subdiff{u}{\x_0}$. Our goal will now be to show that $\subdiff{u}{\x_0}=\coord{\Gsubdiff{u}{\x_0}}{\x_0, u(\x_0)}$, which would conclude the corollary as $\xbar=\X{\x_0}{u(\x_0)}{\Dx \mountain(\x_0)}$ by \eqref{Twist} (recall, since $u$ is nice, by \eqref{G5} we have $(\x_0, \xbar, \z)\in\gendom$).
 
To this end, let $\pbar_0$ be an exposed point of $\subdiff{u}{\x_0}$, i.e. for some unit length $v_0\in\tanspM{\x_0}$,
\begin{align}\label{eqn: exposed point of subdifferential}
\inner{\pbar-\pbar_0}{v_0}<0,\qquad \forall\; \pbar\in\subdiff{u}{\x_0}\setminus\curly{\pbar_0}. 
\end{align}
We will show that $(\x_0, \pbar_0)$ is a limit in $\cotanspM{}$ of $(\x_k, \Dx u(\x_k))$ for some sequence $\x_k\to\x_0$. If this were the case, since $u$ is nice, by \eqref{Twist} and \eqref{G5} we can see that $\curly{\X{\x_k}{u(\x_k)}{\Dx u(\x_k)}}=\Gsubdiff{u}{\x_k}$ for each $k$. Then by continuity of $\G$ and $u$, we have that $\X{\x_0}{u(\x_0)}{\pbar_0}\in \Gsubdiff{u}{\x_0}$, thus we could conclude that any exposed point of $\subdiff{u}{\x_0}$ is contained in $\coord{\Gsubdiff{u}{\x_0}}{\x_0, u(\x_0)}$. Since by \cite[Theorem 18.7]{Roc70}, $\subdiff{u}{\x_0}$ is the convex hull of its exposed points, combining with Proposition~\ref{prop: local to global} we would obtain $\subdiff{u}{\x_0}\subset\coord{\Gsubdiff{u}{\x_0}}{\x_0, u(\x_0)}$. The reverse inclusion is immediate, hence this would complete the proof.

 Now by Remark~\ref{rem: nice functions are nice}, $u$ is differentiable almost everywhere, hence we can choose a sequence $v_k\in\tanspM{\x_0}$ such that $u$ is differentiable at $x_k:=\exp_{\x_0}{v_k}$ while $\tfrac{v_k}{\gnorm[\x_0]{v_k}}\to v_0$, and $(\x_k, \Dx u(\x_k))$ converges to $(\x_0, \pbar_\infty)$ in $\cotanspM{}$ for some $\pbar_\infty\in\cotanspM{\x_0}$. 
In particular,
\begin{align*}
 u(\x_k)&\geq u(\x_0)+\inner{\pbar_0}{v_k}+o(\gnorm[\x_0]{v_k}),\\
 u(\x_0)&\geq u(\x_k)+\inner{\Dx u(\x_k)}{\exp_{\x_k}^{-1}{\x_0}}+o(\gnorm[\x_k]{\exp_{\x_k}^{-1}{\x_0}}).
\end{align*}
Plugging the second inequality above into the first, canceling terms, and dividing both sides by $\gnorm[\x_0]{v_k}$, we obtain
\begin{align*}
 0\geq \inner{\pbar_0}{\frac{v_k}{\gnorm[\x_0]{v_k}}}+\inner{\Dx u(\x_k)}{\frac{\exp_{\x_k}^{-1}{\x_0}}{\gnorm[\x_0]{v_k}}}+\gnorm[\x_0]{v_k}^{-1}(o(\gnorm[\x_0]{v_k})+o(\gnorm[\x_k]{\exp_{\x_k}^{-1}{\x_0}})).
\end{align*}
By using geodesic normal coordinates around $\x_0$, we find taking $k\to\infty$ that this leads to
\begin{align*}
 0\geq\inner{\pbar_0-\pbar_\infty}{v_0}.
\end{align*}
However, by continuity, $\Xbar{\x_k}{u(\x_k)}{\Dx u(\x_k)}\to \xbar_\infty$ for some $\xbar_\infty\in\Gsubdiff{u}{\x_0}$. Since $u$ is nice, by \eqref{G5} we must have $(\x_0, \xbar_\infty, \H(\x_0, \xbar_\infty, u(\x_0))\in\gendom$, thus we see that 
\begin{align*}
\pbar_\infty=\Dx \G(\x_0, \xbar_\infty, \H(\x_0, \xbar_\infty, u(\x_0)))\in\subdiff{u}{\x_0},
\end{align*}
and by \eqref{eqn: exposed point of subdifferential} we must have $\pbar_0=\pbar_\infty$ as desired.
\end{proof}

\begin{DEF}\label{DEF: G-cone}
 Suppose $u$ is a \emph{nice} $\G$-convex function, $\mountain$ is $\G$-affine, and let $\sublevelset:=\{ \x\in\outerdom\mid u(\x)\leq \mountain(\x)\}$ with $\x_0\in\sublevelset^{\interior}$. Then the \emph{$\G$-cone with base $\sublevelset$, vertex $\x_0$, and height $\mountain(x_0)-u(x_0)$} is the function defined by
\begin{align*}
 \Gcone{\x_0, \sublevelset}(\x):=\sup\curly{\G(\x, \xbar, \H(\x_0, \xbar, u(\x_0)))\mid \xbar\in\outertarget,\ \G(\y, \xbar, \H(\x_0, \xbar, u(\x_0)))\leq \mountain(\y),\ \forall\; \y\in\sublevelset}.
\end{align*}
\end{DEF}
\begin{rem}\label{rem: Gcone}
 Since $u$ is $\G$-convex, clearly $\Gcone{\x_0, \sublevelset}(\x_0)=u(\x_0)$. Now $\Gcone{\x_0, \sublevelset}$ may not be $\G$-convex on $\outerdom$ (given $\x\in\outerdom$, it is not clear that there exists an $\xbar\in\outertarget$ for which $(\x, \xbar, \H(\x, \xbar, \Gcone{\x_0, \sublevelset}(\x)))\in\gendom$). However, we can see that since $u$ is nice, by \eqref{G5} we have at the vertex $\x_0$,
\begin{align}
 \Gsubdiff{\Gcone{\x_0, \sublevelset}}{\x_0}=\curly{\xbar\in\outertarget\mid\G(\y, \xbar, \H(\x_0, \xbar, u(\x_0)))\leq \mountain(\y),\ \forall\;\y\in\sublevelset}\neq\emptyset.\label{eqn: Gcone subdiff at vertex}
\end{align}
Also note, as long as $u$ is \emph{nice} the proof of Proposition~\ref{prop: local to global} yields that $\coord{\Gsubdiff{\Gcone{\x_0, \sublevelset}}{\x_0}}{\x_0, u(\x_0)}$ is convex.
\end{rem}
\begin{lem}\label{lem: G-dual inside G-subdifferential image}
  Suppose $u$, $\mountain$, $\x_0\in\sublevelset^{\interior}$ are as in Definition~\ref{DEF: G-cone}, and suppose $\sublevelset\subset\outerdom^{\interior}$. Then
  \begin{align*}
    \Gsubdiff{\Gcone{\x_0, \sublevelset}}{\x_0}\subset\Gsubdiff{u}{\sublevelset}.
  \end{align*}
\end{lem}

\begin{proof}
  Fix $\xbar\in \Gsubdiff{\Gcone{\x_0, \sublevelset}}{\x_0}$ and define 
  \begin{align*}
    \zmax&\defin \max_{\x\in\sublevelset}\H(\x, \xbar, u(\x)),
  \end{align*}
then $\zmax=\H(\xmax, \xbar, u(\xmax))$ for some $\xmax \in \sublevelset^{\cl}$; since $u$ is nice, by \eqref{G5} it follows that $(\xmax, \xbar, \zmax)\in\gendom$.
  Since $\zmax \geq H(\x,\xbar,u(\x))$ for all $\x \in \sublevelset$ and $\G_z<0$, it follows that
  \begin{align*}
    \G(\x, \xbar, \zmax) \leq \G(x, \xbar, \H(\x, \xbar, u(\x)))=u(\x),\;\;\forall\;\x\in \sublevelset,
  \end{align*} 
  while
  \begin{align*}
    \G(\xmax, \xbar, \zmax) =u(\xmax).
  \end{align*}
  Now if $\xmax \in \sublevelset^{\bdry}$, we can calculate (recalling that $\xbar\in \Gsubdiff{\Gcone{\x_0, \sublevelset}}{\x_0}$)
\begin{align*}
\G(\xmax, \xbar, \H(\x_0, \xbar, u(\x_0)))&\leq \mountain(\xmax)\\
&=u(\xmax)=\G(\xmax, \xbar, \zmax)
\end{align*}
by the definition of $\zmax$. Then applying $\H(\xmax, \xbar, \cdot)$ to both sides, we have
\begin{align*}
 \H(\x_0, \xbar, u(\x_0))\geq \zmax,
\end{align*}
in other words we may actually choose $\xmax=\x_0$. Thus in any case, we may assume $\xmax \in \sublevelset^{\interior}$; then $\G(\x, \xbar, \zmax)$ locally supports $u$ from below in $\sublevelset$. In particular, $\xbar\in\Gsubdiff{u}{\sublevelset}$ by Corollary~\ref{cor: local to global}.
  \end{proof}

\begin{prop}\label{prop: sections are convex}
 Suppose $\mountain(\cdot):=\G(\cdot, \xbar, \z)$ is a \emph{nice} $\G$-affine function,
 and let $\sublevelset:=\curly{\x\in\outerdom^{\cl}\mid u(\x)\leq \mountain(\x)}$. Then $\coord{\sublevelset}{\xbar, \z}$ is convex.
\end{prop}
\begin{proof}
We remark that since $\mountain$ is nice, \eqref{G5} implies $\coord{\outerdom^{\cl}}{\xbar, \z}$ is well-defined; in turn \eqref{DomConv} implies it is convex.

 Fix any arbitrary $\G$-affine function $\mountainhat(\cdot)=\G(\cdot, \xbarhat, \zhat)$, and let $\sublevelsethat:=\curly{\x\in\outerdom^{\cl}\mid\mountainhat(\x)\leq \mountain(\x)}$. Consider $\x_0$, $\x_1\in\sublevelsethat$, and let $\x(s):=\Gseg{\x_0}{\x_1}{\xbar, \z}$; again since $\mountain$ is nice, \eqref{G5} and \eqref{DomConv} implies $x(s)$ is well-defined and remains in $\outerdom^{\cl}$ for all $s\in[0, 1]$.
 
  Now suppose
  \begin{align*}
 \G(\x_1, \xbarhat, \H(\x_0, \xbarhat, \mountain(\x_0)))>  \G(\x_1, \xbar, \H(\x_0, \xbar, \mountain(\x_0)))=\mountain(\x_1).
 \end{align*} 
 Clearly the expression on the left is in the domain of $\H(\x_1, \xbarhat, \cdot)$, while $\mountain(\x_1)$ is as well since $\mountain$ is nice. Thus we can take $\G(\x_0, \xbarhat, \H(\x_1, \xbarhat, \cdot))$ of both sides (which preserves monotonicity), to obtain 
  \begin{align*}
\mountain(\x_0)=\G(\x_0, \xbarhat, \H(\x_0, \xbarhat, \mountain(\x_0)))>  \G(\x_0, \xbarhat,\H(\x_1, \xbarhat, \mountain(\x_1))),
 \end{align*}
 thus by possibly relabelling $\x_0$ and $\x_1$, we can assume that 
 \begin{align*}
 \G(\x_1, \xbarhat, \H(\x_0, \xbarhat, \mountain(\x_0)))\leq \mountain(\x_1).
 \end{align*}
Now since $\mountain$ is nice, $\Gfivelower<\inf_\outerdom\mountain\leq\sup_\outerdom\mountain<\Gfiveupper$. Thus we may apply \eqref{QQConv} along $\x(s)$ with $[\QQConvlower, \QQConvupper]=[\inf_\outerdom\mountain, \sup_\outerdom\mountain]$ (also with some associated constant $M\geq 1$).
Doing so we find that
\begin{align*}
 \mountain(x(s))&=\mountain(x(s))+Ms\brackets{\G(\x_1, \xbarhat, \H(\x_0, \xbarhat, \mountain(\x_0)))-\mountain(\x_1)}_+\\
 &\geq \G(x(s), \xbarhat, \H(\x_0, \xbarhat, \mountain(\x_0)))\\
 &\geq \G(x(s), \xbarhat, \H(\x_0, \xbarhat, \mountainhat(\x_0)))=\mountainhat(x(s)).
\end{align*}
Here the inequality in the last line is due to the fact that $\mountainhat(\x_0)\leq \mountain(\x_0)$, combined with monotonicity properties of $\H$ and $\G$ in the scalar parameters. As a result, we see that $\coord{\sublevelsethat}{\xbar, \z}$ is convex.

 Finally note that, $u=\sup \mountainhat$ for some collection of $\G$-affine functions $\mountainhat$. Thus we can see that $\coord{\sublevelset}{\xbar, \z}=\bigcap\coord{\sublevelsethat}{\xbar, \z}$, which by the first part of the proof is an intersection of convex sets and must be convex itself.
\end{proof}

\subsection{$\G$ and the Riemannian metric} From this point through the end of Section \ref{section: engulfing}, we assume that $\G$ satisfies \eqref{Twist}, \eqref{DualTwist}, \eqref{Nondeg}, and \eqref{QQConv}, \eqref{DualQQConv}, and let $u$ be a \emph{very nice} $\G$-convex function with associated \emph{very nice} interval $\Niceinterval\subset(\Gfivelower, \Gfiveupper)$.

\begin{rem}\label{rem: universal constants}
By an abuse of notation, we will often refer to a \emph{very nice constant}, by which we mean a constant that depends on $\Niceinterval$, the domains $\outerdom$, $\outertarget$, the dimension $n$, and the constant $\lipbound$ in \eqref{Lip} through the following quantities: the modulus of continuity of $E$ and $E^{-1}$, $\sup{\norm{\det{\nondegmatrixentries}}^{\pm 1}}$, $\sup{\norm{\det{\dualnondegmatrixentries}}^{\pm 1}}$, $\sup\Norm{\nondegmatrixentries}^{\pm 1}$, $\sup\Norm{\dualnondegmatrixentries}^{\pm 1}$ ($\|\cdot\|$ being the Hilbert-Schmidt norm of the matrix), $\inf\norm{\Gz}$, $\sup\norm{\Gz}$, $\inf\norm{\Hu}$, $\sup\norm{\Hu}$, and $M\geq 1$ corresponding to $\Niceinterval$ from \eqref{QQConv} and \eqref{DualQQConv}. All suprema and infima above are taken over $\x\in\outerdom$, $\xbar\in\outertarget$, $\u\in\Niceinterval$, and with the understanding that $\z=\H(\x, \xbar, \u)$; the above quantities can be assumed finite and nonzero by \eqref{Nondeg} and \eqref{G5}. The rationale for this terminology is that in various situations, $\outerdom$, $\outertarget$, $n$, and the various quantities involving $\G$ and $\H$ are fixed, with the only real dependence on the constant coming from the range of the scalar parameter $\u$ which will be constrained in the interval $\Niceinterval$; since we generally fix one \emph{very nice} function, the interval $\Niceinterval$ will be fixed as well.
\end{rem}

\begin{lem}\label{lem: comparability}
If $(\xbar, \z)\in\outertarget^{\cl}\times \real$ satisfies the condition 
\begin{align}\label{eqn: very nice focus}
\G(\cdot, \xbar, \z)\in\Niceinterval\text{ on all of }\outerdom^{\cl},
\end{align}
then $\pmap{\xbar}{\z}{\cdot}$ is a bi-Lipschitz mapping from $\outerdom^{\cl}$ to $\coord{\outerdom^{\cl}}{\xbar, \z}$. Moreover the Lipschitz constants of both this map and its inverse are bounded by some \emph{very nice} constant.

Similarly, if $(\x, \u)\in \outertarget^{\cl}\times\Niceinterval$, then $\pbarmap{\x}{\u}{\cdot}$ is a bi-Lipschitz mapping from $\outertarget^{\cl}$ to $\coord{\outertarget^{\cl}}{\x, \u}$, and the Lipschitz constants of $\pbarmap{\x}{\u}{\cdot}$ and its inverse are bounded by a \emph{very nice} constant.
\end{lem}
\begin{proof}
Before we begin, recall the definitions of $d_\outerdom$ and $L_g$ introduced in Remark~\ref{rem: existence of very nice solutions}. 
Fix $(\xbar, \z)\in\outertarget^{\cl}\times \real$ satisfying \eqref{eqn: very nice focus}. By \eqref{G5}, we then have $(\x, \xbar, \z)\in\gendom$ for any $\x\in\outerdom^{\cl}$. Fix $\x_1$, $\x_2\in\outerdom^{\cl}$, then by \eqref{DomConv} the $\G$-segment $\x(s):=\Gseg{\x_1}{\x_2}{\xbar, \z}$ is well-defined and remains in $\outerdom^{\cl}$, in particular it is differentiable for all $s\in[0, 1]$.

Then by \eqref{eqn: xdot representation} we calculate
\begin{align}\label{eqn: Gsegment length calculation}
 L_g(\Gseg{\x_1}{\x_2}{\xbar, \z})&=\int_0^1\gnorm[\x(s)]{-\Gz(\x(s), \xbar, \z)\dualnondegmatrixinv{\x(s)}{\xbar}{\z}(\pmap{\xbar}{\z}{\x_2}-\pmap{\xbar}{\z}{x_1})}ds.
\end{align}
Now since $\G(\x_1, \xbar, \z)\in\Niceinterval$, we see that $-\Gz(\x(s), \xbar, \z)$ has \emph{very nice}, positive upper and lower bounds, while the operator norms of $\dualnondegmatrixinv{\x(s)}{\xbar}{\z}$ and $\dualnondegmatrix{\x(s)}{\xbar}{\z}$ also have \emph{very nice} upper bounds. Thus we see for some \emph{very nice} $C>0$,
\begin{align*}
C^{-1} \gbarnorm[\xbar]{\pmap{\xbar}{\z}{\x_2}-\pmap{\xbar}{\z}{\x_1}}\leq  L_g(\Gseg{\x_1}{\x_2}{\xbar, \z})\leq C \gbarnorm[\xbar]{\pmap{\xbar}{\z}{\x_2}-\pmap{\xbar}{\z}{\x_1}}.
\end{align*}
Clearly we always have $\gdist{\x_1}{\x_2}\leq L_g(\Gseg{\x_1}{\x_2}{\xbar, \z})$, so this implies that $\X{\xbar}{\z}{\cdot}$ is globally Lipschitz on $\coord{\outerdom^{\cl}}{\xbar, \z}$ with a Lipschitz constant that is \emph{very nice}.

We now prove $\pmap{\xbar}{\z}{\cdot}$ is globally Lipschitz from $\outerdom^{\cl}$ to $\coord{\outerdom^{\cl}}{\xbar, \z}$, and its Lipschitz constant is bounded by some \emph{very nice} constant. Indeed, first note that this mapping is $C^1$ on $\outerdom^{\cl}$, with $C^1$ norm bounded by a \emph{very nice} constant, thus it is sufficient to show there exists a \emph{very nice} constant $C>0$ such that for any $\x_1$, $\x_2\in\outerdom^{\cl}$, there is a piecewise $C^1$ curve $\gamma$ connecting $\x_1$ to $\x_2$, remaining entirely within $\outerdom^{\cl}$, for which $C\gdist{\x_1}{\x_2}\geq L_g(\gamma)$. Suppose this is not the case, then there is a sequence $\x^1_k$, $\x^2_k$ for which $\frac{\gdist{\x^1_k}{\x^2_k}}{d_\outerdom(\x^1_k, \x^2_k)}\to 0$. By compactness of $\outerdom^{\cl}$, we can assume $\x^1_k$ and $x^2_k$ converge. By \eqref{eqn: Gsegment length calculation}, we see that $d_\outerdom(\x^1_k, \x^2_k)$ has a uniform, \emph{very nice} upper bound, hence it must be that $\gdist{\x^1_k}{\x^2_k}\to 0$. This implies that both $x^1_k$ and $x^2_k$ converge to some $x_\infty\in \outerdom^{\cl}$, while by continuity of $\pmap{\xbar}{\z}{\cdot}$ on $\outerdom^{\cl}$, we must have $\pmap{\xbar}{\z}{x^1_k}$ and $\pmap{\xbar}{\z}{x^2_k}$ converging to some $\p_\infty\in \coord{\outerdom^{\cl}}{\xbar, \z}$.
Additionally, note that since $\pmap{\xbar}{\z}{\cdot}$ has a \emph{very nice} upper bound on its $C^1$ norm, it is \emph{locally} Lipschitz in $\outerdom^{\interior}$ with a \emph{very nice} constant. Thus if $\x_\infty\in\outerdom^{\interior}$, by combining with \eqref{eqn: Gsegment length calculation} we would have for large enough $k$, 
\begin{align*}
\frac{\gdist{\x^1_k}{\x^2_k}}{d_\outerdom(\x^1_k, \x^2_k)}&\geq \frac{\gdist{\x^1_k}{\x^2_k}}{L_g(\Gseg{\x_1}{\x_2}{\xbar, \z})} \geq \frac{\gdist{\x^1_k}{\x^2_k}}{C \gbarnorm[\xbar]{\pmap{\xbar}{\z}{\x_2}-\pmap{\xbar}{\z}{\x_1}}}\geq C,
\end{align*}
 a contradiction. Thus it must be that the limiting points $\x_\infty\in \outerdom^{\bdry}$ and $\p_\infty\in \coord{\outerdom^{\bdry}}{\xbar, \z}$.

At this point we make an aside to show that the domain $\outerdom$ has a Lipschitz boundary in the sense that any point in $\outerdom^{\bdry}$ has an open neighborhood (in $\M$) on which it can be represented as the graph of a Lipschitz function in some local coordinate system, where the Lipschitz constant of this function is uniformly bounded. Since $\outerdom^{\bdry}$ is compact, it is clearly sufficient to show the boundary is locally Lipschitz. Fix a point $\x\in\outerdom^{\bdry}$, and a small open neighborhood $\mathcal{O}$ of $\x$ in $\M$. By the extension lemma (for example, \cite[Lemma 2.27]{LeeJ13}) there exists a $C^1$ extension of $\pmap{\xbar}{\z}{\cdot}$ to $\mathcal{O}$. Since the derivative of $\pmap{\xbar}{\z}{\cdot}$ is invertible on $\outerdom^{\cl}$ by \eqref{Nondeg}, by possibly shrinking $\mathcal{O}$ we can assume that this extension is also a $C^1$ diffeomorphism on $\mathcal{O}$. We continue to use the notation $\X{\xbar}{\z}{\cdot}$ to refer to the inverse of this extension. Now take a ball $B\subset \cotanspMbar{\xbar}$ centered at $\p:=\pmap{\xbar}{\z}{\x}$, small enough so its closure is contained in $\coord{\mathcal{O}}{\xbar, \z}$. Then $B\cap\coord{\outerdom}{\xbar, \z}$ is open, $\coord{\mathcal{O}}{\xbar, \z}$ is an open neighborhood of the closure of $B\cap\coord{\outerdom}{\xbar, \z}$, and $\p$ is contained in the boundary of $B\cap\coord{\outerdom}{\xbar, \z}$. Moreover, $B\cap\coord{\outerdom}{\xbar, \z}$ is convex by \eqref{DomConv}, hence has a locally Lipschitz boundary. Thus we can apply \cite[Theorem 4.1]{HofmannMitreaTaylor07} to find that $\outerdom^{\bdry}$ is locally Lipschitz near $\x$, finishing our aside.

Finally, we return to our main argument. Fix local coordinates near $\x_\infty$, using these coordinates we identify a neighborhood of $\x_\infty$ with a subset of $\real^n$. By the aside above, we find a neighborhood on which $\outerdom^{\cl}$ is written as the graph of a Lipschitz function $\Phi$, over some subset of $\real^{n-1}$ in these coordinates. If $k$ is large enough, then $x^1_k$ and $x^2_k$ are contained in this neighborhood. We now define a special curve $\gamma_k$. Draw a straight line segment between $x^1_k$ and $x^2_k$. If this segment does not intersect $\outerdom^{\bdry}$, then we take $\gamma_k$ to be this segment. Otherwise, between the first and last points where the segment intersects $\outerdom^{\bdry}$, take $\gamma_k$ as the image under $\Phi$ of the projection of this line segment onto $\real^{n-1}$. Clearly $\gamma_k$ is then a Lipschitz curve with $L_g(\gamma_k)\leq C\gdist{x^1_k}{x^2_k}$ for some constant $C>0$ depending only on the domain $\outerdom$ (independent of $k$). In turn, this implies a bound $d_\outerdom(\x^1_k, \x^2_k)\leq C\gdist{x^1_k}{x^2_k}$ on the intrinsic distance, thus we cannot have $\frac{\gdist{\x^1_k}{\x^2_k}}{d_\outerdom(\x^1_k, \x^2_k)}\to 0$, finishing our proof.

The statement concerning $\outertarget$ is proven similarly, but using \eqref{eqn: xbardot representation} and \eqref{eqn: zdot representation} in place of \eqref{eqn: xdot representation} in obtaining the analogue of \eqref{eqn: Gsegment length calculation}, and\eqref{DualDomConv} in place of \eqref{DomConv} in various places.
\end{proof}
The above lemma immediately yields the following corollary.
\begin{cor}\label{cor: comparability}
 Let us write $a\sim b$ to mean there exists a \emph{very nice} constant $C>0$ for which $C^{-1} a\leq b\leq Ca$. Then under \eqref{eqn: very nice focus}, we have
\begin{align}\label{eqn: source distance comparison}
 \gdist{\x_1}{\x_2}&\sim\gbarnorm[\xbar]{\pmap{\xbar}{\z}{\x_1}-\pmap{\xbar}{\z}{\x_2}},\quad \forall\;\x_1,\ \x_2\in\outerdom^{\cl}.
\end{align}
 Also for any $(\x, \u)\in \outertarget^{\cl}\times\Niceinterval$,
 \begin{align}\label{eqn: target distance comparison}
 \gbardist{\xbar_1}{\xbar_2}&\sim\gnorm[\x]{\pbarmap{\x}{\u}{\xbar_1}-\pbarmap{\x}{\u}{\xbar_2}},\quad \forall\;\xbar_1,\ \xbar_2\in\outertarget^{\cl}.
\end{align}
Finally, in each of the respective situations above, we have
\begin{align}
 \Leb{\arbitrary}&\sim \Leb{\coord{\arbitrary}{\xbar, \z}},\notag\\
 \Leb{\arbitrarybar}&\sim \Leb{\coord{\arbitrarybar}{\x, \u}}\label{eqn: volume comparability}
\end{align}
for any measurable $\arbitrary\subset\outerdom^{\cl}$ or $\arbitrarybar\subset\outertarget^{\cl}$.
\end{cor}
Finally, we present a lemma relating the difference of two $\G$-affine functions with the difference of their linearizations. The lemma relies on \eqref{DualQQConv} in a crucial way.
\begin{lem}\label{lem: linearization}
Let $\x_0\in\outerdom$, $\xbar_0$, $\xbar_1\in\outertarget$, $\u_0\in \Niceinterval$, $\z_0:=\H(\x_0, \xbar_0, \u_0)$ for $i=0$, $1$, and $\xbar(t):=\Gseg{\xbar_0}{\xbar_1}{\x_0, \u_0}$. Then there exists a \emph{very nice} constant $C>0$ such that for any $\x\in\outerdom$ satisfying $(\x, \xbar(t), \H(\x_0, \xbar(t), \u_0))\in\gendom$ for all $t\in[0, 1]$, we have
\begin{align}
 & \inner{\pmap{\xbar_0}{\z_0}{\x}-\pmap{\xbar_0}{\z_0}{\x_0}}{\nondegmatrixinv{\x_0}{\xbar_0}{\z_0}(\pbarmap{\x_0}{\u_0}{\xbar_1}-\pbarmap{\x_0}{\u_0}{\xbar_0})}\notag\\
 &\qquad\leq \frac{C}{1-t'} \brackets{\G(\x, \xbar_1, \H(\x_0, \xbar_1, \u_0))-\G(\x, \xbar(t'), \H(\x_0, \xbar(t'), \u_0))}_+, \quad\forall\;t'\in [0, 1).\label{eqn: linearization lower bound}
\end{align}
Additionally, for any $\x\in\outerdom$ such that $x(s):=\Gseg{\x_0}{\x}{\xbar_0, \z_0}$ is well-defined and contained in $\outerdom^{\cl}$, we have
\begin{align}
& \norm{\G(\x, \xbar(t), \H(\x_0, \xbar(t), \u_0))-\G(\x, \xbar_0, \z_0)}\notag\\
&\leq Ct\gbarnorm[\xbar_0]{\pmap{\xbar_0}{\z_0}{\x}-\pmap{\xbar_0}{\z_0}{\x_0}}\gnorm[\x_0]{\pbarmap{\x_0}{\u_0}{\xbar_1}-\pbarmap{\x_0}{\u_0}{\xbar_0}},\quad\forall\;\;t\in [0, 1].\label{eqn: linearization upper bound}
\end{align}
\end{lem}
\begin{proof}
To obtain the first inequality we first calculate (also using \eqref{eqn: xbardot representation}):
 \begin{align*}
  &\left.\diff{t}\G(\x, \xbar(t), \H(\x_0, \xbar(t), \u_0))\right\vert_{t=0}\\
  &=\inner{\Dbar \G(\x, \xbar_0, \z_0)+\Gz(\x, \xbar_0, \z_0)\Dbar\H(\x_0, \xbar_0, \u_0)}{\xbardot(0)}\\
  &=\Gz(\x, \xbar_0, \z_0)\inner{\frac{\Dbar \G}{\Gz}(\x, \xbar_0, \z_0)+\Dbar\H(\x_0, \xbar_0, \u_0)}{\xbardot(0)}\\
  &=-\Gz(\x, \xbar_0, \z_0)\inner{-\frac{\Dbar \G}{\Gz}(\x, \xbar_0, \z_0)+\frac{\Dbar \G}{\Gz}(\x_0, \xbar_0, \z_0)}{\nondegmatrixinv{\x_0}{\xbar_0}{\z_0}(\pbarmap{\x_0}{\u_0}{\xbar_1}-\pbarmap{\x_0}{\u_0}{\xbar_0})},
\end{align*}
and note $-\Gz(\x, \xbar_0, \z_0)=-\Gz(\x, \xbar_0, \H(\x_0, \xbar_0, \u_0))$ has strictly positive upper and lower bounds that are \emph{very nice}. The inequality \eqref{eqn: linearization lower bound} then follows by dividing \eqref{DualQQConv} through by $t>0$ and taking the limit as $t\searrow 0$.

The inequality \eqref{eqn: linearization upper bound} follows by a simple but tedious calculation:
 \begin{align*}
    &\G(\x,\xbar(t),\z(t))-\G(\x,\xbar_0,\z_0) \notag\\
    &= \G(\x(1), \xbar(t), \z(t))-\G(\x(1), \xbar(0), \z(0))+\G(\x(0), \xbar(t), \z(t))-\G(\x(0), \xbar(0), \z(0))\notag\\
    &=\int_0^1 \diff{s}\paren{\G(\x(s), \xbar(t), \z(t))-\G(\x(s), \xbar(0), \z(0))}ds\notag\\
    &=\int_0^1\inner{\Dx\G(\x(s),\xbar(t), \z(t))-\Dx\G(\x(s),\xbar(0), \z(0))}{\xdot(s)}ds\notag\\
    &=\int_0^t\int_0^1\pdiff{t'}\inner{\Dx\G(\x(s),\xbar(t'), \z(t'))}{\xdot(s)}ds dt'\notag\\
    &=\int_0^t\int_0^1\inner{\Dbar\Dx\G(\x(s),\xbar(t'), \z(t'))\xbardot(t')+\Dx\Gz(\x(s), \xbar(t'), \z(t'))\zdot(t')}{\xdot(s)}ds dt'.
  \end{align*}
  Now if we write 
\begin{align*}
\p:=\pmap{\xbar_0}{\z_0}{\x},\quad&\p_0:=\pmap{\xbar_0}{\z_0}{\x_0},\\
 \pbar_1:=\pbarmap{\x_0}{\u_0}{\xbar_1},\quad &\pbar_0:=\pbarmap{\x_0}{\u_0}{\xbar_0},
\end{align*}
by Proposition \ref{prop: motivation for nondegeneracy} \eqref{eqn: xdot representation}, \eqref{eqn: xbardot representation}, and \eqref{eqn: zdot representation}, the final expression in the calculation above can be written as  
\begin{align*}
 \int_0^t\int_0^1 \inner{M_{t', s}(\pbar_1-\pbar_0)}{\p-\p_0}+\inner{V_{t', s}}{\p-\p_0}\inner{\Vbar_{t'}}{\pbar_1-\pbar_0}ds dt'
\end{align*}
  for some linear transformations $M_{t', s}: \cotanspM{\x_0}\to\tanspM{\basexbarA}$, and vectors $V_{t', s}\in \tanspMbar{\basexbar}$, $\Vbar_{t'}\in\tanspM{\x_0}$. As $\u_0\in\Niceinterval$, a routine calculation yields that $\gbarnorm[\basexbar]{V_{t', s}}$, $\gnorm[\x_0]{\Vbar_{t'}}$, and the operator norm of $M_{t', s}$ have \emph{very nice} bounds, thus by applying Cauchy-Schwarz we obtain \eqref{eqn: linearization upper bound}.
\end{proof}

\section{An Aleksandrov-type estimate}\label{section: Aleksandrov estimate}

\begin{DEF}\label{DEF:notation_supporting_plane}
  If $\arbitrary\subset\cotanspMbar{\xbar}$ is convex and $\direction{}\in\S^{n-1}\subset \cotanspMbar{\xbar}$ is a unit direction, we denote \emph{the supporting plane to $\arbitrary$ with outward normal $\direction{}$} by $\plane{\arbitrary}{\direction{}}$.
\end{DEF}

  We also recall the standard notion from Riemannian geometry of the musical isomorphism
  \begin{DEF}\label{def:musical_isomorphism}
   If $v\in \cotanspMbar{\xbar}$ for some $\xbar\in \Mbar$, then we define $\raisecovect{v}\in \tanspMbar{\xbar}$ implicitly by the relation  
\begin{align*}
\inner{\raisecovect{v}}{w}= \innergbar[\xbar]{v}{w},\quad \forall\;w\in\cotanspMbar{\xbar}.
\end{align*}
The map $\raisecovect{}: \cotanspMbar{\xbar}\to\tanspMbar{\xbar}$ is called the \emph{musical isomorphism}.
  \end{DEF}	  
  
  \begin{rem}\label{rem:b_bet}
	  We also recall the following very simple elementary formula for the distance from a point in a set to a supporting plane of the set: if $\mathcal{A}\subset \cotanspMbar{\xbar}$ is convex, $p_0\in \mathcal{A}$, and $v\in \cotanspMbar{\xbar}$ is unit length for some $\xbar\in\Mbar$, then 
\begin{align*}
 \dist{p_0}{\plane{\mathcal{A}}{v}}=\sup_{p\in \mathcal{A}}{\innergbar[\xbar]{v}{p-p_0}}=\sup_{p\in \mathcal{A}^\bdry}{\innergbar[\xbar]{v}{p-p_0}}.
\end{align*}
  \end{rem}	  

\begin{thm}[John's Lemma]\label{John's lemma}
  If $\arbitrary\subset \real^n$ is a convex set with nonempty interior, there exists an ellipsoid $\ellipsoid$ whose center of mass coincides with that of $\arbitrary$, and a constant $\alpha(n)$ depending only on $n$ such that 
\begin{align*}
 \alpha(n)\ellipsoid\subset\arbitrary\subset\ellipsoid.
\end{align*}
\end{thm}

  In this section, we assume the hypotheses of Theorem~\ref{thm: G-aleksandrov estimate}. Namely, we fix a \emph{very nice} $\G$-convex function $\u$ with associated \emph{nice} interval $\Niceinterval$, a \emph{nice} $\G$-affine function $\basemountainA=\G(\cdot,\basexbarA,\basezA)$, and a point $\x_0\in\sublevelset^{\interior}$ where $\sublevelset := \curly{\x\in\outerdom \mid \u(\x)\leq \basemountainA(\x)}$. Again, since $\basemountainA$ is nice, by \eqref{G5} and \eqref{DomConv} we have that $\coord{\outerdom}{\basexbarA, \basezA}$ is well-defined and convex. We also assume that $\diam{(\sublevelset)}<\epsilon$ for some \emph{very nice} constant $\epsilon>0$, to be determined, and there exists some ball $B\subset \cotanspMbar{\basexbarA}$ such that $\sectioncoord\subset B\subset 3B\subset \coord{\outerdom}{\basexbarA, \basezA}$ (this ball $B$ does not have to be of uniform size).
 \begin{figure}[H]
  \centering
    \includegraphics[height=.45\textwidth]{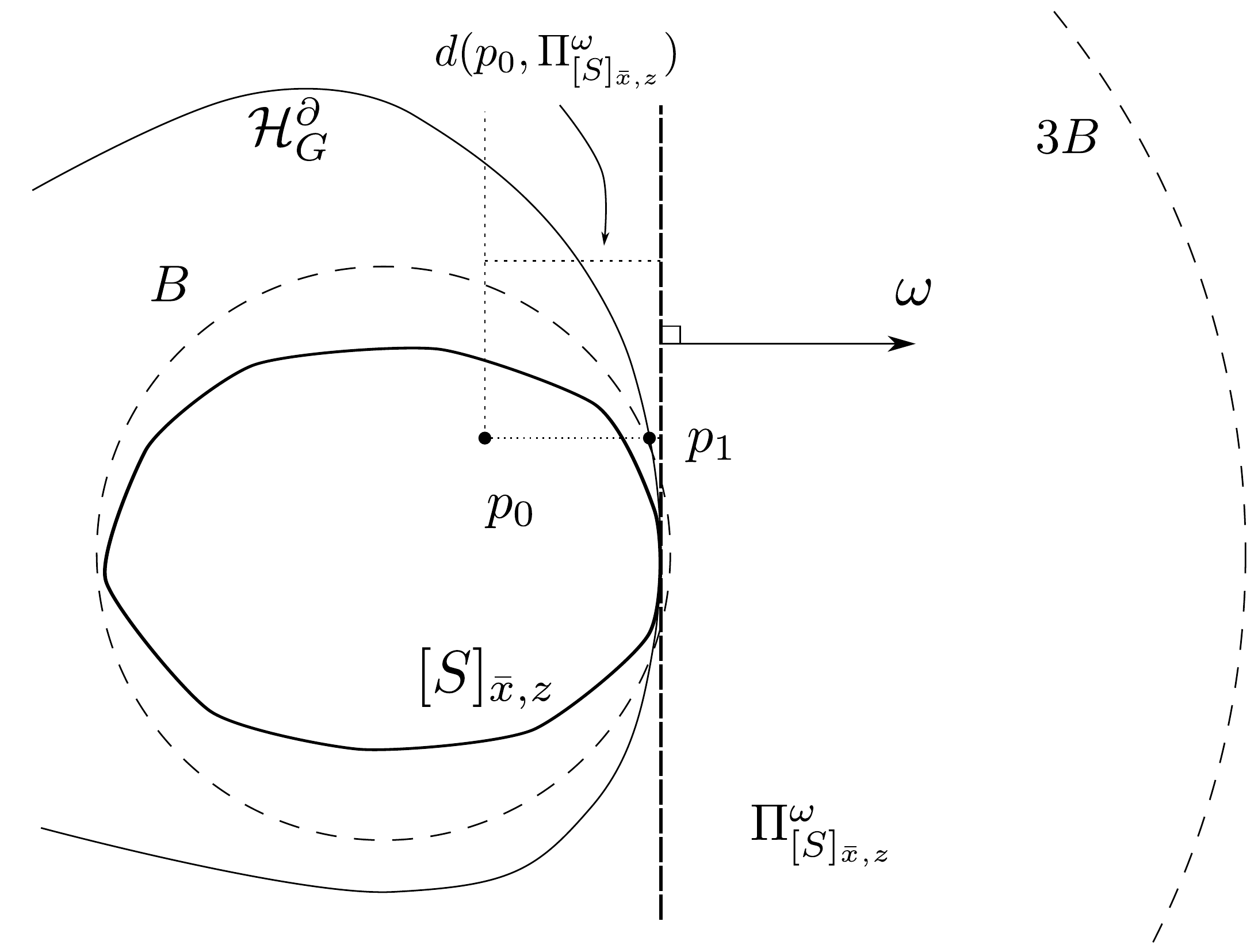}
     \caption{}\label{figure: Sec4fig1}
\end{figure}
   \begin{lem}\label{lem: G-distance bound}
    Given $\direction{} \in \mathbb{S}^n\subset \cotanspMbar{\basexbarA}$, there exists $\turnedfocus\in\Gsubdiff{\Gcone{\x_0, \sublevelset}}{\x_0}$ and a \emph{very nice} constant $C>0$ such that
	\begin{align}
      \norm{\frac{\nondegmatrix{\x_0}{\basexbarA}{\basezA}\raisecovect{\direction{}}}{\gnorm[\x_0]{\nondegmatrix{\x_0}{\basexbarA}{\basezA}\raisecovect{\direction{}}}}-\frac{\pbarmap{\x_0}{u(\x_0)}{\turnedfocus}-\pbarmap{\x_0}{u(\x_0)}{\basexbarA}}{\gnorm[\x_0]{\pbarmap{\x_0}{u(\x_0)}{\turnedfocus}-\pbarmap{\x_0}{u(\x_0)}{\basexbarA}}}}&< \frac{1}{2\sqrt{n}},\label{eqn: small rotation}\\
      C\gnorm[\x_0]{\pbarmap{\x_0}{u(\x_0)}{\turnedfocus}-\pbarmap{\x_0}{u(\x_0)}{\basexbarA}}&\geq \frac{\basemountainA(\x_0)-u(\x_0)}{\dist{\pmap{\basexbar}{\basez}{\x_0}}{\plane{\sectioncoord}{\direction{}}}}.\label{eqn: bound on direction in Gdual}
    \end{align}	  
  \end{lem}

\begin{proof}
Fix an $\direction{} \in \mathbb{S}^n\subset \cotanspMbar{\basexbarA}$, and take any $\contactpoint\in\sublevelset^{\bdry}$ such that $\contactpointcoord\in\sectioncoord^{\bdry}\cap \plane{\sectioncoord}{\direction{}}$.  We intend to follow the proof in \cite[Lemma 4.7]{GK14}, but since $u$ may not be semi-convex we must find an appropriate alternative to \cite[Corollary 23.7.1]{Roc70} which was utilized there. First we define $U(\p):=u(\X{\basexbarA}{\basezA}{\p})-\basemountainA(\X{\basexbarA}{\basezA}{\p})$, it is clear that this function is $\Gtil$-convex on $\coord{\outerdom}{\basexbarA, \basezA}$ where 
\begin{align*}
\Gtil(\p, \xbar', \z'):=\G(\X{\basexbarA}{\basezA}{\p}, \xbar', \z')-\basemountainA(\X{\basexbarA}{\basezA}{\p}),
\end{align*}
 with $\gendomtil:=\curly{(\p, \xbar', \z')\mid (\X{\basexbarA}{\basezA}{\p}, \xbar', \z')\in\gendom}$. $\Gtil$ has the same $C^2$ regularity as $\G$, and satisfies the same set of conditions, including \eqref{QQConv} and \eqref{DualQQConv}. Thus as in the proof of Corollary~\ref{cor: local to global}, we can show that $\subdiff{U}{\p}=\curly{\lim_{k\to\infty}DU(\p_k)\mid \p_k\to\p}$. Now by \cite[Theorem 2.5.1]{Clarke90}, this implies that $\subdiff{U}{\p}=\partial^C U(\p)$, where $\partial^CU$ is the \emph{Clarke} or \emph{generalized subdifferential} of $U$ (see \cite[Chapter 2.1]{Clarke90}). Since $\sectioncoord=\curly{\p\in\coord{\outerdom^{\cl}}{\basexbarA, \basezA}\mid U(\p)\leq 0}$, by combining \cite[Theorem 2.4.7, Corollary 1]{Clarke90} and \cite[Proposition 2.4.4]{Clarke90} we find there exists a $t^*>0$ such that $t^*\direction{}\in\subdiff{U}{\contactpointcoord}$ (again, identifying $T^*_{\contactpointcoord}\cotanspMbar{\basexbarA}\cong\cotanspMbar{\basexbarA}\cong\tanspMbar{\basexbarA}$). Thus we can continue as in the proof of \cite[Lemma 4.7]{GK14}, to find that if 
\begin{align}
\turnedseg(t):&= \X{\contactpoint}{\basemountainA(\contactpoint)}{\pbarmap{\contactpoint}{\basemountainA(\contactpoint)}{\basexbarA}+t \nondegmatrix{\contactpoint}{\basexbarA}{\basezA}\raisecovect{\direction{}}}\notag\\
&=\X{\contactpoint}{u(\contactpoint)}{\pbarmap{\contactpoint}{u(\contactpoint)}{\basexbarA}+t \nondegmatrix{\contactpoint}{\basexbarA}{\basezA}\raisecovect{\direction{}}},\label{eqn: tilting at boundary}
\end{align}
then $\turnedseg(t^*)\in\Gsubdiff{u}{\contactpoint}$ (we have used here that $\basemountainA(\contactpoint)=u(\contactpoint)$ since $\sublevelset^\bdry\subset\outerdom^{\interior}$).

Then writing $\turnedz(t):=\H(\contactpoint, \turnedseg(t), u(\contactpoint))$, we see that $\G(\x_0, \turnedseg(t^*), \turnedz(t^*))\leq u(\x_0)$ while $\G(\cdot, \turnedseg(0), \turnedz(0))\equiv \basemountainA$, hence there exists some value $0\leq t^{**}\leq t^*$ for which 
\begin{align}\label{eqn: characterizing good focus}
\G(\x_0, \turnedseg(t^{**}), \turnedz(t^{**}))=u(\x_0);
\end{align}
 let us write
\begin{align*}
 \turnedfocus:&=\turnedseg(t^{**}),\ \Gconez:=\turnedz(t^{**}).
\end{align*}
Since $\basemountainA$ is nice, we may take $\outerdom^\prime:=\curly{\y\in\sublevelset\mid \sup_{t\in[0, t^*]}\G(y, \turnedseg(t), \turnedz(t))\leq \basemountainA(\y)}$ and $\turnedseg(t)$, $\turnedz(t)$ in place of $\xbar(t)$, $\z(t)$ in the proof of Proposition~\ref{prop: local to global} to see that $\G(\cdot, \turnedfocus, \Gconez)\leq \basemountainA(\cdot)$ on $\sublevelset$, or in other words (recalling \eqref{eqn: Gcone subdiff at vertex}) $\turnedfocus\in\Gsubdiff{\Gcone{\x_0, \sublevelset}}{\x_0}$ as desired.

Now recalling that $\sectioncoord\subset B$ for some ball, it is not hard to see that writing
\begin{align*}
 \p_0:=\pmap{\basexbar}{\basez}{\x_0},
\end{align*}
the orthogonal projection of $\p_0$ onto $\plane{\sectioncoord}{\direction{}}$ is contained in $3B\subset \coord{\outerdom}{\basexbarA, \basezA}$, and hence the whole line segment in between (by \eqref{DomConv} and since $\basemountainA$ is nice). At the same time, $\mathcal{H}_\G:=\coord{\curly{\G(\cdot, \turnedfocus, \Gconez)\leq \basemountainA(\cdot)}}{\basexbarA, \basezA}$ is convex by Proposition~\ref{prop: sections are convex}, and by differentiating $\G(\cdot, \turnedfocus, \Gconez)$ it can be seen that $\direction{}$ is an outer unit normal to $\mathcal{H}_\G$ at $\contactpointcoord$. Thus, there exists $\parallelpointcoord$ in the intersection of $\mathcal{H}_\G^\bdry$ with the ray $\curly{\p_0+s\direction{}\mid s\geq 0}$ with 
\begin{align*}
\parallelpoint:=\X{\basexbarA}{\basezA}{\parallelpointcoord}\in\outerdom,
\end{align*}
 and
\begin{align}
\gbarnorm[\basexbarA]{\parallelpointcoord-\p_0}&\leq \dist{\p_0}{\plane{\sectioncoord}{\direction{}}},\label{eqn: bound by plane distance}
\end{align}
(see Figure~\ref{figure: Sec4fig1}).
  Now let us write
  \begin{align*}
    \x(s) & := \Gseg{\x_0}{\parallelpoint}{\basexbarA,\basezA},\qquad\xbar(t)  := \Gseg{\basexbarA}{\turnedfocus}{\x_0,u(\x_0)},\\
	\z(t) & :=\H(\x_0,\xbar(t),u(\x_0))\\
	 \turnedfocuscoord:&=\pbarmap{\x_0}{u(\x_0)}{\turnedfocus},\qquad \pbar_0:=\pbarmap{\x_0}{u(\x_0)}{\basexbarA};
  \end{align*}
  note that $\z(1)=\Gconez$ by \eqref{eqn: characterizing good focus}. Additionally, since $u$ and $\basemountainA$ are \emph{nice}, by \eqref{G5}, \eqref{DomConv}, and \eqref{DualDomConv}, $\x(s)$, $\xbar(t)$ are well-defined and remain in $\outerdom^{\cl}$, $\outertarget^{\cl}$ respectively.
  Now if we write 
\begin{align*}
 \turnedfocuscoord:=\pbarmap{\x_0}{u(\x_0)}{\turnedfocus},\qquad &\pbar_0:=\pbarmap{\x_0}{u(\x_0)}{\basexbarA},
\end{align*}
by  \eqref{eqn: linearization upper bound} combined with \eqref{eqn: bound by plane distance}, for some \emph{very nice} $C>0$ we arrive at the inequality 
\begin{align}
 \G(\parallelpoint,\turnedfocus,\Gconez)-\G(\parallelpoint,\basexbarA,\z(0))&\leq C\gnorm[\basexbarA]{\turnedfocuscoord-\pbar_0}\gbarnorm[\x_0]{\parallelpointcoord-\p_0}\notag\\
 &\leq C\gnorm[\basexbarA]{\turnedfocuscoord-\pbar_0}\dist{\p_0}{\plane{\sectioncoord}{\direction{}}}.\label{eqn: double integral expression}
\end{align}
On the other hand recalling the choice of $\parallelpoint$, and since $u$ and $\mountain$ are \emph{nice}, 
\begin{align*}
\G(\parallelpoint,\turnedfocus,\Gconez)-\G(\parallelpoint,\basexbarA,\z(0))
&\geq C_1(\basemountainA(\x_0)-u(\x_0))+ \G(\parallelpoint, \turnedfocus, \Gconez)-\basemountainA(\parallelpoint)\\
&=C_1(\basemountainA(\x_0)-u(\x_0))
\end{align*}
for some \emph{very nice} $C_1>0$ by applying the mean value property; combining with \eqref{eqn: double integral expression} we thus arrive at \eqref{eqn: bound on direction in Gdual}.

Finally, define
\begin{align*}
  \origvect:=\nondegmatrix{\x_0}{\basexbarA}{\basezA}\raisecovect{\direction{}},\;\;\; \newvect :=(t^{**})^{-1}\left (\turnedfocuscoord-\pbar_0\right ).
\end{align*}
Thanks to Lemma \ref{lem: comparability}, we know that $\pbarmap{\x_0}{u(\x_0)}{\X{\contactpoint}{u(\contactpoint)}{\cdot}}$ is a Lipschitz map on $\cotanspM{\contactpoint}$ with a \emph{very nice} Lipschitz constant. Then we find for some \emph{very nice} $C>0$, 
\begin{align*}
&\gnorm[x_0]{\turnedfocuscoord-\pbar_0}\\
 &=\gnorm[x_0]{\pbarmap{\x_0}{u(\x_0)}{\X{\contactpoint}{u(\contactpoint)}{\pbarmap{\contactpoint}{u(\contactpoint)}{\basexbarA}+t^{**} \nondegmatrix{\contactpoint}{\basexbarA}{\basezA}\raisecovect{\direction{}}}}-\pbarmap{\x_0}{u(\x_0)}{\X{\contactpoint}{u(\contactpoint)}{\pbarmap{\contactpoint}{u(\contactpoint)}{\basexbarA}}}}\\
 &\leq C\gnorm[\contactpoint]{\pbarmap{\contactpoint}{u(\contactpoint)}{\basexbarA}+t^{**} \nondegmatrix{\contactpoint}{\basexbarA}{\basezA}\raisecovect{\direction{}}-\pbarmap{\contactpoint}{u(\contactpoint)}{\basexbarA}}\\
 &=Ct^{**}\gnorm[\contactpoint]{\nondegmatrix{\contactpoint}{\basexbarA}{\basezA}\raisecovect{\direction{}}}\\
 &\leq Ct^{**}
\end{align*}
(here we have also used \eqref{Nondeg} and the fact that $\direction{}$ has unit length). As a result,
\begin{align*}
 \gnorm[x_0]{\newvect}&\leq C.
\end{align*}
Now since $\direction{}$ is unit length, we see $\frac{1}{\gnorm[x_0]{\origvect}}\leq C$ for a \emph{very nice} $C\geq 1$. Now we claim that if $\diam\paren{\sublevelset}$ is smaller than a certain \emph{very nice} constant (recall also Lemma~\ref{lem: comparability}), we can ensure $\gnorm[x_0]{\origvect-\newvect}<\frac{1}{4C\sqrt{n}}$. 

To see this, consider the map
\begin{align*}
  F: \textnormal{Dom}(F) \subset \outerdom^{\cl}\times \outertarget^{\cl}\times T^*M \to T^*M,
\end{align*}
defined by
\begin{align*}
  F(x_1, \xbar, x_2,V) := \pbar_{\x_1,\u(\x_1)}\left (\exp^G_{\x_2,u(x_2)} ( \pbar_{\x_2,u(\x_2)}(\xbar)+V     )  \right) - \pbar_{\x_1,\u(\x_1)}(\xbar),\quad V\in\cotanspM{\x_2}.
\end{align*}
Note that 
\begin{align}
  F(\x_1, \x_2, 0) & = 0, \;\;\forall\;\x_1,\ \x_2,\label{eqn: first F}\\
  F(\x_0, \x_0, \xbar_0, V) & = V, \;\;\forall\; V\in \tanspM{\x_0}.	\label{eqn: second F}
\end{align}	
Moreover $F$ is differentiable with respect to the $(\x_2, V)$ variables along the fibers of $T^*M$, the derivative being continuous in all four variables. This continuity depends only on the modulus of continuity of $\nondegmatrix{\x}{\xbar}{\z}$ (and its inverse), and the modulus of continuity of $u$ which is controlled by a \emph{very nice} constant (see Remark~\ref{rem: nice functions are nice}). Thus all constants obtained below will also be \emph{very nice}. From this point on, we will always take $\x_1=\x_0$ and $\xbar=\xbar_0$ in $F$, thus we will write simply $F(\x, V)$ for brevity.

Now, the map $F$ was constructed so that the point
\begin{align*}
  \X{\contactpoint}{u(\contactpoint)}{\pbarmap{\contactpoint}{u(\contactpoint)}{\basexbarA}+t^{**} \nondegmatrix{\contactpoint}{\basexbarA}{\basezA}\raisecovect{\direction{}}}
\end{align*}
is the same as
\begin{align*}
  \exp_{\x_0,\u(x_0)}^G\left (\pbarmap{\x_0}{u(\x_0)}{\basexbarA}+ F(\x_{\direction{}},t^{**}\nondegmatrix{\contactpoint}{\basexbarA}{\basezA}\raisecovect{\direction{}})\right ).
\end{align*}
Thus, to obtain the desired bound all we have to do is show that
\begin{align*}		
  (t^{**})^{-1} \left ( F(\x_{\direction{}},t^{**}\nondegmatrix{\contactpoint}{\basexbarA}{\basezA}\raisecovect{\direction{}}) - t^{**}\nondegmatrix{\x_0}{\basexbarA}{\basezA}\raisecovect{\direction{}}\right )
\end{align*}
has a small enough length. We distinguish two cases, first, let us assume that $t^{**}\leq t_0$ for some $t_0>0$ to be determined below. Then, differentiating at $V=0$ and recalling \eqref{eqn: first F} leads to
\begin{align*}		
  & (t^{**})^{-1}  F(\x_{\direction{}},t^{**}\nondegmatrix{\contactpoint}{\basexbarA}{\basezA}\raisecovect{\direction{}}) -  \nondegmatrix{\x_0}{\basexbarA}{\basezA}\raisecovect{\direction{}} \\
  & = D_VF(\x_{\direction{}},0)\nondegmatrix{\contactpoint}{\basexbarA}{\basezA}\raisecovect{\direction{}}+o(1)-  \nondegmatrix{\x_0}{\basexbarA}{\basezA}\raisecovect{\direction{}}\\	
  & =  \left ( D_VF(\x_{\direction{}},0) \nondegmatrix{\contactpoint}{\basexbarA}{\basezA}-\nondegmatrix{\x_0}{\basexbarA}{\basezA} \right ) \raisecovect{\direction{}}+o(1).
\end{align*}
Where $o(1)$ represents a vector whose length goes to zero as $t_0 \to 0$. Then, $t_0>0$ is chosen as a \emph{very nice} constant such that
\begin{align*}
  & \gnorm[\x_{\direction{}}]{(t^{**})^{-1}  F(\x_{\direction{}},t^{**}\nondegmatrix{\contactpoint}{\basexbarA}{\basezA}\raisecovect{\direction{}} ) -  \nondegmatrix{\x_0}{\basexbarA}{\basezA}\raisecovect{\direction{}}}\\
  & \leq \gnorm[x_{\direction{}}]{\left ( D_VF(\x_{\direction{}},0)\nondegmatrix{\contactpoint}{\basexbarA}{\basezA}-\nondegmatrix{\x_0}{\basexbarA}{\basezA} \right )\raisecovect{\direction{}} } +\frac{1}{2}\frac{1}{4C\sqrt{n}}.
\end{align*}
On the other hand, since $F$ is continuously differentiable in $V$ we have that 
\begin{align*}
  \x \mapsto D_VF(\x,0)\nondegmatrix{\x}{\basexbarA}{\basezA}-\nondegmatrix{\x_0}{\basexbarA}{\basezA}
\end{align*}
defines a linear transformation-valued continuous map with respect to $\x$. Furthermore, a standard computation shows that this map is zero for $\x=\x_0$. It follows there exists a \emph{very nice} constant $\delta_0$ such that if $\dist{\x_{\direction{}}}{\x_0} \leq \delta_0$  then
\begin{align*}
  \gnorm[x_0]{\origvect-\newvect}<\frac{1}{4C\sqrt{n}}. 
\end{align*}	
This takes care of the case $t^{**}\leq t_0$. Now, suppose that $t^{**}\geq t_0$. Let us choose $\delta$ small enough so that
\begin{align*}
  \gnorm[\x_0]{F(\x_{\direction{}},t^{**}\nondegmatrix{\contactpoint}{\basexbarA}{\basezA}\raisecovect{\direction{}})- t^{**}\nondegmatrix{\x_0}{\basexbarA}{\basezA}\raisecovect{\direction{}}}
\leq \frac{t_0}{4C\sqrt{n}}	
\end{align*}
which is possible thanks to \eqref{eqn: second F} and the \emph{very nice} control on the modulus of continuity of $F$.\\
 
Having the desired bound on $\gnorm[x_0]{\origvect-\newvect}$, we can calculate
\begin{align*}
 \gnorm[x_0]{\frac{\origvect}{\gnorm[x_0]{\origvect}}-\frac{\newvect}{\gnorm[x_0]{\newvect}}}&\leq\frac{1}{\gnorm[x_0]{\origvect}} \paren{\gnorm[x_0]{\origvect-\newvect}+ \gnorm[x_0]{\newvect-\frac{\gnorm[x_0]{\origvect}}{\gnorm[x_0]{\newvect}}\newvect}}\\
 &\leq\frac{1}{\gnorm[x_0]{\origvect}} \paren{\gnorm[x_0]{\origvect-\newvect}+ \norm{\gnorm[x_0]{\origvect}-\gnorm[x_0]{\newvect}}}< \frac{1}{2\sqrt{n}}
\end{align*}
and we arrive at \eqref{eqn: small rotation} as desired.
\end{proof}

Next we apply the above lemma to a specific basis of $n$ directions: let $\direction{1}\in \cotanspMbar{\basexbarA}$ be the unit vector of interest in Theorem~\ref{thm: G-aleksandrov estimate} and $\curly{\direction{i}}_{i=2}^n$ be an orthonormal collection in $\cotanspMbar{\basexbarA}$ aligned with the axial directions of the John ellipsoid of $\sectioncoord$ in such a way that $\innergbar[\xbar]{\direction{1}}{\direction{i}}\leq \frac{1}{\sqrt{n}}$ for every $2\leq i\leq n$.
\begin{lem}\label{lem: S G-dual lower volume bound}
For the above choice of $\curly{\direction{i}}_{i=1}^n$, there exists a \emph{very nice} constant $C>0$ such that
  \begin{align*}
    C\Leb{\Gsubdiff{\Gcone{\x_0, \sublevelset}}{\x_0}} \geq (\basemountainA(\x_0)-u(\x_0))^n \prod \limits_{i=1}^n\frac{1}{\dist{\pmap{\basexbar}{\basez}{\x_0}}{\plane{\sectioncoord}{\direction{i}}}}.
  \end{align*}	  
\end{lem}
\begin{figure}[H]
  \centering
    \includegraphics[height=.45\textwidth]{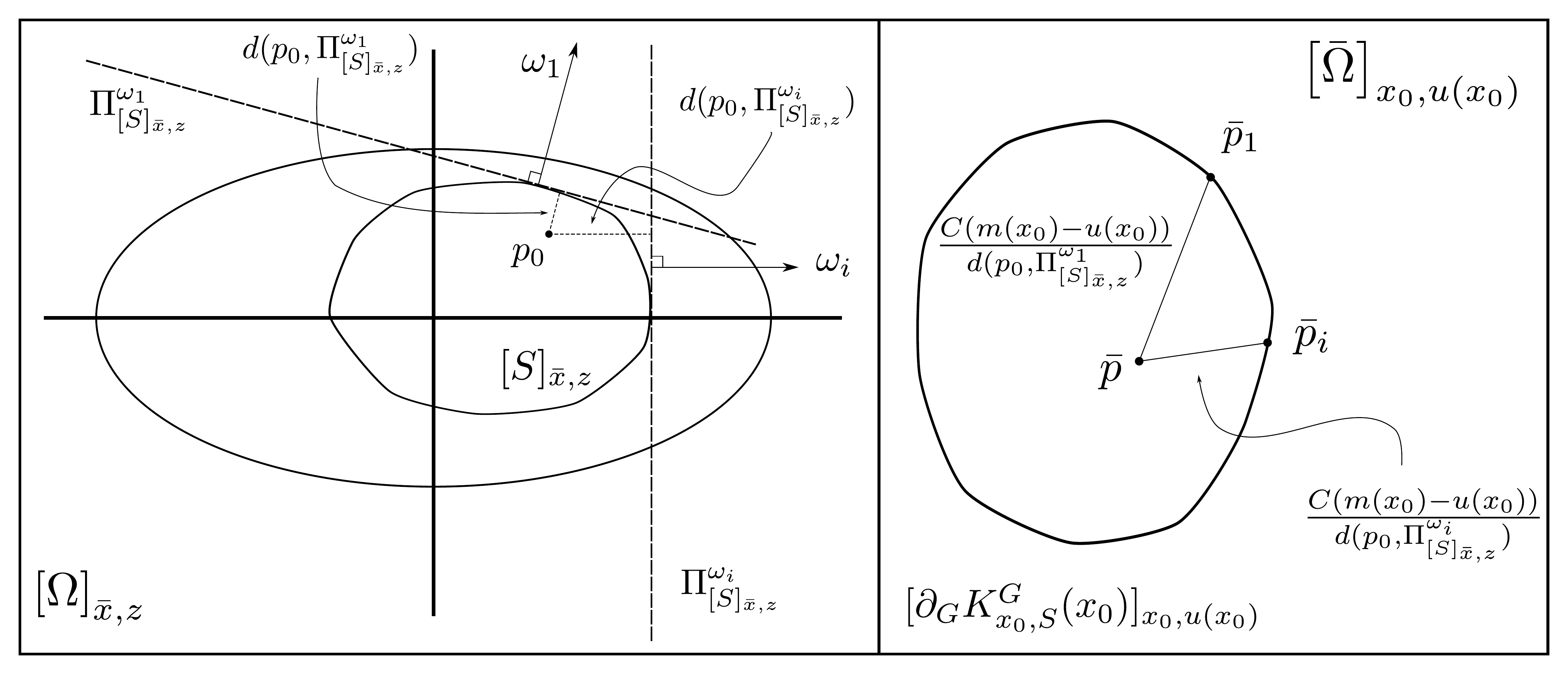}
     \caption{}\label{figure: Sec4fig2}
\end{figure}

\begin{proof}
Let us write $\xbar_i:=\xbar_{\direction{i}}$ which are obtained by applying Lemma~\ref{lem: G-distance bound} to the directions $\direction{i}$, $\pbar_i:=\pbarmap{\x_0}{u(\x_0)}{\xbar_i}$ for $1\leq i\leq n$, and also $\pbar:=\pbarmap{\x_0}{u(\x_0)}{\xbar}$. Then we have $\pbar$, $\pbar_i\in \coord{\Gsubdiff{\Gcone{\x_0, \sublevelset}}{\x_0}}{\x_0, u(\x_0)}$, hence by Remark~\ref{rem: Gcone}
\begin{align*}
\ch\curly{\pbar,\ \pbar_i\mid 1\leq i\leq n}\subset \coord{\Gsubdiff{\Gcone{\x_0, \sublevelset}}{\x_0}}{\x_0, u(\x_0)}.
\end{align*}
Now \eqref{eqn: small rotation} combined with \eqref{Nondeg} implies that the directions $\curly{\pbar_i-\pbar}_{i=1}^n$ span a parallelepiped whose volume is comparable by a \emph{very nice} constant to that of $\ch\curly{\gnorm[\x_0]{\pbar_i-\pbar}\direction{i}\mid 1\leq i\leq n}$, which in turn (due to our assumption on the angles between $\direction{i}$ and $\direction{1}$) has volume comparable to $\prod\limits_{i=1}^n\gnorm[\x_0]{\pbar_i-\pbar}$ by a constant depending only on $n$. Combining this with \eqref{eqn: bound on direction in Gdual} and recalling \eqref{eqn: volume comparability} from Remark~\ref{lem: comparability}, we obtain the claimed inequality.
\end{proof}
Just as in the proof of \cite[Lemma 4.8]{GK14}, and using Corollary~\ref{cor: comparability} (the estimate \eqref{eqn: volume comparability} in particular), we can obtain the following bound.
\begin{lem}\label{lem: volume bound}
There exists a \emph{very nice} constant such that
  \begin{align*}
    C\Leb{\sublevelset}\geq \supsegment{\sectioncoord}{\direction{1}} \prod \limits_{i=2}^n\dist{\plane{\sectioncoord}{\direction{i}}}{\plane{\sectioncoord}{-\direction{i}}}
  \end{align*}	
\end{lem}

It is now straightforward to combine the two last lemmas to obtain the analogue of the Aleksandrov estimate.

\begin{proof}[Proof of Theorem \ref{thm: G-aleksandrov estimate}]

Multiplying the inequalities from Lemmas \ref{lem: S G-dual lower volume bound} and \ref{lem: volume bound} yields,
\begin{align*}
  C\Leb{\Gsubdiff{\Gcone{\x_0, \sublevelset}}{\x_0}} \Leb{\sublevelset} & \geq (\basemountainA(\x_0)-u(\x_0))^n \frac{\supsegment{\sectioncoord}{\direction{1}}}{\dist{\pmap{\basexbar}{\basez}{\x_0}}{\plane{\sectioncoord}{\direction{1}}}}\prod \limits_{i=2}^n \frac{\dist{\plane{\sectioncoord}{\direction{i}}}{\plane{\sectioncoord}{-\direction{i}}}}{\dist{\pmap{\basexbar}{\basez}{\x_0}}{\plane{\sectioncoord}{\direction{i}}}},\\
    & \geq (\basemountainA(\x_0)-u(\x_0))^n \frac{\supsegment{\sectioncoord}{\direction{1}}}{\dist{\pmap{\basexbar}{\basez}{\x_0}}{\plane{\sectioncoord}{\direction{1}}}}.
\end{align*}
Rearranging and applying Lemma~\ref{lem: G-dual inside G-subdifferential image}, the theorem follows.
\end{proof}

\section{The sharp growth estimate}\label{section: sharp growth estimate}
In this section we will work toward proving the estimate Theorem~\ref{thm: Sharp growth}. The strategy of our proof will essentially follow \cite[Section 3]{GK14}, however we must redo \cite[Lemmas 3.8 and 3.10]{GK14} using our conditions \eqref{QQConv} and \eqref{DualQQConv}. Throughout this section, let us fix a $\G$-convex function $u$, $\basemountain$, and $\arbitrary\subset\basesection$ as in the hypotheses of Theorem~\ref{thm: Sharp growth}: namely that $u$ is \emph{very nice}, $\Nicelower<\basemountain<\Niceupper$ on $\outerdom^{\cl}$, $\dilationconst M\coord{\arbitrary}{\basexbar, \basez}\subset \coord{\basesection}{\basexbar, \basez}$ for some \emph{very nice} $\dilationconst>0$, and $\sup_{\arbitrary}{\basemountain}+ \sup_{\arbitrary}{(\basemountain-u)}< \Gfiveupper$ (we also remind the reader that $M$ is the constant in \eqref{QQConv} and \eqref{DualQQConv} associated to the choice $[\QQConvlower, \QQConvupper]=\Niceinterval$).

This first lemma replaces \cite[Lemmas 3.8]{GK14}; but contains a crucial difference. The underlying idea here is that we would like to control $\Leb{\Gsubdiff{u}{\arbitrary}}$ from above by the $\G$-subdifferential of some $\G$-cone at \emph{one point} (which is much better behaved). This amounts to showing that $\G$-affine functions supporting to $u$ also can be vertically shifted to support to a $\G$-cone; as in the Euclidean case, one cannot take the whole section $\basesection$ as the base of this $\G$-cone, a smaller dilate is taken to make sure the $\G$-cone is ``steep enough.'' However, in order to show the inclusion we must rely on \eqref{QQConv}, thus we essentially must consider a $\G$-cone whose vertex lies on $\basemountain$ instead of below it in $\basesection$. Since the dependence of $\G$ on the scalar parameter is nonlinear, this is no longer a vertical translation of the usual $\G$-cone, thus we must instead consider a related $\G$-dual set (compare Definitions~\ref{DEF: G-dual} and~\ref{DEF: G-cone}). We also comment here that we do not require condition \eqref{eqn: section not too deep} in the following proof.
\begin{lem}\label{lem: thin cone}
There is a choice of \emph{very nice} $\dilationconst>0$ for which 
\begin{align*}
 \Gsubdiff{u}{\arbitrary}\subset \Gdual{\arbitrary}{\cm, \basemountain, \supheight}
\end{align*}
where $\supheight:=\sup_{\arbitrary}(\basemountain-u)$ and $\pmap{\basexbar}{\basez}{\cm}$ is the center of mass of $\coord{\basesection}{\basexbar, \basez}$.
\end{lem}
\begin{proof}
Again we comment that by the assumptions on $\basemountain$ combined with \eqref{G5}, we have $(\x, \basexbar, \basez)\in\gendom$ for every $\x\in\outerdom$; in particular $\coord{\basesection}{\basexbar, \basez}$ is well-defined.

 Fix some $\sptpoint\in \arbitrary$ and $\sptxbar\in\Gsubdiff{u}{\sptpoint}$, and let $\sptmountain(\cdot):=\G(\cdot, \sptxbar, \H(\sptpoint, \sptxbar, u(\sptpoint)))$; thus $\sptmountain$ is \emph{nice} and supporting to $u$ from below at $\sptpoint$, and $\Nicelower\leq\sptmountain\leq \Niceupper$ on $\outerdom^{\cl}$. Also let $\sptmountainshift(\cdot):=\G(\cdot, \sptxbar, \H(\cm, \sptxbar, \basemountain(\cm)))$, and let $\maxpoint$ be the point in $\arbitrary^{\cl}$ where the difference $\sptmountainshift-\basemountain$ is maximized. In order to show that $\sptxbar\in \Gdual{\arbitrary}{\cm, \basemountain, \supheight}$, our goal is to show that $\sptmountainshift(\maxpoint)-\basemountain(\maxpoint)\leq \supheight$.
 
 We note here by using the mean value theorem in the scalar parameter, 
\begin{align}
\sptmountainshift(x)-\sptmountain(x)
&=\G(\x, \sptxbar, \H(\cm, \sptxbar, \basemountain(\cm)))-\G(\x, \sptxbar, \H(\cm, \sptxbar, \sptmountain(\cm)))\notag\\
&=\Gz(\x, \sptxbar, \H(\cm, \sptxbar, \mountain_\theta))\Hu(\cm, \sptxbar, \mountain_\theta)(\basemountain(\cm)-\sptmountain(\cm))\label{eqn: sharp growth mean value theorem}
\end{align}
 where $\mountain_\theta:=(1-\theta)\basemountain(\cm) +\theta\sptmountain(\cm)$ for some $\theta\in[0, 1]$. 
By our assumptions $\mountain_\theta\in \Niceinterval$, in particular the product $\Gz\Hu$ in \eqref{eqn: sharp growth mean value theorem} has strictly positive, \emph{very nice} upper and lower bounds, which we write $\supbound$ and $\infbound$. 
Thus we arrive at the inequalities
\begin{align}
 \infbound(\basemountain(\cm)-\sptmountain(\cm))+\sptmountain(x)&\leq\sptmountainshift(x),\label{eqn: mountain lower bound}\\
\supbound(\basemountain(\cm)-\sptmountain(\cm))+\sptmountain(x)&\geq \sptmountainshift(x)\label{eqn: mountain upper bound}
\end{align}
 for any $x$.
 
 Next we can see there exist points $\bdrysptpoint$, $\bdrymaxpoint\in\basesection^\bdry$ so that $\sptpoint$ and $\maxpoint$ lie on $\Gseg{\cm}{\bdrysptpoint}{\basexbar, \basez}$ and $\Gseg{\cm}{\bdrymaxpoint}{\basexbar, \basez}$ respectively (let us write $\sptseg(s):=\Gseg{\cm}{\bdrysptpoint}{\basexbar, \basez}$ and $\maxseg(s):=\Gseg{\cm}{\bdrymaxpoint}{\basexbar, \basez}$). Moreover, since $\dilationconst M\coord{\arbitrary}{\basexbar, \basez}\subset \coord{\basesection}{\basexbar, \basez}$,  there exist $0<\sptparam, \maxparam<\tfrac{1}{\dilationconst M}$  for which $\sptseg(\sptparam)=\sptpoint$ and $\maxseg(\maxparam)=\maxpoint$. By the boundedness assumptions on $\basemountain$ and \eqref{DomConv} both of these $\G$-segments are well-defined, and by Proposition~\ref{prop: sections are convex} lie entirely in $\coord{\basesection}{\basexbar, \basez}$. 
Additionally, the boundedness assumptions on $\basemountain$ allow us to apply \eqref{QQConv} along both of these $\G$-segments as below.
 
By \eqref{QQConv} along $\maxseg(s)$, we obtain
\begin{align}
 \sptmountainshift(\maxpoint)-\basemountain(\maxpoint)&\leq M\maxparam[\sptmountainshift(\bdrymaxpoint)-\basemountain(\bdrymaxpoint)]_+\notag\\
 &\leq \dilationconst^{-1}[\sptmountainshift(\bdrymaxpoint)-\basemountain(\bdrymaxpoint)]_+.\label{eqn: QQConv in first case}
\end{align}
At this point let us take 
\begin{align*}
 \dilationconst:=\frac{2\supbound}{\infbound},
\end{align*}
which again is \emph{very nice}; we then consider a number of cases.
 
Case 1: If $\sptmountainshift(\bdrymaxpoint)\leq\basemountain(\bdrymaxpoint)$, then the above inequality already implies $ \sptmountainshift(\maxpoint)-\basemountain(\maxpoint)\leq 0\leq \supheight$ and we are finished.

Case 2: Otherwise we can take $x=\bdrymaxpoint$ in \eqref{eqn: mountain upper bound} and combine with \eqref{eqn: QQConv in first case} to obtain
\begin{align}
 \sptmountainshift(\maxpoint)-\basemountain(\maxpoint)&\leq \dilationconst^{-1}[\supbound(\basemountain(\cm)-\sptmountain(\cm))+\sptmountain(\bdrymaxpoint)-\basemountain(\bdrymaxpoint)]\notag\\
 &\leq \dilationconst^{-1}\supbound(\basemountain(\cm)-\sptmountain(\cm))\notag\\
 &=\frac{\infbound}{2}(\basemountain(\cm)-\sptmountain(\cm)),\label{eqn: case 2 bound}
\end{align}
the second inequality is due to the fact that $\sptmountain(\bdrymaxpoint)\leq u(\bdrymaxpoint)\leq \basemountain(\bdrymaxpoint)$ since $\bdrymaxpoint\in\basesection$.

At this point we can apply \eqref{QQConv} along $\sptseg(s)$ to obtain as above,
\begin{align*}
\sptmountainshift(\sptpoint)-\basemountain(\sptpoint)&\leq M\sptparam[\sptmountainshift(\bdrysptpoint)-\basemountain(\bdrysptpoint)]_+\notag\\
 &\leq \dilationconst^{-1}[\sptmountainshift(\bdrysptpoint)-\basemountain(\bdrysptpoint)]_+.
 \end{align*}
 Combining this time with $x=\sptpoint$ in \eqref{eqn: mountain lower bound}, we see that
\begin{align}
 &\dilationconst^{-1}[\sptmountainshift(\bdrysptpoint)-\basemountain(\bdrysptpoint)]_+\notag\\
 &\geq \infbound(\basemountain(\cm)-\sptmountain(\cm))+\sptmountain(\sptpoint)-\basemountain(\sptpoint)\notag\\
 &=\infbound(\basemountain(\cm)-\sptmountain(\cm))-(\basemountain(\sptpoint)-u(\sptpoint))\notag\\
 &\geq \infbound(\basemountain(\cm)-\sptmountain(\cm))-\supheight.\label{eqn: QQConv in second case}
\end{align}

Case 2a: If $\sptmountainshift(\bdrysptpoint)\leq\basemountain(\bdrysptpoint)$, by rearranging the above we see $\basemountain(\cm)-\sptmountain(\cm)\leq \infbound^{-1}\supheight$, which 
combined with \eqref{eqn: case 2 bound} yields
\begin{align*}
\sptmountainshift(\maxpoint)-\basemountain(\maxpoint)&\leq \frac{\infbound}{2}\infbound^{-1}\supheight\leq \supheight
\end{align*}
as desired.

Case 2b: Otherwise in the final case, we once again apply \eqref{eqn: mountain upper bound} with $x=\bdrysptpoint$ and combine with \eqref{eqn: QQConv in second case} to obtain (using that $\bdrysptpoint\in\basesection$ as well),
\begin{align*}
\infbound(\basemountain(\cm)-\sptmountain(\cm))&\leq \supheight+\dilationconst^{-1} \supbound(\basemountain(\cm)-\sptmountain(\cm))\\
&=\supheight+\frac{\infbound}{2}(\basemountain(\cm)-\sptmountain(\cm))
\end{align*}
or rearranging,
\begin{align*}
\frac{\infbound}{2}(\basemountain(\cm)-\sptmountain(\cm))\leq \supheight
\end{align*}
Clearly combining this bound with \eqref{eqn: case 2 bound} gives $\sptmountainshift(\maxpoint)-\basemountain(\maxpoint)\leq \supheight$, finishing the proof.
\end{proof}

With the above Lemma~\ref{lem: thin cone} and Lemma~\ref{lem: linearization} in hand we can connect the $\G$-dual set with the usual polar dual from convex geometry (defined below), in the appropriate coordinates defined via \eqref{Twist}; this easily leads to our claimed estimate in Theorem~\ref{thm: Sharp growth}.

\begin{DEF}\label{DEF: polar dual}
  Let $V$ be an linear space, $\arbitrary\subset V$, $\p_0 \in \arbitrary^{\interior}$, $\q_0\in V^*$, and $\lambda>0$. The \emph{polar dual of $\arbitrary$ of scale $\lambda$, center $\p_0$, and base $\q_0$}, denoted $\polardual{\arbitrary}{\p_0, \q_0, \lambda}\subset V^*$, is the set given by
  \begin{align*}
    \polardual{\arbitrary}{\p_0, \q_0, \lambda}\defin\curly{\q\in V^*\mid\inner{\q-\q_0}{\p-\p_0}\leq \lambda,\ \forall\; p\in \arbitrary}.
  \end{align*}
\end{DEF}

\begin{lem}\label{lem: G-dual in polar dual}
There exists a \emph{very nice} $C>0$ such that
\begin{align*}
\Leb{ \Gdual{\arbitrary}{\cm, \basemountain, \supheight}}\leq C\Leb{\arbitrary}^{-1}\supheight^n.
\end{align*}
\end{lem}
\begin{proof}
 Fix $\Gdualybar\in\Gdual{\arbitrary}{\cm, \basemountain, \supheight}$ and $\x\in\arbitrary$; recall that $\basemountain(\cdot)=\G(\cdot, \basexbar, \basez)$. Note by \eqref{G5} and \eqref{Nondeg}, we can see that $\nondegmatrixinv{\cm}{\basexbar}{\basez}$ is well-defined. We claim that 
\begin{align}\label{eqn: sharp growth claim inequality 1}
 \inner{\pmap{\basexbar}{\basez}{\x}-\pcm}{\nondegmatrixinv{\cm}{\basexbar}{\basez}\pbarmap{\cm}{\basemountain(\cm)}{\Gdualybar}-\baseq}\leq \polarheightconst \supheight
\end{align}
for some \emph{very nice} $\polarheightconst>0$, where 
\begin{align*}
\pcm:&=\pmap{\basexbar}{\basez}{\cm},\\
\baseq:&=\nondegmatrixinv{\cm}{\basexbar}{\basez}\pbarmap{\cm}{\basemountain(\cm)}{\basexbar}.
\end{align*}

  First fix an $\x\in\arbitrary$, let $\xbar(t):=\Gseg{\basexbar}{\Gdualybar}{\cm, \basemountain(\cm)}$ and write $\z(t):=\H(\cm, \xbar(t), \basemountain(\cm))$; since $\basemountain(\cm)\in \Niceinterval$, by \eqref{DualDomConv} and \eqref{G5} we see $\xbar(t)$ is well-defined and remains in $\outertarget^{\cl}$.
  Now, we can assume 
  \begin{align*}
  \inner{\pmap{\basexbar}{\basez}{\x}-\pcm}{\nondegmatrixinv{\cm}{\basexbar}{\basez}\pbarmap{\cm}{\basemountain(\cm)}{\Gdualybar}-\baseq}>0,
  \end{align*}
   otherwise \eqref{eqn: sharp growth claim inequality 1} is immediate. First, it is clear that $\basexbar\in\Gdual{\arbitrary}{\cm, \basemountain, \supheight}$ so the second claim in Proposition~\ref{prop: local to global} implies that $\Gseg{\xbar}{\ybar}{\cm, \basemountain(\cm)}\subset\Gdual{\arbitrary}{\cm, \basemountain, \supheight}$ (recall by our assumption \eqref{eqn: section not too deep}, we have $\displaystyle\sup_{\arbitrary}\basemountain+\supheight<\Gfiveupper$). Combining with \eqref{eqn: section not too deep}, for any $t\in[0, 1]$ we must have
  \begin{align*}
  \G(\x, \xbar(t), \z(t))\leq \basemountain(\x)+\supheight<\Gfiveupper.
  \end{align*}
Next let $[0, \firsttime]\subset [0, 1]$ be the maximal subinterval ($\firsttime$ necessarily strictly positive) on which $\G(\x, \xbar(t), \z(t))\geq\Nicelower$. By \eqref{G5}, we can apply \eqref{DualQQConv} along $\xbar(t)$ on $[0, \firsttime]$ (after reparametrizing) to see that $\G(\x, \xbar(t), \z(t))$ cannot have any strict local maxima in $(0, \firsttime)$. The calculation in Lemma~\ref{lem: linearization} shows $\vertbar{\diff{t}\G(\x, \xbar(t), \z(t))}_{t=0}>0$ by our assumption, thus we must actually have $t_0=1$. As a result
  \begin{align*}
  \Nicelower\leq \G(\x, \xbar(t), \z(t))<\Gfiveupper,
  \end{align*} 
or by \eqref{G5}, $(\x, \xbar(t), \z(t))\in\gendom$ for all $t\in[0, 1]$. 
 We can thus apply \eqref{eqn: linearization lower bound} from Lemma~\ref{lem: linearization} to obtain for a \emph{very nice} $C_0>0$, 
\begin{align*}
 &\inner{\pmap{\basexbar}{\basez}{\x}-\pcm}{\nondegmatrixinv{\cm}{\basexbar}{\basez}\pbarmap{\cm}{\basemountain(\cm)}{\Gdualybar}-\baseq}\\
   &\qquad\leq C_0M \brackets{\G(\x, \xbar(1), \z(1))-\G(\x, \xbar(0), \z(0))}_+\\
&\qquad =C_0M \brackets{\G(\x, \Gdualybar, \H(\cm, \Gdualybar, \basemountain(\cm)))-\basemountain(\x)}\\
&\qquad\leq C_0 M \supheight,
\end{align*}
and we obtain \eqref{eqn: sharp growth claim inequality 1} with $\polarheightconst:=C_0M$.
As a result we see this implies that
\begin{align*}
\nondegmatrixinv{\cm}{\basexbar}{\basez} \coord{\Gdual{\arbitrary}{\cm, \basemountain, \supheight}}{\cm, \basemountain(\cm)}\subset \polardual{\paren{\coord{\arbitrary}{\basexbar, \basez}}}{\pcm, \baseq, \polarheightconst\supheight},
\end{align*}
thus by taking the volume of both sides (recall also Lemma~\ref{lem: comparability} and Corollary~\ref{cor: comparability}) and combining with \cite[Lemma 3.9]{GK14} (note we do not need $\coord{\arbitrary}{\basexbar, \basez}$ to be convex, as we can apply the result to the convex hull of $\coord{\arbitrary}{\basexbar, \basez}$ to obtain the same inequality, using that the polar dual of a set is unchanged by taking its convex hull) we obtain the lemma for another choice of \emph{very nice} $C>0$.
\end{proof}

\begin{proof}[Proof of Theorem~\ref{thm: Sharp growth}]
 Combining the above Lemmas~\ref{lem: thin cone} and~\ref{lem: G-dual in polar dual}, the theorem is immediate.
\end{proof}

\section{Localization and strict $\G$-convexity}\label{section: localization}

  In this section and the next one we shall use the estimates from Theorems \ref{thm: G-aleksandrov estimate} and \ref{thm: Sharp growth} to prove the strict $\G$-convexity of a $\G$-convex solution to a $\G$-Monge-Amp\'ere equation with a nondegenerate $\G$-Monge-Amp\'ere measure (recall Definition \ref{def:Aleksandrov_solutions}).

   It is assumed that the support of the $\G$-Monge-Amp\'ere measure (denoted $\innerdom^{\cl}$) lies in the interior of $\outerdom$. Moreover, it is assumed that $\innertarget := \Gsubdiff{u}{\innerdom}$ is such that $\innertarget^{\cl}\subset \outertarget^{\interior}$ and is $\G$-convex with respect to $(\x,\u(x))$ for all $x\in \innerdom$.
   
  \begin{REM} One expects the strict $\G$-convexity to also hold in a situation where $\outerdom$ and $\outertarget$ are strictly $\G$-convex, as oppose to assuming that the closure of $\innerdom$ is contained in the interior of $\outerdom$. This is for instance what is done in the work of Figalli, Kim, and McCann \cite{FKM13} in the case of optimal transport.
  \end{REM}

	Moreover, we will assume for the rest of the section that $u$ is a \emph{very nice} $\G$-convex function. Recall that this assumption is needed, even if the data is smooth (as discussed in Section~\ref{section: literature reflector problems}). Thus, for the rest of Section \ref{section: localization} we will fix
	\begin{align}\label{eqn:Localization_u_is_very_nice}
	  u:\outerdom \to \mathbb{R},\;\;\textnormal{\emph{very nice}, and solving \eqref{eqn: generated Jacobian equation} in the sense of Definition \ref{def:Aleksandrov_solutions} for some } \Lambda>0.
	\end{align}
	
	The first of our theorems in this section says that ``singularities'' (in the sense of failure of strict $\G$-convexity), if they happen at all, must propagate all the way to the boundary of $\outerdom$.
    \begin{thm}\label{thm: localization}
         Let $u$ be as in \eqref{eqn:Localization_u_is_very_nice}. If $\xbar_0 \in \innertarget$ and $\z_0$  are such that $\mountain_0(\cdot) = \G(\cdot,\xbar_0,\z_0)$ is supporting to $u$ at some $x_0\in\innerdom^{\interior}$, then the set 
         \begin{equation*}
              \sublevelset_0:=\{u=\mountain_0 \}
         \end{equation*}
         is a single point, or else every extremal point of $\coord{\sublevelset_0}{\xbar_0,z_0}$ is contained on the boundary of $\coord{\outerdom}{\xbar_0,z_0}$.
    \end{thm}
Using this result we will prove Theorem~\ref{thm:strict_convexity} later in the section.

    \subsection{Some elementary tools}  Let us review some notions from convex geometry (see for example,~\cite{Roc70}) and linear algebra.

    \begin{DEF}\label{def: normal cones}
         Suppose that $\mathcal{\arbitrary}$ is a convex subset of $\cotanspMbar{\xbar}$ and $p_e\in  \mathcal{\arbitrary}^\bdry$. Then, the \emph{strict normal cone of $\mathcal{\arbitrary}$ at $p_e$} and \emph{normal cone of $\mathcal{\arbitrary}$ at $p_e$} are defined by
         \begin{align*}
              \normal^0_{p_e}\left(\mathcal{\arbitrary}\right):&=\{q\in \cotanspMbar{\xbar}\mid \innergbar{q}{p-p_e} <0,\ \forall p\neq p_e\in \mathcal{\arbitrary}\},\\
              \normal_{p_e}\left(\mathcal{\arbitrary}\right):&=\{q\in \cotanspMbar{\xbar}\mid \innergbar{q}{p-p_e} \leq 0,\ \forall p\in \mathcal{\arbitrary}\}.
         \end{align*}
         If $\normal^0_{p_e}\left(\mathcal{\arbitrary}\right)$ is nonempty, $p_e$ is called an \emph{exposed point} of $\mathcal{\arbitrary}$.
    \end{DEF}

    \begin{rem}\label{rem: remark on normal cones}
         It is well known, $\normal_{p_e}\left(\mathcal{\arbitrary}\right)$ and $\normal^0_{p_e}\left(\mathcal{\arbitrary}\right)$ are convex cones. Also $\normal_{p_e}\left(\mathcal{\arbitrary}\right)$ is closed, and contains $0$ and at least one nonzero vector for any $p_e\in\mathcal{\arbitrary}^\bdry$.
    \end{rem}

    \subsection{Tilting and chopping.} The proof of Theorem \ref{thm: localization} goes by a contradiction. If $\sublevelset_0$ has more than one point and also contains an interior exposed point (when seen in cotangent coordinates), then one may find sections $\sublevelset_t := \{u\leq \mountain_t\}$ ($t$ small and positive) with a geometry that contradicts the combined estimates from Theorem \ref{thm: G-aleksandrov estimate} and Theorem \ref{thm: Sharp growth}. The sections $\sublevelset_t$ will be obtained by adequately ``chopping'' the original contact set $\sublevelset_0$ with a family of $\G$-affine functions $\mountain_t$ which are obtained by ``tilting'' the original function $\mountain_0$.

  The next two lemmas deal with the selection of the family of $\G$-affine functions $\mountain_t$. We do not yet need the fact that $u$ is an Aleksandrov solution here, just the fact that it is \emph{very nice}.
  \begin{lem}\label{lem: localization claim}
    Let $\mountain_0(\cdot):=\G(\cdot, \xbar_0, \z_0)$ be a $\G$-affine function supporting to $u$ somewhere in $\outerdom$ with $\xbar_0\in \innertarget$, and define 
    \begin{align*}
      \sublevelset_0:&=\curly{u=\mountain_0}.
    \end{align*}
    Also, suppose $\p_e$ is an exposed point of $\subzerocoord$, that $\e{0} \in \normal^0_{\p_e}\paren{\subzerocoord}$ is unit length, and $\subzero$ contains at least two points. 
    Then for any fixed $\perturb>0$ there exists a family of \emph{nice} $\G$-affine functions $\curly{\mountain^\perturb_t}_{t>0}$, (depending on $\sublevelset_0$ and $\e{0}$), such that for all small enough $t>0$ we have
    \begin{align}
      \mountain_0(\x_e)=u(\x_e)&<\mountain^\perturb_t(\x_e),\label{eqn: exposed point is inside chopped section}\\
      \coord{\sublevelset_{\perturb, t}}{\xbar_0, \z_0}&\subset B_{\perturb}(\p_e),\label{eqn: chopped sections close to exposed point}\\
      \Nicelower<&\mountain_t(x)<\Niceupper,\quad \forall\;x\in \outerdom.\label{eqn: delta tiltings are nice}
    \end{align}
    where $\sublevelset_{\perturb, t}:=\{u\leq \mountain^\perturb_t\}$.
  \end{lem}

    \begin{proof}
         Let us write 
         \begin{align*}
              \x_e:=\X{\xbar_0}{\z_0}{\p_e},\quad \pbar_0:=\pbarmap{\x_e}{\u(\x_e)}{\xbar_0},
         \end{align*}
         and note that since $\mountain_0$ is supporting to $u$ at $\x_e$ we have
\begin{align*}
 \z_0=\H(\x_e, \xbar_0, u(\x_e)).
\end{align*}
         We will now define $\mountain^\perturb_t$. Note that by \eqref{Nondeg}, we have $\nondegmatrix{\x_e}{\xbar_0}{\z_0}\raisecovect{\e{0}}\neq 0$ (see Definition \ref{def:musical_isomorphism} for the definition of $\raisecovect{\e{0}}$). 
         Since $\innertarget^{\cl}\subset\outertarget^{\interior}$, for $t>0$ sufficiently small, $\pbar_0+t\nondegmatrix{\x_e}{\xbar_0}{\z_0}\raisecovect{\e{0}}$ remains in $\outertargetcoord{\x_e, u(\x_e)}$, hence 
         \begin{equation*}
              \xbar(t):=\Xbar{\x_e}{u(\x_e)}{\pbar_0+t\nondegmatrix{\x_e}{\xbar_0}{\z_0}\raisecovect{\e{0}}}
         \end{equation*}
         is a well-defined $\G$-segment for such $t$ (we comment here that the smallness of $t$ does not need to be \emph{very nice}, in fact it is allowed to depend on $\x_e$, $\sublevelset_0$, and $\e{0}$, and we may have need to take it smaller later in this proof). Also define
\begin{align*}
 \z(t):&=\H(\x_e, \xbar(t), u(\x_e))=\Z{\x_e}{\pbar_0+t\nondegmatrix{\x_e}{\xbar_0}{\z_0}\raisecovect{\e{0}}}{u(\x_e)},\\
 \z_{\perturb}(t):&=\H(\x_e, \xbar(t), u(\x_e)+\perturbparam t)=\Z{\x_e}{\pbar_0+t\nondegmatrix{\x_e}{\xbar_0}{\z_0}\raisecovect{\e{0}}}{u(\x_e)+\perturbparam t},
\end{align*}
for some small $\perturbparam>0$ to be chosen later.
         Now we consider the $\G$-affine functions
         \begin{align*}
              \mountain^\perturb_t(\x):=\G(\x, \xbar(t), \z_{\perturb}(t)),
         \end{align*}
         note \eqref{eqn: exposed point is inside chopped section} follows immediately. 
         First, a simple compactness argument yields 
\begin{align}\label{eqn: close to zero section}
 \coord{\sublevelset_{\perturb, t}}{\xbar_0, \z_0}\subset\nbhdof{r(t)}{\coord{\sublevelset_0}{\xbar_0, \z_0}}
\end{align}
for some $r(t)=o(1)$ as $t\to 0$ (with $\perturbparam$ fixed), while again a compactness argument along with the inclusion $\e{0}\in \normal^0_{\p_e}\paren{\subzerocoord}$ gives existence of a $\perturbparamtilde>0$ such that
\begin{align}\label{eqn: small chopping zero section}
 \coord{\sublevelset_{0}}{\xbar_0, \z_0}\cap \curly{p\in\cotanspMbar{\xbar_0}\mid\innergbar{p-p_e}{\e{0}}\geq -\perturbparamtilde} \subset B_{\perturb/2}(\p_e).
\end{align}
Next since $u$ is \emph{very nice}, we see that if $\perturbparam$ is sufficiently small, then $[u(\x_e), u(x_e)+\perturbparam t]\subset \Niceinterval$. Thus by using the mean value property as in the calculation of \eqref{eqn: sharp growth mean value theorem} there exists a \emph{very nice} $C_1>0$ such that 
\begin{align}\label{eqn: first tilting close to tilting}
 \lvert\G(\x, \xbar(t), \z_{\perturb}(t))-  \G(\x, \xbar(t), \z(t))\rvert \leq C_1 \perturbparam t,
\end{align}
and in turn if $\mountain_0(\x)\leq\mountain^\perturb_t(\x)$ we have
         \begin{align*}
         0&\leq \G(\x, \xbar(t), \z_{\perturb}(t))-\G(\x, \xbar_0, \z_0)\\
         &\leq C_1\perturbparam t +\G(\x, \xbar(t), \z(t))-\G(\x, \xbar_0, \z_0).
         \end{align*}
         At the same time, since $\xbar(s)$ remains entirely in $\outertarget$ by \eqref{DualDomConv} and $u$ is \emph{very nice}, the quantity $-\Gz(\x, \xbar(s), \z(s))$ is bounded away from zero and infinity by a \emph{very nice} constant. Thus for some \emph{very nice} $C>0$, (also using \eqref{eqn: xbardot representation}, \eqref{eqn: zdot representation})
\begin{align}
 -C_1\perturbparam 
 &\leq t^{-1}\int_0^t\diff{s}\G(\x, \xbar(s), \z(s))ds\notag\\
 &=t^{-1}\int_0^t (-\Gz(\x, \xbar(s), \z(s)))\inner{-\frac{\Dbar\G}{\Gz}(\x, \xbar(s), \z(s))+\frac{\Dbar\G}{\Gz}(\x_e, \xbar(s), \z(s))}{\xbardot(s)}ds\notag\\
 &\leq Ct^{-1}\int_0^t\inner{\pmap{\xbar(s)}{\z(s)}{\x}-\pmap{\xbar(s)}{\z(s)}{\x_e}}{\nondegmatrixinv{\x_e}{\xbar(s)}{\z(s)}\nondegmatrix{\x_e}{\xbar_0}{\z_0}\raisecovect{\e{0}}}ds\notag\\
 &= C\inner{\pmap{\xbar(s')}{\z(s')}{\x}-\pmap{\xbar(s')}{\z(s')}{\x_e}}{\nondegmatrixinv{\x_e}{\xbar(s')}{\z(s')}\nondegmatrix{\x_e}{\xbar_0}{\z_0}\raisecovect{\e{0}}},\label{eqn: halfspace containment}
\end{align}
for some $s'\in [0, t]$. We pause to note here, $s'$ is determined by $t$, thus this last expression can be viewed as a family of functions in the variable $\x\in \outerdom^{\cl}$, parametrized by $t\geq 0$. By the $C^2$ assumption we have on $\G$, as $t$ approaches $0$ the expression converges uniformly in $\x\in \outerdom^{\cl}$ to the quantity
\begin{align*}
 \inner{\pmap{\xbar_0}{\z_0}{\x}-\pmap{\xbar_0}{\z_0}{\x_e}}{\nondegmatrixinv{\x_e}{\xbar_0}{\z_0}\nondegmatrix{\x_e}{\xbar_0}{\z_0}\raisecovect{\e{0}}}=\innergbar{\pmap{\xbar_0}{\z_0}{\x}-\p_e}{\e{0}}.
\end{align*}
As a result, first taking $\perturbparam$ small, we have for all $t>0$ small, the inclusion
\begin{align*}
 \coord{\curly{\mountain_0\leq \mountain^t_{\perturb}}}{\xbar_0, \z_0}\subset \curly{p\in\cotanspMbar{\xbar_0}\mid\innergbar{p-p_e}{\e{0}}\geq -\frac{\perturbparamtilde}{2}}.
\end{align*}
Since $\mountain_0$ is supporting to $u$ we have
\begin{align*}
 \sublevelset_{\perturb, t}\subset\curly{\mountain_0\leq \mountain^t_{\perturb}},
\end{align*}
combining with \eqref{eqn: close to zero section} and \eqref{eqn: small chopping zero section} we obtain \eqref{eqn: chopped sections close to exposed point}.

Finally, this last argument of uniform convergence shows that if $t$ is taken sufficiently small, 
\begin{align}\label{eqn: uniform closeness of first tilting}
\sup_{\y\in\outerdom}\norm{\G(\y, \xbar(t), \z(t))-\mountain_0(\y)}\text{ is small},
\end{align}
 hence combined with \eqref{eqn: first tilting close to tilting} and the fact that $u$ is \emph{very nice}, we can ensure \eqref{eqn: delta tiltings are nice} holds.
\end{proof}

	In the next lemma the notation $\plane{\subcoord}{\pm \w{}}$ is again used (see Definition \ref{DEF:notation_supporting_plane}).	

    \begin{lem}\label{lem: chopping convergence}
         Let $\mountain_0$, $\sublevelset_0$, $\p_e$, and $\x_e$ be as in Lemma~\ref{lem: localization claim} above, and suppose $\x_e\in\outerdom^{\interior}$. Then, for each $\perturb>0$, we can find (each of the following depending on $\sublevelset_0$), some $t_0>0$, a family of $\G$-affine functions $m_t(\cdot):=\G(\cdot, \xbar_t, \z_t)$, a family of unit length $\e{t}\in\cotanspMbar{\xbar_t}$ defined for $t\in[0, t_0]$, and some $\epsilon_0>0$ which satisfy the following for all $t\in [0, t_0]$:
         \begin{align}
	          \mountain_t(\x_e)>u(\x_e),\quad&\lim_{t\searrow 0}\mountain_t(\x_e)= u(\x_e),\label{eqn: chopped sections have interior}\\
	          \coord{\sublevelset_{t}}{\xbar_0, \z_0}&\subset B_{\perturb}(\p_e),\label{eqn: chopped sections close to exposed point 2}\\
	          \Nicelower<\mountain_t(x)&<\Niceupper,\quad \forall\;x\in \outerdom,\label{eqn: tiltings are nice}\\
              \min \curly{\frac{\mountain_t(\x_e)-u(\x_e)}{\sup \limits_{\sublevelset_t} (\mountain_t-u)}\;,\;\frac{\mountain_t(\x_e)-u(\x_e)}{\sup \limits_{\sublevelset_t} (\mountain_t-u)+u(\x)-\mountain_t(\x)}} &\geq \epsilon_0,\qquad \forall x\in\subzero\setminus\bigcup_{t\in(0, t_0]}\sublevelset_t\label{eqn: lower bound on localization ratios}\\
              \lim \limits_{t\to 0^+} \frac{\dist{\pmap{\xbar_t}{\z_t}{\x_e}}{\plane{\coord{\sublevelset_t}{\xbar_t, \z_t}}{\e{t}}  \cup \plane{\coord{\sublevelset_t}{\xbar_t, \z_t}}{-\e{t}} }}{\supsegment{\coord{\sublevelset_t}{\xbar_t, \z_t}}{\e{t}}}  &= 0,\label{eqn: supporting plane collapses}.
         \end{align}
         Here we have written $\sublevelset_t:=\{u\leq \mountain_t\}$. 
    \end{lem}

    \begin{proof}
         Fix $\delta>0$. By \cite[Lemma 7.4]{GK14}, there exists a unit length $\e{0}\in \normal^0_{\p_e}\paren{\subzerocoord}$ and $\lambda_0>0$ such that $\p_e-\lambda \e{0}\in \subzerocoord$ for all $\lambda \in [0, \lambda_0]$. Let $\mountain_t:=\mountain^\delta_t$ be obtained by applying Lemma~\ref{lem: localization claim} with this choice of $\e{0}$ and $\delta$, and we may assume both that $t_0\leq \lambda_0$ and $t_0$ is small enough to obtain all the properties detailed in Lemma~\ref{lem: localization claim} when $t\leq t_0$. Also let $\xbar_t:=\xbar(t)$ and $\z_t:=\z_{\perturb}(t)$ as defined in the proof of Lemma~\ref{lem: localization claim} above. Then \eqref{eqn: exposed point is inside chopped section} immediately implies \eqref{eqn: chopped sections have interior}, \eqref{eqn: chopped sections close to exposed point} implies \eqref{eqn: chopped sections close to exposed point 2}, \eqref{eqn: delta tiltings are nice} implies \eqref{eqn: tiltings are nice}, and each $\mountain_t$ is \emph{nice}.
               
         Now we will show \eqref{eqn: lower bound on localization ratios}. First, by \eqref{eqn: first tilting close to tilting} and a calculation similar to \eqref{eqn: halfspace containment}, Cauchy-Schwarz, \eqref{Nondeg}, and Lemma \ref{lem: comparability} we find a \emph{very nice} $C>0$ for which
\begin{align*}
 \sup\limits_{\outerdom }(\mountain_t -\mountain_0)&=\sup\limits_{y\in\outerdom }\brackets{\G(y, \xbar_t, \z_t)-\G(y, \xbar_0, \z_0)}\\
 &\leq Ct(1+\perturbparam).
\end{align*}
          Recalling that $\mountain_0\leq u$, we obtain
         \begin{align*}
              \frac{\mountain_t (\x_e)-u(\x_e)}{\sup\limits_{\sublevelset_t }{(\mountain_t -u)}}&\geq \frac{\mountain_t (\x_e)-u(\x_e)}{\sup\limits_{\sublevelset_t }(\mountain_t -\mountain_0)}\\
&\geq \frac{\perturbparam t}{Ct(1+\perturbparam)}\\
              &\geq \frac{\perturbparam }{C(1+\perturbparam)}.
         \end{align*}
Next note that since $x\not\in\bigcup_{t\in(0, t_0]}\sublevelset_t$ the denominator of the second expression in the minimum in \eqref{eqn: lower bound on localization ratios} is always strictly positive. Then since $x\in\sublevelset_0$ we have
\begin{align*}
\sup\limits_{\sublevelset_t }{(\mountain_t -u)}+u(x)-\mountain_t(x)\leq \sup\limits_{\sublevelset_t }(\mountain_t -\mountain_0)+\mountain_0(x)-\mountain_t(x),
\end{align*}
 by an argument much as above we obtain \eqref{eqn: lower bound on localization ratios} for the choice
         \begin{align*}
         	  \epsilon_0=\frac{\perturbparam }{2C(1+\perturbparam)}.
         \end{align*}
        We now work toward showing \eqref{eqn: supporting plane collapses}, to this end take any $\x_{cp}\in\subzero$. Recalling \eqref{eqn: uniform closeness of first tilting}, we can apply \eqref{eqn: linearization lower bound} in Lemma~\ref{lem: linearization} and use the mean value theorem as in \eqref{eqn: first tilting close to tilting} to find a \emph{very nice} $C>0$ such that
\begin{align*}
\mountain_t(\x_{cp})-u(\x_{cp})
 &=\mountain_t(\x_{cp})-\mountain_0(\x_{cp})\\
 &= \G(\x_{cp}, \xbar(t), \z_{\perturb}(t))-\G(\x_{cp}, \xbar_0, \z_0)\\
 &\geq \G(\x_{cp}, \xbar(t), \z(t))-\G(\x_{cp}, \xbar_0, \z_0)+C\perturbparam t\\
 &\geq  \frac{C}{M}\inner{\pmap{\xbar_0}{\z_0}{\x_{cp}}-\p_e}{\nondegmatrixinv{\x_e}{\xbar_0}{\z_0}(\pbarmap{\x_e}{u(\x_e)}{\xbar(t)}-\pbar_0)}+C\perturbparam t\\
 &=t\frac{ C}{M}\inner{\pmap{\xbar_0}{\z_0}{\x_{cp}}-\p_e}{\raisecovect{\e{0}}}+C\perturbparam t\\
&= t\frac{ C}{M}\innergbar[\xbar_0]{\pmap{\xbar_0}{\z_0}{\x_{cp}}-\p_e}{\e{0}}+C\perturbparam t.
\end{align*}
Since $\subzero$ contains at least one point besides $\x_e$ and $\subzerocoord$ is convex by Proposition~\ref{prop: sections are convex} (recall $\mountain_0$ is assumed \emph{nice}), we may choose $\x_{cp}\in\subzero$, $\x_{cp}\neq\x_e$ in such a way that the final expression in the above calculation is always nonnegative. In particular
\begin{align}\label{eqn: all tiltings contain another point}
 \x_{cp}\in\sublevelset_t, \quad t\in[0, t_0].
\end{align}
Finally, we define 
\begin{align*}
 \e{t}:=\frac{\pmap{\xbar_t}{\z_t}{\X{\xbar_0}{\z_0}{\p_e+l_0\e{0}}}-\pmap{\xbar_t}{\z_t}{\x_e}}{\gbarnorm[\xbar_t]{\pmap{\xbar_t}{\z_t}{\X{\xbar_0}{\z_0}{\p_e+l_0\e{0}}}-\pmap{\xbar_t}{\z_t}{\x_e}}}\in\cotanspMbar{\xbar_t}
\end{align*}
for some sufficiently small $l_0>0$ such that the above expression is defined. Suppose by contradiction that \eqref{eqn: supporting plane collapses} fails, then there exists $\epsilon>0$ and a sequence of $t_k>0$ going to zero such that 
         \begin{align}
         \epsilon\leq \frac{\dist{\p^k_e}{\plane{\coord{\sublevelset_k}{\xbar_k, \z_k}}{\e{k}}}}{\supsegment{\coord{\sublevelset_k}{\xbar_k, \z_k}}{\e{k}}},\quad \forall\;k\label{eqn: collapsing support plane contradiction}
         \end{align} 
         where for ease of notation, we write $\sublevelset_k:=\sublevelset_{t_k}$, $\xbar_k:=\xbar_{t_k}$, $\z_k:=\z_{t_k}$, $\p^k_e:=\pmap{\xbar_{t_k}}{\z_{t_k}}{\x_e}$, and $\e{k}:=\e{t_k}$. By compactness, we can assume all of these sequences converge, it is clear that $\xbar_k\to \xbar_0$, $\z_k\to\z_0$, and $\e{k}\to \e{0}$ and $\p^k_e\to\p_e$ (in $\cotanspMbar{}$). Now we can see that        
\begin{align}\label{eqn: poor man's blowup term lower bound}
 \lim_{k\to\infty}\supsegment{\coord{\sublevelset_k}{\xbar_k, \z_k}}{\e{k}}=\supsegment{\subzerocoord}{\e{0}}\geq t_0>0
\end{align}
        by our choice of $\e{0}$. 
        Now recalling 	Remark~\ref{rem:b_bet}, we obtain the existence of a sequence $\p_k\in \coord{\sublevelset_{k}}{\xbar_k, \z_k}$ such that for all $k$
        \begin{align*}
        \dist{\p^k_e}{\plane{\coord{\sublevelset_k}{\xbar_k, \z_k}}{\e{k}}}=\innergbar[\xbar_k]{\p_k-\p^k_e}{\e{k}},
        \end{align*}
         by compactness of $\outerdom^{\cl}$ we may assume that $\X{\xbar_k}{\z_k}{\p_k}$ converges to some $\x_\infty\in\outerdom^{\cl}$ as $k\to\infty$; we easily see $\x_\infty\in \subzero$. Then combining with \eqref{eqn: poor man's blowup term lower bound}, rearranging \eqref{eqn: collapsing support plane contradiction}, and passing to the limit, we would obtain
         \begin{align*}
              0< \epsilon t_0\leq \innergbar[\xbar_0]{\pmap{\xbar_0}{\z_0}{\x_\infty}-\p_e}{\e{0}}.
         \end{align*}
However, as $\e{0}\in\normal^0_{\p_e}\paren{\subzerocoord}$ this implies $\pmap{\xbar_0}{\z_0}{\x_\infty}=\p_e$ immediately giving a contradiction. Thus we have shown \eqref{eqn: supporting plane collapses} finishing the proof.

\end{proof}

    \subsection{Proof of Theorem \ref{thm: localization}} From this point on, the rest of the proof is analogous to the argument in \cite{GK14}, specifically the proofs of \cite[Theorem 5.7]{GK14}, and \cite[Lemmas 5.8 and 5.9]{GK14}, using Lemma~\ref{lem: chopping convergence} in place of \cite[Lemma 5.5]{GK14}.
    
    Some points of note. The sets $\spt{\rho}$ and $\spt{\bar\rho}$, and $\partial_c u$ from \cite{GK14} should be replaced by $\innerdom$, $\innertarget$, and $\partial_G u$ respectively, while Theorem~\ref{thm: G-aleksandrov estimate} and Theorem~\ref{thm: Sharp growth} should take the places of \cite[Theorem 4.1, Lemma 3.7]{GK14}. By \eqref{eqn: tiltings are nice} and \eqref{eqn: chopped sections close to exposed point 2} (choosing a small enough $\delta>0$), we can apply Theorem~\ref{thm: G-aleksandrov estimate} to the sections $\sublevelset_t$ when $t$ is sufficiently small from Lemma~\ref{lem: chopping convergence}. Also we see that by \eqref{eqn: lower bound on localization ratios} and \eqref{eqn: tiltings are nice} we will have   
\begin{align*}
 \sup_{\sublevelset_t} \mountain_t+\sup_{\sublevelset_t}(\mountain_t-u)&\leq \Niceupper+\epsilon_0^{-1}(\mountain_t(\x_e)-u(\x_e))
\end{align*}
thus by the second part of \eqref{eqn: chopped sections have interior}, for $t>0$ small enough we obtain \eqref{eqn: section not too deep}; hence we can also apply Theorem~\ref{thm: Sharp growth}.
    
    Finally, the set $S^{\rm big}_t$ appearing in the proof of \cite[Lemma 5.9]{GK14} should be redefined as   
\begin{align*}
 S^{\rm big}_t:=\curly{\x\in \outerdom\mid u(\x)\leq \G(\x, \xbar_t, \H(\x_0, \xbar_t, u(\x_0)))},
\end{align*}
    where $\mountain_t(\cdot)=\G(\cdot, \xbar_t, \z_t)$ and for some choice of $\x_0\not \in \sublevelset_t$. Here, we note that as in \eqref{eqn: first tilting close to tilting}, we have 
\begin{align*}
 \lvert \G(\x, \xbar_t, \H(\x_0, \xbar_t, u(\x_0)))-\mountain_t(\x)\rvert <C(u(x_0)-\mountain_t(x_0))
\end{align*}
thus combining with \eqref{eqn: tiltings are nice} and choosing $\x_0$ close enough to the boundary of $\sublevelset_t$, we can ensure $\Nicelower<\G(\x, \xbar_t, \H(\x_0, \xbar_t, u(\x_0))<\Niceupper$ and \eqref{eqn: section not too deep}, for all $x\in \outerdom$ and $t>0$ small. With this choice of $\x_0$, we are able to apply Theorem~\ref{thm: Sharp growth} to $\G(\x, \xbar_t, \H(\x_0, \xbar_t, u(\x_0))$, and the proof of  \cite[Lemma 5.9]{GK14} can now be followed.
   
   \begin{flushright} $\square$ \end{flushright}

\subsection{Strict convexity} For the remainder of this section we fix $\x_0\in\innerdom^{\interior}$, $\xbar_0\in\Gsubdiff{u}{\x_0}$, and also write
    \begin{align*}
\z_0:=\H(\x_0, \xbar_0, u(\x_0)),\quad    \mountain_0(\cdot):=\G(\cdot, \xbar_0, \z_0),\\
         \p_0:=\pmap{\xbar_0}{\z_0}{\x_0}, \quad 
         \pbar_0:=\pbarmap{\x_0}{u(\x_0)}{\xbar_0}=\pbarmap{\x_0}{\mountain_0(\x_0)}{\xbar_0}.
    \end{align*}
    Additionally, in this section we will be using the Riemannian inner product $\innerg[x_0]{\cdot}{\cdot}$ on $\cotanspM{x_0}$.

    \begin{lem}\label{lem: good direction a3w}
         Suppose that the conditions of Theorem~\ref{thm:strict_convexity} hold and $\contact$ contains more than one point. Then there is some nonzero $\qbar_0 \in \cotanspMbar{x_0}$ such that
         \begin{align}\label{eqn: cone is in interior a3w}
              \paren{B_r(\pbar_0)\setminus B_{\frac{r}{2}}(\pbar_0)}\cap\cone{\qbar_0}{r} \subset \innertargetcoord{\x_0, u(\x_0)}^{\interior}
         \end{align}
         for all sufficiently small and positive $r$. Here, $\cone{\qbar_0}{r}$ denotes the cone
         \begin{align}\label{eqn: cone def a3w}
              \cone{\qbar_0}{r}:= \curly{\pbar\in \outertargetcoord{\x_0, u(\x_0)} \mid r\innerg[x_0]{\pbar-\pbar_0}{\frac{\qbar_0}{\gnorm[x_0]{\qbar_0}}}\geq \gnorm[x_0]{\proj{\qbar_0^\perp}{\pbar-\pbar_0}} },
         \end{align}
         and $\proj{\qbar_0^\perp}{\pbar}$ is the projection of $\pbar$ onto the $(n-1)$--dimensional affine space containing $\pbar_0$, which is $\g{x_0}$-orthogonal to $\qbar_0$. Moreover, the linear function on $\outerdomcoord{\xbar_0, \z_0}$ defined by
         \begin{align}\label{eqn: linear function def}
              \linear{}{p}:= \inner{\p}{\frac{\nondegmatrixinv{x_0}{\xbar_0}{\z_0}\qbar_0}{\gnorm[x_0]{\qbar_0}}}
         \end{align}
         attains a unique maximum on $\contactcoord$.
    \end{lem}

    \begin{proof}
        This proof is essentially identical to that of \cite[Lemma 6.3]{GK14}.
    \end{proof}

    \begin{lem}\label{lem: halfspace inequality}
         Suppose $\qbar_0\in \outertargetcoord{\x_0, u(\x_0)}$ is chosen as in Lemma~\ref{lem: good direction a3w} above, $\linear{}{p}$ is defined by \eqref{eqn: linear function def}, and $\contact$ contains more than one point. Then if $\pmax\in\subzerocoord$ is the unique point where $\linear{}{\cdot}$ attains its maximum over $\subzerocoord$, we have
         \begin{align}\label{eqn: max at boundary a3w}
              \pmax\in \subzerocoord\cap \outerdomcoord{\xbar_0, \z_0}^\bdry.
         \end{align} 
         Additionally we have the inequality
         \begin{align}\label{eqn: halfspace inequality a3w}
              \inf_{x}\linear{}{\pmap{\xbar_0}{\z_0}{\x}}> \linear{}{\pmax}-o(1),\qquad r\to 0,
         \end{align}
         where for each $r>0$ small, the infimum is taken over the set of $\x\not \in\contact$ satisfying
         \begin{align*}
              \coord{\partial_{\G} u(\x)}{\x_0, \u(\x_0)}\cap\cone{\qbar_0}{r}\cap \paren{B_{r}(\pbar_0)\setminus B_{r/2}(\pbar_0)}\neq \emptyset.
         \end{align*}
    \end{lem}

    \begin{proof}
                 Since the maximum of a linear function on a convex set must be attained at at least one of its extremal points, $\pmax$ must be an extremal point of $\contactcoord$. However, since $\contact$ contains more than one point by assumption, Theorem~\ref{thm: localization} yields \eqref{eqn: max at boundary a3w}.
        
                We now work towards the inequality \eqref{eqn: halfspace inequality a3w}. Fix some $r_0>0$ to be determined, take $r\in (0, r_0)$, and let $x\not \in \contact$. Also suppose $\pbar_r\in \coord{\Gsubdiff{u}{\x}}{\x_0, u(\x_0)}\cap \cone{\qbar_0}{r}\cap \paren{B_r(\pbar_0)\setminus B_{\frac{r}{2}}(\pbar_0)}$ and define
\begin{align*}
 \p:&=\pmap{\xbar_0}{\z_0}{\x}, \qquad \xbar_r:=\Xbar{\x_0}{u(\x_0)}{\pbar_r},\qquad \xmax:=\X{\xbar_0}{\z_0}{\pmax},\\
 \xbar(t):&=\Gseg{\xbar_0}{\xbar_r}{\x_0, u(\x_0)},\qquad
 \z(t):=\H(\x_0, \xbar(t), u(\x_0))=\H(\x_0, \xbar(t), \mountain_0(\x_0)),\\
 \mountain_r(\cdot):&=\G(\cdot, \xbar_r, \H(\x, \xbar_r, u(\x))).
\end{align*}
Now since $\u$ is \emph{very nice}, $\mountain_0(\cdot)\in \Niceinterval$ on $\outerdom^{\cl}$. Then by \eqref{DomConv} we can apply \eqref{eqn: linearization upper bound} to find some \emph{very nice} constant $C>0$ such that
\begin{align*}
 \norm{\G(\xmax, \xbar(t), \z(t))-\mountain_0(\xmax)}<Ct\gbarnorm[\xbar_0]{\pmax-\p}\gnorm[\x_0]{\pbar_r-\pbar_0}<Cr_0.
\end{align*}
In particular, if $r_0$ is sufficiently small we must have $(\xmax, \xbar(t), \z(t))\in \gendom$ for all $t\in [0, 1]$, \emph{for any choice of $\pbar_r$ and $\x$}.
Next since $\mountain_r$ is supporting to $u$, we must have for some \emph{very nice} constant $C>0$ that
\begin{align*}
         0&=u(\xmax)-\mountain_0(\xmax)\\
         &\geq\mountain_r(\xmax)-\mountain_0(\xmax)\\
         &=\G(\xmax, \xbar_r, \H(\x_0, \xbar_r, \mountain_r(\x_0)))-\mountain_0(\xmax)\\
         &\geq  -C(\mountain_0(\x_0)-\mountain_r(\x_0))+\G(\xmax, \xbar_r, \H(\x_0, \xbar_r, u(\x_0)))-\G(\xmax, \xbar_0, \H(\x_0, \xbar_0, u(\x_0))).
         \end{align*}
         Here we have used the fact that $u(\x_0)=\mountain_0(\x_0)$, and the \emph{very nice} $C$ arises once again from using the mean value theorem and the facts that both $\mountain_0$ and $\mountain_r$ lie in $\Niceinterval$.

Note that since $\mountain_r(\x)=u(\x)>\mountain_0(\x)$ while $\mountain_r(\x_0)\leq u(\x_0)=\mountain_0(\x_0)$ and $\outerdom$ is assumed path-connected, there exists some $\x_0'\in \outerdom$ such that 
\begin{align*}
 \mountain_r(\x'_0)=\mountain_0(\x'_0).
\end{align*}
Thus by \eqref{DomConv} again we can apply \eqref{eqn: linearization upper bound}, and calculate
\begin{align*}
\norm{\mountain_0(\x_0)-\mountain_r(\x_0)}&=\norm{\G(\x_0, \xbar_r, \H(\x'_0, \xbar_r, \mountain_0(\x'_0)))-\G(\x_0, \xbar_0, \H(\x'_0, \xbar_0, \mountain_0(\x'_0)))}\\
&\leq C\gbarnorm[\xbar_0]{\pmap{\xbar_0}{\H(\x'_0, \xbar_0, \mountain_0(\x'_0))}{\x_0}-\pmap{\xbar_0}{\H(\x'_0, \xbar_0, \mountain_0(\x'_0))}{\x'_0}}\gnorm[\x'_0]{\pbarmap{\x'_0}{\mountain_0(\x'_0)}{\xbar_r}-\pbarmap{\x'_0}{\mountain_0(\x'_0)}{\xbar_0}}\\
&\leq C\gnorm[\x_0]{\pbar_r-\pbar_0}\leq Cr,
\end{align*}
where we have used the bound on $\mountain_0$, \eqref{eqn: target distance comparison}, \eqref{eqn: source distance comparison}, and boundedness of $\outerdom$ to obtain the final inequality.
We may then apply Lemma~\ref{lem: linearization}, \eqref{eqn: linearization lower bound} to conclude
\begin{align}
 Cr&\geq \inner{\pmax-\p}{\nondegmatrixinv{\x_0}{\xbar_0}{\z_0}(\pbar_r-\pbar_0)}\notag\\
 &=\inner{\pmax-\p}{\innerg[x_0]{\pbar_r-\pbar_0}{\frac{\qbar_0}{\gnorm[x_0]{\qbar_0}}}\frac{\nondegmatrixinv{\x_0}{\xbar_0}{\z_0}\qbar_0}{\gnorm[x_0]{\qbar_0}}+\nondegmatrixinv{\x_0}{\xbar_0}{\z_0}\proj{\qbar_0^\perp}{\pbar_r-\pbar_0}}\notag\\
 &\geq \innerg[x_0]{\pbar_r-\pbar_0}{\frac{\qbar_0}{\gnorm[x_0]{\qbar_0}}}\linear{}{\pmax-\p}-C\gnorm[x_0]{\proj{\qbar_0^\perp}{\pbar_r-\pbar_0}}
\end{align}
again for some \emph{very nice} $C>0$. We now prove that
         \begin{align*}
              \innerg[x_0]{\pbar_r-\pbar_0}{\qbar_0}> 0.
         \end{align*} 
Indeed, $\innerg[x_0]{\pbar_r-\pbar_0}{\qbar_0}\geq 0$ as $\pbar_r\in \cone{\qbar_0}{r}$, but $\innerg[x_0]{\pbar_r-\pbar_0}{\qbar_0}=0$ would imply $\pbar_r=\pbar_0$ which would contradict $\pbar_r\not\in B_{r/2}(\pbar_0)$. Thus we may divide by $\innerg[x_0]{\pbar_r-\pbar_0}{\frac{\qbar_0}{\gnorm[x_0]{\qbar_0}}}$, rearrange, and use that $\pbar_r\in \cone{\qbar_0}{r}$ to obtain
         \begin{align*}
              \linear{}{\p}&\geq \linear{}{\pmax}-C\paren{\frac{\gnorm[x_0]{\proj{\qbar_0^\perp}{\pbar_r-\pbar_0}}}{\innerg[x_0]{\pbar_r-\pbar_0}{\frac{\qbar_0}{\gnorm[x_0]{\qbar_0}}}}}-Cr\\
              &\geq \linear{}{\pmax}-Cr,
         \end{align*}
proving \eqref{eqn: halfspace inequality a3w}.
    \end{proof}
     
    \begin{cor}\label{cor: preimage trapping a3w}
         Suppose that the conditions of Lemma~\ref{lem: good direction a3w} hold. Let $\qbar_0\in \cotanspM{x_0}$ and $\pmax\in \subzerocoord\cap\outerdomcoord{\xbar_0, \z_0}^\bdry$ satisfy the conclusions of Lemma~\ref{lem: halfspace inequality}, and let $\cone{\qbar_0}{r}$ be as defined by \eqref{eqn: cone def a3w}. Then given any $\epsilon>0$, there exists $r_\epsilon>0$ such that for any $x\in\outerdom^{\cl}\setminus \contact$ satisfying
         \begin{align*}
              \coord{\Gsubdiff{u}{\x}}{\x_0, u(\x_0)}\cap \cone{\qbar_0}{r_\epsilon}\cap \left(B_{r_\epsilon}(\pbar_0)\setminus B_{\frac{r_\epsilon}{2}}(\pbar_0)\right)\neq \emptyset,
         \end{align*}
         we must have
         \begin{align*}
              \gbarnorm[\xbar_0]{\pmap{\xbar_0}{\z_0}{\x}-\pmax}<\epsilon.
         \end{align*} 
    \end{cor}

    \begin{proof}
         Let $\linear{}{\cdot}$ be defined by \eqref{eqn: linear function def}. The proof is by a compactness argument. Suppose by contradiction that the corollary fails, then for some $\epsilon_0>0$, there is a sequence of $r_k>0$ decreasing to $0$ as $k\to \infty$, and sequences $\{x_k\}_{k=1}^\infty\subset \outerdom^{\cl}\setminus \contact$, and $\pbar_k\in \cone{\qbar_0}{r_k}\cap \left(B_{r_k}(\pbar_0)\setminus B_{\frac{r_k}{2}}(\pbar_0)\right)$ such that 
         \begin{align}
              \pbar_k&\in\coord{\Gsubdiff{u}{\x_k)}}{\x_0, u(\x_0)},\notag\\
              \gbarnorm[\xbar_0]{\pmap{\xbar_0}{\z_0}{\x_k}-\pmax}&\geq \epsilon_0\label{eqn: contradiction assumption a3w}
         \end{align}
         for all $k$. It is clear that $\pbar_k\to\pbar_0$ as $k\to\infty$, and by the compactness of $\outerdomcoord{\xbar_0, \z_0}^{\cl}$ a subsequence $\x_k\to \x_\infty$ for some $x_\infty\in \outerdom^{\cl}$. Writing         
\begin{align*}
 \xmax:=\X{\xbar_0}{\z_0}{\pmax},\qquad \xbar_k:=\Xbar{\x_0}{u(\x_0)}{\pbar_k},
\end{align*}
since $\xmax\in \contact$ we calculate 
         \begin{align*}
              \mountain_0(\xmax)=u(\xmax)
              &\geq \G(\xmax, \xbar_k, \H(\x_k, \xbar_k, u(\x_k)))\\
              &\to \G(\xmax, \xbar_0, \H(\x_\infty, \xbar_0, u(\x_\infty)))
         \end{align*}
         as $k\to \infty$. Then taking $ \G(\x_\infty, \xbar_0, \H(\xmax, \xbar_0, \cdot))$ of both sides we have
\begin{align*}
\mountain_0(\x_\infty)&= \G(\x_\infty, \xbar_0, \H(\xmax, \xbar_0, \mountain_0(\xmax)))\\
&\geq u(\x_\infty)
\end{align*}
 or in other words, $x_\infty\in \contact$. However, since $x_k$ satisfies \eqref{eqn: halfspace inequality a3w} with $r=r_k$, taking $k\to\infty$ implies $\linear{}{\pmap{\xbar_0}{\z_0}{\x_\infty}}\geq \linear{}{\pmax}$. Now by \eqref{eqn: contradiction assumption a3w} we have $\pmap{\xbar_0}{\z_0}{\x_\infty}\neq \pmax$, which contradicts the uniqueness of $\pmax$ as the maximizer of $\linear{}{\cdot}$ over $\subzerocoord$.
    \end{proof}

\subsection{Proof of Theorem~\ref{thm:strict_convexity}} Suppose the theorem fails and the contact set $\contact$ contains more than one point. Since $\innerdom$ is compactly contained in $\outerdom$, we take $\epsilon$ such that
    \begin{align}\label{eqn: smaller than distance}
      0<\epsilon <\dist{\innerdomcoord{\xbar_0, \z_0}}{\outerdomcoord{\xbar_0, \z_0}^\bdry}.
    \end{align}
    Next take $\qbar_0\in \cotanspM{x_0}$ obtained from applying Lemma~\ref{lem: good direction a3w}, and $r_\epsilon>0$ associated to our choice of $\epsilon$ by Corollary~\ref{cor: preimage trapping a3w}. Then by \eqref{eqn: cone is in interior a3w} there exists a set $A\subset \innerdom$ such that         
    \begin{align*}
      \coord{\Gsubdiff{u}{A}}{\x_0, u(\x_0)}= \cone{\qbar_0}{r_\epsilon}\cap \paren{B_{r_\epsilon}(\pbar_0)\setminus B_{\frac{r_\epsilon}{2}}(\pbar_0)},
    \end{align*}
    while by Corollary~\ref{cor: preimage trapping a3w} and \eqref{eqn: smaller than distance} we must have 
    \begin{align*}
      A\subset \subzero\cap \paren{\outerdom^{\cl}\setminus \innerdom}.
    \end{align*}
    Then since $u$ is an Aleksandrov solution (and also recalling \eqref{eqn: volume comparability}), for some constant $C>0$ depending on $\Niceinterval$ we have 
    \begin{align*}
      0  &  <  \Leb{\cone{\qbar_0}{r_\epsilon}\cap \paren{B_{r_\epsilon}(\pbar_0)\setminus B_{\frac{r_\epsilon}{2}}(\pbar_0)}}\\&\leq C \Leb{A\cap \innerdom}\\
         &  \leq C\Leb{\subzero}=0,
    \end{align*}
    thus finishing the proof by contradiction.

\begin{flushright}$\square$\end{flushright}
	
\section{Engulfing property and H\"older regularity for the gradient}\label{section: engulfing}

\subsection{Engulfing property of sections} In this section we shall prove Theorem \ref{thm:C1alpha_regularity}. We follow here a method first introduced by Forzani and Maldonado to obtain explicit $C^{1,\alpha}$ bounds for the real Monge-Amp\'ere equation \cite{FM04}. This was later adapted by Figalli, Kim and McCann to prove $C^{1,\alpha}$ regularity of the potential in optimal transport  \cite[Section 9]{FKM13}.

For $\x\in\innerdom^{\interior}$, $\xbar\in \Gsubdiff{u}{x}$, and $h>0$, we will use the notation
\begin{align}
  \sublevelset(\x,\xbar,h) & := \curly{ y\in\outerdom\mid u(y) \leq \mountain_h(y)},\\
  \mountain_h(\cdot) & :=\G(\cdot,\xbar,\z_h),\\
  \z_h & := \H(\x,\xbar,u(x)+h).
\end{align}
We comment here that since $u$ is assumed \emph{very nice}, for any $h>0$ sufficiently small we will have $\mountain_h\in \Niceinterval$ on all of $\outerdom$. Additionally, Theorem~\ref{thm:strict_convexity} implies we may assume the section $\sublevelset(\x,\xbar,h)$ is contained in a ball of arbitrarily small diameter, entirely contained in $\innerdom^{\interior}$; also condition \eqref{eqn: section not too deep} will hold on any subset of $\sublevelset(\x,\xbar,h)$. As a result we can apply Theorem~\ref{thm: G-aleksandrov estimate} to the sections $\sublevelset(\x,\xbar,h)$ as long as $h$ is small, and Theorem~\ref{thm: Sharp growth} to any $A\subset \sublevelset(\x,\xbar,h)$ satisfying \eqref{eqn: dilation condition}. Furthermore, as $u$ is an Aleksandrov solution this implies that for any subset $A\subset \sublevelset(\x,\xbar,h)$ we will always have $\Leb{A}\sim \Leb{\Gsubdiff{u}{A}}$. We point out here that the strict $\G$-convexity of $u$, Theorem~\ref{thm:strict_convexity}, is essential here, as it allows us to actually apply our Aleksandrov estimate Theorem~\ref{thm: G-aleksandrov estimate} to all sections with small enough height.

Since later, we will be concerned with a dilation of the section $\sublevelset(\x_0, \xbar_0, h)$ with respect to $\pmap{\xbar_0}{\z_h}{\x_0}$ (instead of the center of mass of the section), we begin with a preliminary result showing that $\pmap{\xbar_0}{\z_h}{\x_0}$ is actually fairly close to the center of mass.
\begin{prop}\label{prop:crossections_are_wellcentered}
  There exists a \emph{very nice} $\gamma\in(0,1)$ and $h_0>0$ such that for any $h\in (0, h_0)$,
  \begin{align*}
    \pmap{\xbar_0}{\z_h}{\x_0} \in \gamma \coord{\sublevelset(\x_0,\xbar_0,h)}{\xbar_0,\z_h},
  \end{align*}
  where the dilation above is with respect to the center of mass of $\coord{\sublevelset(\x_0,\xbar_0,h)}{\xbar_0,z_h}$.
\end{prop}

\begin{proof}
 We will write $\sublevelset_h:=\sublevelset(\x_0,\xbar_0,h)$ for the duration of this proof. 
 
 Let us define
  \begin{align*}
    t_0:= \inf  \curly{ t\in[0,1]\;\mid\; \pmap{\xbar_0}{\z_h}{\x_0}\in t\coord{\sublevelset_h}{\xbar_0,\z_h}	 },
  \end{align*}	
our goal is then to prove $t_0\leq \gamma<1$ for some \emph{very nice} $\gamma$; note we may assume, say, $t_0>\frac{1}{2}$ otherwise we are already done. Then by combining Theorem~\ref{thm: G-aleksandrov estimate} and the main result of \cite{FKM13b}, (and recalling that $u$ is an Aleksandrov solution) we obtain a \emph{very nice} $c_0>0$ such that 
  \begin{align}\label{eqn: upper bound in section}
(\mountain_{h}(\x_0)-\u(x_0))^n \leq c_0 \Leb{\sublevelset_h}^2(1-t_0\;)^{\frac{1}{2^{n-1}}}.	
  \end{align}	
On the other hand, (as we have done many times) by the mean value theorem applied in the scalar parameter, if $h<h_0$ with some small, \emph{very nice} $h_0>0$ then for some \emph{very nice} $C$ we can see that
  \begin{align*}
    \mountain_{h}(\x) & = \G(\x,\xbar_0,\H(\x_0,\xbar_0,\u(\x_0)+h))\\
	  & \leq \G(\x,\xbar_0,\H(\x_0,\xbar_0,\u(\x_0)))+Ch = \mountain_0(\x)+Ch,
  \end{align*}	  
  here the fact that $u$ is \emph{very nice} allows us to assume $C$ is also \emph{very nice}. 
Since $\mountain_0\leq u$ everywhere, the above leads to
  \begin{align*}
    \sup\limits_{\sublevelset_h} (\mountain_h-\u) \leq Ch = C(\mountain_{h}(\x_0)-\u(x_0)).
  \end{align*}
By shrinking $h_0$ further if necessary, with a \emph{very nice} dependance, and using that $u$ is \emph{very nice}, we see $\mountain_h$ and $\arbitrary=\X{\xbar_0}{\z_h}{(KM)^{-1}\coord{\sublevelset_h}{\xbar_0,\z_h}}$ will satisfy condition \eqref{eqn: section not too deep}. Thus applying Theorem~\ref{thm: Sharp growth} and combining with the above inequality yields another \emph{very nice} $c_1>0$ such that 
  \begin{align*}
    (\mountain_{h}(\x_0)-\u(x_0))^n  \geq c_1 \Leb{\sublevelset_h}^2.   	
  \end{align*}	
  Combining the above inequality with \eqref{eqn: upper bound in section}, it follows that 
\begin{align*}
 \paren{\frac{c_1}{c_0}}^{2^{n-1}}\leq 1-t_0 \implies t_0\leq 1-\paren{\frac{c_1}{c_0}}^{2^{n-1}}.
\end{align*}
Thus the proposition is proven with the choice $\gamma := 1-\left (\frac{c_1}{c_0}\right )^{2^{n-1}}<1$, which is also seen to be \emph{very nice}.
\end{proof}

The next lemma proves a rather strong property of a solution $u$: sections of different heights are roughly homothetic to one another (in cotangent coordinates). The lemma uses in a crucial way the strict $\G$-convexity of $u$, which guarantees that $\sublevelset(\x,\xbar,h)$ is contained in a neighborhood of $h$ when $h$ is small enough (as discussed at the beginning of the section).

\begin{lem}\label{lem:comparison_of_sections}	
  There exists a \emph{very nice} constant $\beta\in (0, 1)$ such that
\begin{align*}
\coord{\sublevelset(\x,\xbar,h)}{\xbar,\z_{2 h}} \subset \beta\coord{\sublevelset(\x,\xbar, 2 h)}{\xbar,\z_{2 h}},\;\;\forall\;h\in(0,h_0).
\end{align*}
Here the dilation $\beta\coord{\sublevelset(\x,\xbar, 2 h)}{\xbar,\z_{2 h}}$ is with respect to $\pmap{\xbar}{\z_{2h}}{\x_0}$.
\end{lem}

\begin{proof}
Clearly it suffices to show that if $\pmap{\xbar}{\z_{2h}}{y} \in \coord{\sublevelset(\x,\xbar, 2h)}{\xbar,\z_{2 h}}\setminus \paren{\beta\coord{\sublevelset(\x,\xbar,2 h)}{\xbar,\z_{2 h}}}$ then $y\not\in \sublevelset(\x,\xbar,h)$. 

We now work toward this claim, assume we have such a $y$. At the end of the proof, we will wish to take $\beta$ close to $1$, hence there is no harm in assuming from the start that $\beta\geq \frac{1}{2}$. In particular, by Proposition~\ref{prop:crossections_are_wellcentered} above we can see that $\pmap{\xbar_0}{\z_{2h}}{y}$ is outside of some \emph{very nice} dilate of $\coord{\sublevelset(\x,\xbar, 2h)}{\xbar,\z_{2 h}}$ \emph{with respect to its center of mass}. Thus we can once again combine Theorem~\ref{thm: G-aleksandrov estimate} and the main result of \cite{FKM13b} to obtain a \emph{very nice} constant $C$ such that
  \begin{align*}
    \mountain_{2 h}(y)-u(y)\leq C(1-\beta)^{\frac{1}{2^{n-1}}}\Leb{\sublevelset(\x,\xbar,2 h)}^{2/n}.  	  
  \end{align*}	  
  On the other hand, as in the proof of Proposition~\ref{prop:crossections_are_wellcentered} above, for $h<h_0$ we may apply Theorem~\ref{thm: Sharp growth} to obtain $\Leb{\sublevelset(\x,\xbar,2 h)}\leq C(2 h)^{n/2}$, thus combined with the above yields
  \begin{align*}
    \mountain_{2 h}(y)-u(y) \leq C(1-\beta)^{\frac{1}{2^{n-1}}} h.	  
  \end{align*}	  
Rearranging, we have
  \begin{align*}
    u(y) & \geq \mountain_{2 h}(y)-C(1-\beta)^{\frac{1}{2^{n-1}}}h \\
	  & = \G(y,\xbar,\H(\x,\xbar,u(\x)+2h))-2C(1-\beta)^{\frac{1}{2^{n-1}}}h\\
	  & \geq \G(y,\xbar,\H(\x,\xbar,u(\x)+h))+c_0h-2C(1-\beta)^{\frac{1}{2^{n-1}}}h\\
	  & = \mountain_{h}(y)+(c_0 -2C(1-\beta)^{\frac{1}{2^{n-1}}})h  
  \end{align*}	  
  where again, possibly shrinking $h_0$ in a \emph{very nice} manner, again the mean value theorem yields the third line above, with a \emph{very nice} constant $c_0$. 
 Finally by taking $1-\beta$ small enough we have $\u(y)\geq \mountain_{h}(y)$, it then follows that $y\not\in \sublevelset(\x,\xbar,h)$.
\end{proof}

The lemma above says that sections are comparable under affine rescaling (in the right system of cotangent coordinates). The fact that sections are comparable in this manner is a very strong assertion for a $\G$-convex function, and it will imply explicit $C^{1,\alpha}$ estimates (as well as $\G$-convexity estimates).

The next lemma is an analogue of the ``engulfing property'' (see \cite[Theorem 4]{FM04} and \cite[Theorem 9.3]{FKM13} for the Euclidean case and the optimal transport case, respectively).

\begin{lem}\label{lem:engulfing_property}
  There are constants $\Lambda>1$, $t_0>0$ with the following property: if $\x_0,\x_1\in \innerdom$ and $\xbar_0,\xbar_1\in\innertarget$ are such that $\x_1 \in \sublevelset(\x_0,\xbar_0,h)$ for some $h<h_0$, then
  \begin{align*}
    \x_0\in \sublevelset(\x_1,\xbar_1,\Lambda h)	  	
  \end{align*}	  
\end{lem}	

\begin{proof} 
Let $\x^\partial_1\in \sublevelset(\x_0, \xbar_0, 2h)^\bdry$ be such that the $\G$-segment $x(s):=\Gseg{\x_0}{\x^\partial_1}{\xbar, \z_{2h}}$ contains the point $\x_1=\x(s_1)$. Then Lemma \ref{lem:comparison_of_sections} implies that $s_1\leq \beta$. We now claim that for some \emph{very nice} $C>0$,
\begin{align*}
  \G(\x_1,\xbar_1,\H(\x_0,\xbar_1,u(\x_0)+2h))&\leq  \mountain_{2h}(\x_1)+Ch.
\end{align*}
Indeed, if  $\G(\x_1,\xbar_1,\H(\x_0,\xbar_1,u(\x_0)+2h))\leq  \mountain_{2h}(\x_1)$ we are already done. Otherwise we note that since $u$ is \emph{very nice}, we may apply \eqref{QQConv} with $s=s'=s_1$, $\z_0:=\z_{2h}$, and $\x^\partial_1$ in place of $\x_1$ there to conclude 
\begin{align*}
& \G(\x_1,\xbar_1,\H(\x_0,\xbar_1,u(\x_0)+2h))- \mountain_{2h}(\x_1)\\
&\leq \frac{Cs_1}{1-s_1}\paren{\G(\x^\partial_1, \xbar_1, \H(\x_1, \xbar_1, \mountain_{2h}(\x_1)))-\mountain_{2h}(\x^\partial_1)}\\
&\leq \frac{C\beta}{1-\beta}\paren{\G(\x^\partial_1, \xbar_1, \H(\x_1, \xbar_1, \mountain_{2h}(\x_1)))-\mountain_{2h}(\x^\partial_1)}.
\end{align*}
Then since $\mountain_{2h}(\x_1)-u(x_1)\leq 2h$, for small enough $h$ we can again use the mean value theorem to find a \emph{very nice} $C$ such that
\begin{align*}
 &\G(\x^\partial_1, \xbar_1, \H(\x_1, \xbar_1, \mountain_{2h}(\x_1)))-\mountain_{2h}(\x^\partial_1)\\
 &=\G(\x^\partial_1, \xbar_1, \H(\x_1, \xbar_1, u(\x_1)))-\mountain_{2h}(\x^\partial_1)+C(\mountain_{2h}(\x_1)-u(x_1))\\
 &\leq u(\x_1^\partial)-\mountain_{2h}(\x^\partial_1)+Ch\leq Ch,
\end{align*}
proving the claim.

Using this claim, the mean value theorem again, and that $\xbar_0\in \Gsubdiff{u}{\x_0}$ we have a \emph{very nice} $\Lambda>0$ such that
  \begin{align*}
    \G(\x_1,\xbar_1,\H(\x_0,\xbar_1,u(\x_0)+2h))&\leq \G(\x_1,\xbar_0,\H(\x_0,\xbar_0,u(\x_0)+2h))+Ch\\
    &\leq  \G(\x_1, \xbar_0, \H(\x_0, \xbar_0, u(\x_0)))+\Lambda h\\
    &\leq u(\x_1)+\Lambda h.
  \end{align*}	  
  Finally, applying $\G(\x_0, \xbar_1, \H(\x_1, \xbar_1, \cdot))$ to both sides of the above inequality and using the monotonicity of this function in the scalar variable, we obtain  
\begin{align*}
u(\x_0)&< u(\x_0)+2h\\
&=\G(\x_0, \xbar_1, \H(\x_1, \xbar_1, \G(\x_1,\xbar_1,\H(\x_0,\xbar_1,u(\x_0)+2h))))\\
&\leq \G(\x_0, \xbar_1, \H(\x_1, \xbar_1, u(\x_1)+\Lambda h)),
\end{align*}
(as long as $h$ is sufficiently small in a \emph{very nice} manner, the expression $u(\x_1)+\Lambda h$ is indeed contained in the domain of $\G(\x_0, \xbar_1, \H(\x_1, \xbar_1, \cdot))$). 
This proves $\x_0\in \sublevelset(\x_1,\xbar_1,\Lambda h)$.
  \end{proof}

\begin{lem}\label{lem:characterization_engulfing}
  Let $\Lambda$ be as in Lemma \ref{lem:engulfing_property}. There exist \emph{very nice} $\Lambda_1>0$ and $d_0>0$ such that 
  if $\gdist{\x_0}{\x_1}<d_0$, $\xbar_0\in\Gsubdiff{u}{x_0}$, and $\xbar_1\in\Gsubdiff{u}{x_1}$, then
  \begin{align}\label{eqn: monotonicity gain}
  \frac{1}{\Lambda_1}(u(x_1)-\G(x_1, \xbar_0, \H(\x_0, \xbar_0, u(\x_0))))\leq u(\x_0)-\G(\x_0, \xbar_1, \H(\x_1, \xbar_1, u(\x_1))).
  \end{align}
 \end{lem}	
	
\begin{proof}
For $\epsilon>0$ let us write 
\begin{align*}
\tau_\epsilon:=\G(\x_1, \xbar_1, \H(\x_0, \xbar_1, u(\x_0)+\epsilon))-u(\x_1)>0.
\end{align*}
We will eventually take $\epsilon\searrow 0$, so we may assume it is as small as we want, additionally if we assume that $\gdist{\x_0}{\x_1}<d_0$ for some $d_0>0$ depending on the Lipschitz norm of $u$ (which in turn, is controlled by the constant $\lipbound$ in \eqref{Lip}, recall Remark~\ref{rem: nice functions are nice}), we will have $u(\x_0)+\Lambda\tau_\epsilon\in \Niceinterval$. Then 
\begin{align*}
 \G(\x_0, \xbar_1, \H(\x_1, \xbar_1, u(\x_1)+\tau_\epsilon))&=\G(\x_0, \xbar_1, \H(\x_1, \xbar_1, \G(\x_1, \xbar_1, \H(\x_0, \xbar_1, u(\x_0)+\epsilon))))\\
 &=u(\x_0)+\epsilon
\end{align*}
hence $x_0\in \sublevelset(\x_1, \xbar_1, \tau_\epsilon)$, thus by Lemma~\ref{lem:engulfing_property} we must have $\x_1\in \sublevelset(\x_0, \xbar_0, \Lambda\tau_\epsilon)$. This means by the mean value theorem, for a \emph{very nice} $C>0$ we have 
\begin{align*}
 u(\x_1)\leq \G(\x_1, \xbar_0, \H(\x_0, \xbar_0, u(\x_0)+\Lambda\tau_\epsilon))\leq\G(\x_1, \xbar_0, \H(\x_0, \xbar_0, u(\x_0))+C\Lambda\tau_\epsilon,
\end{align*}
or rearranging and taking $\epsilon$ to $0$,
\begin{align*}
 \frac{1}{C\Lambda}(u(\x_1)-\G(\x_1, \xbar_0, \H(\x_0, \xbar_0, u(\x_0)))&\leq\G(\x_1, \xbar_1, \H(\x_0, \xbar_1, u(\x_0)))-u(\x_1).
\end{align*}
Now by using the mean value theorem again, we calculate
\begin{align*}
& \G(\x_1, \xbar_1, \H(\x_0, \xbar_1, u(\x_0)))-u(\x_1)\\
&=\G(\x_1, \xbar_1, \H(\x_0, \xbar_1, u(\x_0)))-\G(\x_1, \xbar_1, \H(\x_1, \xbar_1, u(\x_1)))\\
&=\G(\x_1, \xbar_1, \H(\x_0, \xbar_1, u(\x_0)))-\G(\x_1, \xbar_1, \H(\x_0, \xbar_1, \G(\x_0, \xbar_1, \H(\x_1, \xbar_1, u(\x_1)))))\\
 &\leq C (u(\x_0))-\G(\x_0, \xbar_1, \H(\x_1, \xbar_1, u(\x_1)))).
\end{align*}
The constant $C$ in the final inequality above can seen to be \emph{very nice} for the following reason: by possibly shrinking $d_0$ depending on $\lipbound$ in \eqref{Lip}, we can ensure that $\G(\x_0, \xbar_1, \H(\x_1, \xbar_1, u(\x_1)))$ is sufficiently close to $u(\x_0)$, since $u$ is \emph{very nice} we can ensure $C$ is also \emph{very nice}. Combining these above two inequalities yields \eqref{eqn: monotonicity gain} for $\Lambda_1=C\Lambda$.
\end{proof}		
	
\subsection{Proof of Theorem \ref{thm:C1alpha_regularity}}

 Fix a point $\x_0\in \innerdom^{\interior}$ and an $\xbar_0\in \Gsubdiff{u}{\x_0}$. We shall now show that for some \emph{very nice} $C>0$, 
  \begin{align*}
    u(\x)-\G(\x,\xbar_0,\z_0)\leq \frac{C}{\alpha-\beta}\gdist{\x}{\x_0}^{1+\beta}
  \end{align*}	
for all $x$ in a small neighborhood of $\x_0$, where $\beta<\alpha$ and $\beta\leq \Lambda_1^{-1}$, where $\Lambda_1$ is the constant in \eqref{lem:characterization_engulfing}.
  As $\G(\cdot,\xbar,\z)$ is uniformly $C^{1,\alpha}$ in $\x$,  the $C^{1,\beta}$ regularity of $u$ follows by a standard argument.

  Fix an $\x_1$ with $\gdist{\x_0}{\x_1}<d_0$ and let $\x_g(s)$ be the (unique) unit speed geodesic from $\x_0$ to $\x_1$. The engulfing property, used via Lemma \ref{lem:characterization_engulfing}, will lead us to a differential inequality for $\phi(s)$, where
  \begin{align*}
    \phi(s) := u(\x_g(s))-\G(\x_g(s),\xbar_0,\z_0).	
  \end{align*}	
  To make the idea of the proof clear, let us go over it first in the special case where $\u \in C^1$.\\

 \noindent \textbf{Proof when $u$ is $C^1$.} Let us define
  \begin{align*}
    \xbar_s & := \Gsubdiff{u}{\x_g(s)},\\
	\z_s & := \H(\x_g(s),\xbar_s,\u(\x_g(s)));
  \end{align*}
  note as $u$ is $C^1$, the set $\Gsubdiff{u}{\x_g(s)}$ is actually single valued for each $s$. 
  Differentiating $\phi$ ($\u$ is $C^1$, and the chain rule applies) and using $\xbar_s$, $\z_s$ as defined above, 
  \begin{align*}
    \phi'(s) & = \inner{Du(\x_g(s))-D\G(\x_g(s),\xbar_0,\z_0)}{\xdot_g(s)}\\ 
      & = \inner{D \G(\x_g(s),\xbar_s,\z_s) -DG(\x_g(s),\xbar_0,\z_0))}{\xdot_g(s)}.
  \end{align*}	
  Now, with a $C$ given by the $C^{1,\alpha}$-norm of $\G$ with respect to the first variable (and uniformly in the other two),
  \begin{align*}
    \G(\x_g(s),\xbar_0,\z_0) & \leq \G(\x_0,\xbar_0,\z_0) + s\inner{D\G(\x_g(s),\xbar_0,\z_0)}{\xdot_g(s)}+Cs^{1+\alpha}\\
    \G(\x_g(s),\xbar_s,\z_s) & \geq \G(\x_0,\xbar_s,\z_s) + s\inner{D\G(\x_g(s),\xbar_s,\z_s)}{\xdot_g(s)}-Cs^{1+\alpha}.
  \end{align*}
  Thus
  \begin{align*}
	  & s\inner{D\G(\x_g(s),\xbar_s,\z_s)-D\G(\x_g(s),\xbar_0,\z_0)}{\xdot_g(s)}+2Cs^{1+\alpha}\\
	  &\geq \G(\x_g(s),\xbar_s,\z_s) - \G(\x_g(s),\xbar_0,\z_0) -\G(\x_0,\xbar_s,\z_s) + \G(\x_0,\xbar_0,\z_0)\\
	  &= u(\x_g(s)) - \G(\x_g(s),\xbar_0,\z_0) -\G(\x_0,\xbar_s,\z_s) + u(\x_0)\\
	  &=\phi(s)-\G(\x_0,\xbar_s,\z_s) + u(\x_0).
  \end{align*}
  On the other hand, by Lemma~\ref{lem:characterization_engulfing},
  \begin{align*}
u(\x_0)-\G(\x_0,\xbar_s,\z_s) & \geq \Lambda_1^{-1}(\u(\x_g(s)) - \G(\x_g(s),\xbar_0,\z_0))\\
	  &  = \Lambda_1^{-1}\phi(s),
  \end{align*}
thus by combining the above and rearranging, we have
  \begin{align*}
    s\phi'(s)-(1+\Lambda_1^{-1})\phi(s)+Cs^{1+\alpha}\geq 0. 
  \end{align*}	  
  Using the elementary identity,		
  \begin{align*}
    \frac{d}{ds}\left ( \frac{\phi(s)}{s^{1+\beta}} \right ) = \frac{1}{s^{2+\beta}}\left (s\phi'(s)-(1+\beta)\phi(s) \right )	
  \end{align*}			
  with any choice of $\beta$ such that $\beta<\alpha$ and $\beta\leq \Lambda_1^{-1}$ yields 
  \begin{align*}
    \frac{d}{ds}\paren{\frac{\phi(s)}{s^{1+\beta}}}+ Cs^{\alpha-\beta-1}\geq 0.	
  \end{align*}			
  In particular, for any $s<\gdist{\x_0}{\x_1}$, by integration we obtain
  \begin{align}\label{eqn: differential inequality holder}		
    \frac{\phi(\gdist{\x_0}{\x_1})}{\gdist{\x_0}{\x_1}^{1+\beta}}+\frac{C\gdist{\x_0}{\x_1}^{\alpha-\beta}}{\alpha-\beta}\geq  \frac{\phi(s)}{s^{1+\beta}}+ \frac{Cs^{\alpha-\beta}}{\alpha-\beta} \geq  \frac{\phi(s)}{s^{1+\beta}}.
  \end{align}

At the same time, $\phi$ is bounded by a \emph{very nice} constant due to the fact that $u$ is \emph{very nice}. Thus rearranging \eqref{eqn: differential inequality holder} we have for some \emph{very nice} $C>0$ that
  \begin{align*}
    \phi(s)&\leq C\paren{\frac{1}{\alpha-\beta}+\frac{1}{\gdist{\x_0}{\x_1}^{1+\beta}}}s^{1+\beta}\leq \frac{Cs^{1+\beta}}{(\alpha-\beta)\gdist{\x_0}{\x_1}^{1+\beta}};
  \end{align*}
in terms of $u$ this says that if $\x$ lies on the geodesic connecting $\x_0$ to $\x_1$, then
  \begin{align*}
    0\leq u(\x)-\G(\x,\xbar_0,\z_0) \leq\frac{C}{(\alpha-\beta)\gdist{\x_0}{\x_1}^{1+\beta}}\gdist{\x}{\x_0}^{1+\beta}.
  \end{align*}	
  By considering $\x_1$ in a small geodesic sphere centered at $\x_0$, we obtain the $C^{1,\beta}$ estimate for $\u$.\\
  	
  \noindent \textbf{Proof for arbitrary $\u$.} Let us define  
\begin{align*}
 U(p):&=u(\X{\xbar_0}{\z_0}{\p})-\G(\X{\xbar_0}{\z_0}{\p}, \xbar_0, \z_0)\\
 \p_g(s):&=\pmap{\xbar_0}{\z_0}{\x_g(s)}.
\end{align*}
  Then $\p_g: \real \to \cotanspMbar{\xbar_0}$ is $C^1$ on $(0, 1)$, and $U: \cotanspMbar{\xbar_0}\to \real$ is Lipschitz, and since $\phi(s)=U(\p_g(s))$, one of the chain rules for the Clarke subdifferential \cite[Theorem 2.3.10]{Clarke90} combined with \cite[Proposition 3.3.4 and Corollary]{Clarke90} yields for all $s\in (0, 1)$,
  \begin{align}\label{eqn: clarke containment}
\clarke{\phi}{s} \subset \ch \curly{\inner{q_s}{D\p_g(s)}\mid q_s\in \clarke{U}{\p_g(s)}}.
  \end{align}	
Again by \cite[Theorem 2.5.1]{Clarke90} combined with the representation $\subdiff{U}{\p}=\curly{\lim_{k\to\infty}DU(\p_k)\mid \p_k\to\p}$ (proven as in Corollary~\ref{cor: local to global}), we see that $\subdiff{U}{\p_g(s)}=\clarke{U}{\p_g(s)}$ for all $s$. A tedious but routine calculation now yields that 
\begin{align*}
q_s\in\subdiff{U}{\p_g(s)}&\iff q_s=\paren{D_{\p_g(s)}\X{\xbar_0}{\z_0}{\cdot}}^\ast (\pbar_s-\Dx\G(\x_g(s), \xbar_0, \z_0))
\end{align*}
 for some $\pbar_s\in \subdiff{u}{x(s)}$. Here $\empty^\ast$ is the transpose map, defined by duality using the evaluation map by 
\begin{align*}
 \inner{\paren{D_{\p_g(s)}\X{\xbar_0}{\z_0}{\cdot}}^\ast w^\ast }{v}=\inner{w^\ast}{\paren{D_{\p_g(s)}\X{\xbar_0}{\z_0}{\cdot}}v},\quad \forall\; w^\ast\in \cotanspM{\X{\xbar_0}{\z_0}{\p_g(s)}},\ v\in \tansp{\p_g(s)}{\cotanspMbar{\xbar_0}}.
\end{align*}
 Recalling that $\pmap{\xbar_0}{\z_0}{\cdot}$ is the inverse of $\X{\xbar_0}{\z_0}{\cdot}$ we can rewrite \eqref{eqn: clarke containment} as 
\begin{align*}
\clarke{\phi}{s} \subset \ch \curly{\inner{\pbar_s-\Dx\G(\x_g(s), \xbar_0, \z_0)}{\xdot(s)}\mid \pbar_s\in \subdiff{u}{\x_g(s)}}.
  \end{align*}	
  Now for each $s\in (0, 1)$ any $\pbar_s\in \subdiff{u}{\x_g(s)}$, and \emph{any} choice of $\xbar_s\in \Gsubdiff{u}{\x_g(s)}$, a similar argument as the case when $u$ is assumed $C^1$ yields
  \begin{align*}
s\inner{\pbar_s-D\G(\x_g(s),\xbar_0,\z_0)}{\xdot(s)}&=s\inner{D\G(\x_g(s),\xbar_s,\z_s)-D\G(\x_g(s),\xbar_0,\z_0)}{\xdot(s)}\\
& \geq (1+\Lambda_1^{-1})(\u(\x_g(s))- \G(\x_g(s),\xbar_0,\z_0) )-Cs^{1+\alpha}\\
	&  =(1+\Lambda_1^{-1})\phi(s)-Cs^{1+\alpha},
  \end{align*}
thus it follows that for $s\in (0, 1)$,
  \begin{align*}		
\clarke{\phi}{s} \subset\paren{s^{-1}(1+\Lambda_1^{-1})\phi(s)-Cs^{\alpha},\infty}. 		
  \end{align*}
Finally, for those $s$ at which $\phi$ is differentiable \cite[Proposition 2.2.2]{Clarke90} gives $\phi'(s)\in \clarke{\phi}{s}$, hence at such $s$ we have
  \begin{align*}		
    s\phi'(s) \geq (1+\Lambda_1^{-1})\phi(s)-Cs^{1+\alpha}	.
  \end{align*}		
Since $\phi$ is Lipschitz, the above inequality holds for a.e. $s\in [0, \gdist{\x_0}{\x_1}]$, from here on the proof follows by integration, arguing as in the case where $u$ is $C^1$.

  \begin{flushright}$\square$\end{flushright}		
	
\section{The analogue of the MTW tensor, and its relation to \eqref{QQConv}-\eqref{DualQQConv}}\label{section: G3w implies QQConv}

    In \cite{Tru14}, Trudinger defines the condition \eqref{G3w} below. The condition reduces to the Ma-Trudinger-Wang ((MTW) or sometimes (A3w)) condition on the cost function $c$ from the theory of optimal transport (the case $\G(\x, \xbar, \z)=-c(\x, \xbar)-\z$, see Section~\ref{section: OT}), which is central in questions of regularity. The (MTW) condition and certain stronger variations were used in \cite{MTW05, LTW10} to prove local and in \cite{TW09} to prove global \emph{a priori} $C^2$ estimates of solutions to optimal transport problems, leading to $C^{2, \alpha}$ regularity. In \cite{FKM13}, it is shown under (MTW) that pointwise estimates of Aleksandrov type hold, which can be used to show $C^{1, \alpha}$ regularity of solutions to optimal transport (see also \cite{Liu09}). In a previous paper \cite{GK14}, we introduce conditions called (QQConv), which can be used as starting points to again prove Aleksandrov type estimates in optimal transport, but with lower ($C^3$) regularity of the cost function. In \cite{GK14} we also show the (MTW) condition implies (QQConv); the conditions \eqref{QQConv} and \eqref{DualQQConv} we introduce in this paper reduce to (QQConv) in the optimal transport case. It is also shown by Loeper in \cite{Loe09} that the (MTW) condition leads to certain geometric consequences, and when the cost function is $C^4$, it is \emph{necessary} to obtain regularity of solutions to optimal transport. Trudinger uses \eqref{G3w} in \cite{Tru14} to obtain \emph{a priori} $C^2$ estimates for $\G$-convex solutions of the generated Jacobian equations \eqref{eqn: generated Jacobian equation}.

The goal of this section is to demonstrate that conditions \eqref{QQConv} and \eqref{DualQQConv} are reasonable: in the case of smoother generating function $\G$, the conditions follow from Trudinger's regularity condition \eqref{G3w} below.

In this section, we assume that $\G$ is $C^4$, in the sense indicated in Theorem~\ref{thm: G3w implies QQconv} (i.e., all derivatives up to order 4, where at most two derivatives fall on any single variable at once, are continuous). We first define Trudinger's condition \eqref{G3w}.
\begin{DEF}\label{DEF: G3w}
Fix $\x\in\outerdom$, $(\pbar, \u)\in \curly{(\Dx \G, \G)(\x, \xbar, \z)\mid (\x, \xbar, \z)\in\gendom}$, and a local coordinate system near $\x$ (denoted by $\x^i$) on $\M$; and define 
\begin{align*}
 \A{i}{j}(\x, \pbar, \u):=\G_{\x^i\x^j}(\x, \Xbar{\x}{\u}{\pbar}, \Z{\x}{\pbar}{\u})
\end{align*}
where subscripts refer to coordinate derivatives.

We say \emph{$\G$ satisfies \eqref{G3w}}  if for any such triple $(\x, \pbar, \u)$, and any $\V\in\tanspM{x}$ and $\eta\in\cotanspM{x}$ satisfying $\inner{\eta}{\V}=0$, we have
\begin{align}\label{G3w}\tag{G3w}
\tensor[(\x, \pbar, \u)]{\eta}{\eta}{\V}{\V}:= \Dxx_{\pbar_k\pbar_l}\A{i}{j}(\x, \pbar, \u)\V^i\V^j\eta_k\eta_l\geq 0,
\end{align}
here $\pbar_i$ denote the coordinates induced by $\x^i$ on the cotangent bundle $\cotanspM{}$, and $\Dxx_{\pbar_k\pbar_l}$ are second derivatives with respect to these coordinates. Likewise, we say \emph{$\G$ satisfies \eqref{G3s}} if there is some $c_0>0$ such that for any triple $(\x, \pbar, \u)$ and any $\V\in\tanspM{x}$ and $\eta\in\cotanspM{x}$ satisfying $\inner{\eta}{\V}=0$, we have
\begin{align}\label{G3s}\tag{G3s}
\tensor[(\x, \pbar, \u)]{\eta}{\eta}{\V}{\V}:= \Dxx_{\pbar_k\pbar_l}\A{i}{j}(\x, \pbar, \u)\V^i\V^j\eta_k\eta_l\geq c_0\gnorm{\eta}^2 \gnorm{\V}^2.
\end{align}
Similarly, for fixed $\xbar\in\outertarget$, $\z\in\real$, and $\p\in \curly{-\frac{\Dbar\G}{\Gz}(\x, \xbar, \z)\mid (\x, \xbar, \z)\in\gendom}$, define 
\begin{align*}
 \Adual{k}{l}(\p, \xbar, \z):&=-\brackets{\Dbar_{\xbar^k}\paren{\frac{\Dbar_{\xbar^l} \G}{\Gz}}(\x, \xbar, \z)+\Dx_\z\paren{\frac{\Dbar_{\xbar^k} \G}{\Gz}}(\x, \xbar, \z)\p_l}_{\x=\X{\xbar}{\z}{\p}}\\
 &=\H_{\xbar^k\xbar^l}(\X{\xbar}{\z}{\p}, \xbar, \G(\X{\xbar}{\z}{\p}, \xbar, \z)).
\end{align*}
Then we say \emph{$\G$ satisfies \eqref{dual-G3w}} if for any such triple $(\p, \xbar, \z)$, and any $\Vbar\in\tanspMbar{\xbar}$ and $\etabar\in\cotanspMbar{\xbar}$ satisfying $\inner{\etabar}{\Vbar}=0$, we have
\begin{align}\label{dual-G3w}\tag{G$3^*$w}
\dualtensor[(\p, \xbar, \z)]{\etabar}{\etabar}{\Vbar}{\Vbar}:= \Dxx_{\p_i\p_j}\Adual{k}{l}(\p, \xbar, \z)\etabar_i\etabar_j\Vbar^k\Vbar^l\geq 0.
\end{align}
\end{DEF}

\begin{rem}\label{rem: tensorial condition}
 We recall here that for any fixed $(\x, \pbar, \u)$, the expression in the definition of $\tensor[(\x, \pbar, \u)]{\cdot}{\cdot}{\cdot}{\cdot}$ is a $(2, 2)$-tensor over $\tanspM{x}\times\tanspM{x}\times\cotanspM{x}\times\cotanspM{x}$; hence actually independent of choice of coordinate systems. Indeed, fix $(\x_0, \pbar_0, \u_0)$. We take two coordinate systems $x^i$ and $y^i$ near $x_0$ on $\M$; $\x^i$ and $y^i$ induce coordinates on $\cotanspM{}$ locally near $(\x_0, \pbar_0)$, we denote these by $(\x^i, \pbar_i)$ and $(y^i, \qbar_i)$; the relation being $\qbar_i=\pbar_k\pdiff[x^k]{\y^i}$. Then if $(\x, \pbar)=(\y, \qbar)$ are coordinate representations of the same point in $\cotanspM{}$,
\begin{align*}
\A{i}{j}(\x, \pbar, \u)&=\G_{\y^\alpha\y^\beta}(\y,\Xbar{\x}{\u}{\pbar}, \Z{\x}{\pbar}{\u})\pdiff[\y^\alpha]{\x^i}\pdiff[\y^\beta]{\x^j}+\G_{\y^\alpha}(\y,\Xbar{\x}{\u}{\pbar}, \Z{\x}{\pbar}{\u})\twicepdiff[\y^\alpha]{\x^i}{\x^j}\\
&=\G_{\y^\alpha\y^\beta}(\y,\Xbar{\x}{\u}{\pbar}, \Z{\x}{\pbar}{\u})\pdiff[\y^\alpha]{\x^i}\pdiff[\y^\beta]{\x^j}+\G_{\x^\beta}(\x,\Xbar{\x}{\u}{\pbar}, \Z{\x}{\pbar}{\u})\pdiff[\x^\beta]{\y^\alpha}\twicepdiff[\y^\alpha]{\x^i}{\x^j}\\
&=\G_{\y^\alpha\y^\beta}(\y,\Xbar{\x}{\u}{\pbar}, \Z{\x}{\pbar}{\u})\pdiff[\y^\alpha]{\x^i}\pdiff[\y^\beta]{\x^j}+\pbar_\beta\pdiff[\x^\beta]{\y^\alpha}\twicepdiff[\y^\alpha]{\x^i}{\x^j}.
\end{align*}
Since the second term in the last expression above is linear in the $\pbar_i$ coordinates, it will vanish under two differentiations in those variables. Hence we obtain
\begin{align*}
 \Dx_{\pbar_k\pbar_l}\A{i}{j}(\x, \pbar, \u)&= \Dx_{\pbar_k\pbar_l}G_{\y^\alpha\y^\beta}(\y,\Xbar{\x}{\u}{\pbar}, \Z{\x}{\pbar}{\u})\pdiff[\y^\alpha]{\x^i}\pdiff[\y^\beta]{\x^j}\\
 &=\Dx_{\qbar_r\qbar_s}G_{\y^\alpha\y^\beta}(\y,\Xbar{\y}{\u}{\qbar}, \Z{\y}{\qbar}{\u})\pdiff[\y^\alpha]{\x^i}\pdiff[\y^\beta]{\x^j}\pdiff[\x^k]{\y^r}\pdiff[\x^l]{\y^s}
\end{align*}
and thus the expression giving $\tensor[(\x, \pbar, \u)]{\cdot}{\cdot}{\cdot}{\cdot}$ transforms according to the transformation law for $(2, 2)$-tensors as claimed. A similar argument holds for $\dualtensor[(\p, \xbar, \z)]{\cdot}{\cdot}{\cdot}{\cdot}$.
\end{rem}

There are numerous known cases in optimal transport when \eqref{G3w} holds: see \cite{Loe11, KM12, LV10, FR09, DG10, FRV12} (examples arising in Riemannian geometry) and \cite{MTW05, TW09, LeeMcCann11, LeeLi12} (more general cost functions). Some cases that do not fit into the optimal transport framework that satisfy \eqref{G3w} are demonstrated in \cite[Section 4]{JiangTru14} (near-field parallel beam reflection and near-field refraction) and \cite{Tru14} (near-field point source reflector with flat target). One interesting future direction is to explore how in the near-field point source reflector problem with a more general target, the conditions for regularity presented in \cite{KarWang10} fit within the framework of generated Jacobian equations.

We now devote the remainder of this section toward proving Theorem~\ref{thm: G3w implies QQconv}. In order to do so, we first obtain \eqref{DualQQConv} by adapting the strategy in \cite[Lemma 2.23]{GK14}, which is originally motivated by a computation in \cite[Proposition 4.6]{KM10} stemming from the optimal transport case. The major difference is, of course that the added nonlinear dependency of $\G$ on the scalar variable $\z$ highly complicates matters. Additionally, due to the restriction of the domains $\gendom$ and $\gendomdual$, we must be extremely careful in our computations to ensure that the quantities involved are always well-defined; this is another subtlety not present in the optimal transport case. Finally, since the ``source'' and 	``target'' domains $\outerdom$ and $\outertarget$ do not play an exactly symmetric role, we must carefully exploit the duality relation between $\G$ and $\H$ to obtain \eqref{QQConv} from \eqref{DualQQConv}. Henceforth we will fix a compact subinterval $[\QQConvlower, \QQConvupper]\subset (\Gfivelower, \Gfiveupper)$. 

We comment here that in the lemma below, we will actually make use of \eqref{dual-G3w} instead of \eqref{G3w}; by \cite[Theorem 3.1]{Tru14} it is known that if $\G$ satisfies \eqref{G3w}, then it also satisfies \eqref{dual-G3w}, and vice versa.
\begin{lem}\label{lem: cataclysm}
Suppose $\G$, $\outerdom$, and $\outertarget$ satisfy the hypotheses of Theorem~\ref{thm: G3w implies QQconv} and let $\u_0\in [\QQConvlower, \QQConvupper]$, $x_0$, $x_1 \in \outerdom^{\cl}$, $\xbar_0$, $\xbar_1\in\outertarget^{\cl}$, and $\xbar(t):=\Gseg{\xbar_0}{\xbar_1}{\x_0, \u_0}$. 
Also define
\begin{align*}
f(t) := \G(\x_1, \xbar(t), \H(\x_0, \xbar(t), \u_0)),\;\;t \in[0,1].
\end{align*}
Then, if $(\x_1, \xbar(t), \H(\x_0, \xbar(t), \u_0))\in\gendom$ for all $t\in[0, 1]$, there is a constant $C>0$ depending on $[\QQConvlower, \QQConvupper]$, various derivatives of $\G$, $\outerdom$, and $\outertarget$, but independent of $\x_0$, $\x_1$, $\xbar_0$, and $\xbar_1$ such that
\begin{align*}
f^{\prime\prime}(t) \geq -C\norm{f^{\prime}(t)}, \;\;\forall \;t \in(0,1).	
\end{align*}		
\end{lem}

\begin{proof}
First by \eqref{DualDomConv} and \eqref{G5}, note that $\xbar(t)$ is well-defined and contained in $\outertarget^{\cl}$, and in particular 
\begin{align*}
(\x_0,\xbar(t_0) ,\H(\x_0, \xbar(t), \u_0))\in\gendom, \qquad \forall\;t\in[0, 1].
\end{align*}
For $s$, $t$, $t_0 \in [0,1]$ we introduce the function
\begin{align*}
g(t,s;t_0)=-\frac{\G(\x(s; t_0), \xbar(t), \z(t))}{\Gz(\x(s; t_0), \xbar(t_0), \z(t_0))},
\end{align*}	
where
\begin{align*}
\pbar_0:&=\Dx \G (\x_0, \xbar_0, \H(\x_0, \xbar_0, \u_0)), \qquad
\pbar_1:=\Dx \G (\x_0, \xbar_1, \H(\x_0, \xbar_1, \u_0)),\\
 \z(t):&=\H(\x_0, \xbar(t), \u_0)=\Z{\x_0}{(1-t)\pbar_0+t\pbar_1}{\u_0},
 \end{align*}
 \begin{align*}
\p_0(t_0):&=-\frac{\Dbar \G}{\Gz}(\x_0,\xbar(t_0) ,\z(t_0)),\qquad
\p_1(t_0):=-\frac{\Dbar \G}{\Gz}(\x_1,\xbar(t_0) ,\z(t_0)),\\
 \x(s; t_0):&=\X{\xbar(t_0)}{\z(t_0)}{(1-s)\p_0(t_0)+s\p_1(t_0)}.
\end{align*}
Since $(\x_1,\xbar(t_0) ,\z(t_0))\in\gendom$ for all $t_0\in[0, 1]$ by assumption, \eqref{DomConv} ensures that $\x(s; t_0)=\Gseg{\x_0}{\x_1}{\xbar(t_0), \z(t_0)}$ is well-defined and in particular,
\begin{align}\label{eqn: segment family well-defined}
(\x(s; t_0), \xbar(t_0), \z(t_0))\in\gendom 
\end{align}
 for all $s$, $t_0\in[0, 1]$.
 Also by \eqref{DualTwist} and the definitions of $\p_0(t_0)$, $\p_1(t_0)$, and $\x(s; t_0)$, we must have
\begin{align*}
\x(0; t_0)\equiv \x_0,&\qquad \x(1; t_0)\equiv \x_1
\end{align*}
independent of $t_0\in[0,1]$, thus
\begin{align*}
f(t) = -\Gz(\x_1, \xbar(t_0), \z(t_0))g(t,1;t_0), \;\;\forall\;t,t_0\in[0,1].	
\end{align*}
We shall now show that for some constant $C>0$, which will be the same one as in the statement of this lemma,
\begin{align}\label{eqn: auxiliary convex function}
\vertbar{\twicepdiffsamevar{t}g(t,s;t_0)}_{t=t_0}+s^2C\norm{\inner{\p_1(t_0)-\p_0(t_0)}{\xbardot(t_0)}}\end{align}
is a convex function of $s\in[0, 1]$ which vanishes to first order at $s=0$; hence it is nonnegative for all $s\in[0, 1]$. Taking $s=1$ we then obtain
\begin{align*}
f^{\prime\prime}(t_0) &= -\Gz(\x_1, \xbar(t_0), \z(t_0))\vertbar{\twicepdiffsamevar{t}g(t,1;t_0)}_{t=t_0} \\
&\geq C\Gz(\x_1, \xbar(t_0), \z(t_0))\norm{\inner{\p_1(t_0)-\p_0(t_0)}{\xbardot(t_0)}}, \;\;\forall\;t_0\in[0,1].	
\end{align*}
Further, we shall see by \eqref{eqn: sderivative is h} below with $s=1$,
\begin{align*}
\norm{f^{\prime}(t_0)} &= -\Gz(\x_1, \xbar(t_0), \z(t_0))\norm{\inner{\p_1(t_0)-\p_0(t_0)}{\xbardot(t_0)}},	\;\;\forall\;t_0\in[0,1],
\end{align*}
since $\Gz<0$; thus the lemma will be proved once we check the claims made on  \eqref{eqn: auxiliary convex function}; this we do in four steps. 

\bigskip
\noindent \emph{Step 1.} First we show that for every $t_0\in[0,1]$ we have
\begin{align}
\pdiff{t}g(t,s;t_0)\vert_{t=t_0} &= s\inner{\p_1(t_0)-\p_0(t_0)}{\xbardot(t_0)}.\label{eqn: sderivative is h}
\end{align}
Indeed by \eqref{eqn: zdot representation}, 
\begin{align}
\pdiff{t}g(t,s; t_0)
&=-\frac{\inner{\Dbar \G(\x(s; t_0), \xbar(t), \z(t))}{\xbardot(t)}+\Gz(\x(s; t_0), \xbar(t), \z(t))\zdot(t)}{\Gz(\x(s; t_0), \xbar(t_0), \z(t_0))}\notag\\
 &=\frac{\Gz(\x(s; t_0), \xbar(t), \z(t))}{\Gz(\x(s; t_0), \xbar(t_0), \z(t_0))}\inner{-\frac{\Dbar \G}{\Gz}(\x(s; t_0), \xbar(t), \z(t))+\frac{\Dbar\G}{\Gz}(\x_0, \xbar(t), \z(t))}{\xbardot(t)},\label{eqn: t derivative expression}
\end{align}
thus taking $t=t_0$ we have \eqref{eqn: sderivative is h}.

\bigskip
 \noindent \emph{Step 2.} Now, we make a series of calculations in the flavor of \cite[Proposition 4.6]{KM10}; for our second step we show the zeroth and first order parts of the expression \eqref{eqn: auxiliary convex function} are zero at $s=0$. First by further differentiating \eqref{eqn: t derivative expression} in $t$ and taking $t=t_0$, we see
\begin{align}
 &\vertbar{\twicepdiffsamevar{t}g(t,s;t_0)}_{t=t_0} \notag\\
&=s\inner{\p_1(t_0)-\p_0(t_0)}{\xbarddot(t_0)}\notag\\
&+\frac{s\inner{\p_1(t_0)-\p_0(t_0)}{\xbardot(t_0)}}{\Gz(\x(s; t_0),\xbar(t_0), \z(t_0))}\vertbar{\pdiff{t}\Gz(\x(s; t_0),\xbar(t), \z(t))}_{t=t_0}\notag\\
&+\vertbar{\pdiff{t}\inner{-\frac{\Dbar \G}{\Gz}(\x(s; t_0), \xbar(t), \z(t))}{\xbardot(t_0)}}_{t=t_0}+\vertbar{\pdiff{t}\inner{\frac{\Dbar\G}{\Gz}(\x_0, \xbar(t), \z(t))}{\xbardot(t_0)}}_{t=t_0}\notag\\
&=I+II+III+IV.\label{eqn: second derivative expression}
\end{align}
Since $\x(0; t_0)\equiv \x_0$, we immediately see that for any $t$, $t_0\in[0, 1]$, taking $s=0$ in the above expression,
\begin{align*}
 \vertbar{\twicepdiffsamevar{t}g(t,s;t_0)}_{s=0, t=t_0}=0.
\end{align*}
On the other hand,
\begin{align*}
\vertbar{\pdiff{s}g(t, s; t_0)}_{s=0}
&=-\frac{\inner{\Dx\G(\x_0, \xbar(t), \z(t))}{\xdot(0; t_0)}}{\Gz(\x_0, \xbar(t_0), \z(t_0))}-\rho(t_0)\G(x_0, \xbar(t), \z(t))\\
&=\frac{\inner{(1-t)\pbar_0+t\pbar_1}{\xdot(0; t_0)}}{\Gz(\x_0, \xbar(t_0), \z(t_0))}-\rho(t_0)\u_0
\end{align*} 
where $\rho(t_0)$ is some expression independent of $t$, thus differentiating the above twice in $t$ we see
\begin{align*}
 \vertbar{\frac{\partial^3}{\partial s\partial t^2}g(t,s;t_0)}_{s=0, t=t_0}=0
\end{align*}
for all $t$, $t_0\in[0, 1]$.

 \bigskip
\noindent \emph{Step 3.} Next note that since $\G$ satisfies \eqref{G3w}, by \cite[Theorem 3.1]{Tru14} it satisfies \eqref{dual-G3w}. Then as $\dualtensor{\cdot}{\cdot}{\cdot}{\cdot}$ is a $(2, 2)$-tensor by Remark~\ref{rem: tensorial condition}, \eqref{dual-G3w} (and recalling \eqref{eqn: segment family well-defined}) implies there exists a constant $C>0$ depending only on $\outerdom$, $\outertarget$, $[\QQConvlower, \QQConvupper]$, and derivatives of $\G$ and $\H$ up to order $4$ (independent of $s$ and $t_0$) such that,
\begin{align}
T^*_{s; t_0}
:&=\dualtensor[((1-s)\p_0(t_0)+s\p_1(t_0), \xbar(t_0), \z(t_0))]{\p_1(t_0)-\p_0(t_0)}{\p_1(t_0)-\p_0(t_0)}{\xbardot(t_0)}{\xbardot(t_0) } \notag\\
&\geq-C \norm{\inner{\p_1(t_0)-\p_0(t_0)}{\xbardot(t_0)}}.\label{eqn: MTW bound}
\end{align}
\bigskip
\noindent \emph{Step 4.} Finally, we work toward showing the convexity of \eqref{eqn: auxiliary convex function} in the $s$ variable. Toward this end, we first claim that
\begin{align}\label{eqn: mixed fourth derivative}
&\vertbar{\frac{\partial^4}{\partial s^2\partial t^2}g(t,s;t_0)}_{t=t_0}=T^*_{s; t_0}\notag\\
&+\inner{\p_1(t_0)-\p_0(t_0)}{\xbardot(t_0)}\twicepdiffsamevar{s}\paren{\frac{s\inner{\Gzz[(2-s)\p_0(t_0)+s\p_1(t_0)]+2\Dbar\Gz}{\xbardot(t_0)}}{\Gz}},
\end{align}
where the arguments of $\Dbar \Gz$, $\Gzz$, and $\Gz$ are $(\x(s; t_0), \xbar(t_0), \z(t_0))$.
Since $I+IV$ in \eqref{eqn: second derivative expression} are affine in the variable $s$, after taking two derivatives in $s$ the terms vanish and do not appear in the expression \eqref{eqn: mixed fourth derivative}. On the other hand (hereafter, derivatives of $\G$ are to be evaluated at $(x(s; t_0), \xbar(t_0), \z(t_0))$; we suppress this notation for brevity),
\begin{align*}
 III&=-\inner{\Dbar\paren{\frac{\Dbar \G}{\Gz}}\xbardot(t_0)+\paren{\frac{\Dbar \G}{\Gz}}_\z\zdot(t_0)}{\xbardot(t_0)}\\
 &=-\inner{\Dbar\paren{\frac{\Dbar \G}{\Gz}}\xbardot(t_0)}{\xbardot(t_0)}-\inner{\paren{\frac{\Dbar \G}{\Gz}}_\z}{\xbardot(t_0)}\inner{\p_0(t_0)}{\xbardot(t_0)}
\end{align*}
 by \eqref{eqn: zdot representation}, while
\begin{align*}
\vertbar{\pdiff{t}\Gz(\x(s; t_0),\xbar(t), \z(t))}_{t=t_0}&=\inner{\Dbar\Gz}{\xbardot(t_0)}+\Gzz\zdot(t_0)\\
&=\inner{\Dbar\Gz}{\xbardot(t_0)}+\Gzz\inner{\p_0(t_0)}{\xbardot(t_0)},
\end{align*}
thus
\begin{align*}
 II+III&=-\inner{\Dbar\paren{\frac{\Dbar \G}{\Gz}}\xbardot(t_0)}{\xbardot(t_0)}-\inner{\paren{\frac{\Dbar \G}{\Gz}}_\z}{\xbardot(t_0)}\inner{\p_0(t_0)}{\xbardot(t_0)}\\
 &+\frac{s\inner{\p_1(t_0)-\p_0(t_0)}{\xbardot(t_0)}}{\Gz} (\inner{\Dbar \Gz}{\xbardot(t_0)} +\Gzz\inner{\p_0(t_0)}{\xbardot(t_0)})\\
&=-\inner{\Dbar\paren{\frac{\Dbar \G}{\Gz}}\xbardot(t_0)}{\xbardot(t_0)}-\inner{\paren{\frac{\Dbar \G}{\Gz}}_\z}{\xbardot(t_0)}\inner{(1-s)\p_0(t_0)+s\p_1(t_0)}{\xbardot(t_0)}\\
&+\frac{2s \inner{\Dbar \Gz}{\xbardot(t_0)}}{\Gz}\inner{\p_1(t_0)-\p_0(t_0)}{\xbardot(t_0)}\\
&+\frac{s\Gzz\inner{(2-s)\p_0(t_0)+s\p_1(t_0)}{\xbardot(t_0)}}{\Gz}\inner{\p_1(t_0)-\p_0(t_0)}{\xbardot(t_0)}.
 \end{align*}
 By comparing to \eqref{dual-G3w} we see differentiating the first two terms above twice in $s$ yields the $T^*_{s; t_0}$ term in \eqref{eqn: mixed fourth derivative}, we obtain the full expression \eqref{eqn: mixed fourth derivative}.
 
 Now the absolute value of the final factor multiplying $\inner{\p_1(t_0)-\p_0(t_0)}{\xbardot(t_0)}$ in \eqref{eqn: mixed fourth derivative} has an upper bound depending on the quantities $\Norm{\Dx\Gz}$, $\Norm{\Dxx\Gz}$, $\Norm{\Dx\Gzz}$, $\Norm{\Dxx\Gzz}$, $\Norm{\Dx\Dbar \Gz}$, $\Norm{\Dxx\Dbar \Gz}$, $\norm{\Gz}^{-1}$, $\gnorm[\x(s; t_0)]{\xdot(s; t_0)}$, $\gnorm[\x(s; t_0)]{\xddot(s; t_0)}$, $\gbarnorm[\xbar(t_0)]{\p_0(t_0)}$, $\gbarnorm[\xbar(t_0)]{\p_1(t_0)}$, and $\gbarnorm[\xbar(t_0)]{\xbardot(t_0)}$; by the compactness of $\outerdom^{\cl}$ and $\outertarget^{\cl}$, all of these quantities have a uniform upper bound depending on the interval $[\QQConvlower, \QQConvupper]$ but independent of $\x_0$, $\x_1$, $\xbar_0$, and $\xbar_1$ (also through \eqref{Nondeg} via \eqref{eqn: xdot representation}). Thus combining \eqref{eqn: mixed fourth derivative} with \eqref{eqn: MTW bound} we obtain
\begin{align*}
\vertbar{\frac{\partial^2}{\partial s^2\partial t^2}g(t,s;t_0)}_{t=t_0}  \geq -2C\norm{\inner{\p_1(t_0)-\p_0(t_0)}{\xbardot(t_0)}},\;\;\forall\;s\in(0,1)
\end{align*}
for some $C>0$ with only dependencies as claimed in the statement of the lemma. This last inequality is nothing else but the fact that the auxiliary function  \eqref{eqn: auxiliary convex function} is indeed convex in $s$ for any fixed $t_0$.
\end{proof}

\begin{cor}\label{cor: gLp redux}
Let $f(t)$ be as in the previous lemma. Then there exists a $M\geq 1$ depending on $[\QQConvlower, \QQConvupper]$ but independent of $\x_0$, $\x_1$, $\xbar_0$, and $\xbar_1$ such that
\begin{align*}
f(t)-f(0)\leq \frac{Mt}{1-t'}\brackets{f(1)-f(t')}_+,\;\;\forall\; t\in [0, 1],\ t'\in [0,1).
\end{align*}
\end{cor}
\begin{proof}
First, we shall show that
\begin{align}\label{eqn:gLpredux f' positivity propagation}
  \textnormal{ if } \exists\; t^*\in[0,1] \textnormal{ s.t. } f^{\prime}(t^*)>0, \textnormal{ then } f^{\prime}(t)>0,\;\;\forall\;t\in (t^*,1].
\end{align}
To see this, let us adapt an argument found at the end of the proof of \cite[Theorem 12.46]{Vil09} as follows: suppose that \eqref{eqn:gLpredux f' positivity propagation} is false, then $f^{\prime}(t)\leq 0$ for some $t \in(t^*, 1]$. In this case, there exists $t_0\in (t^*,1]$ which is the first zero of $f'$ after $t^*$, that is
\begin{align*}
  t_0:=\inf\curly{t\in[t^*, 1]\mid f^{\prime}(t)=0}. 
\end{align*}  
Since $f^{\prime}(t)>0$ for $t\in(t^*, t_0)$, Lemma~\ref{lem: cataclysm} gives that $\diff{t}\log{ f^{\prime}(t)}=\frac{f^{\prime\prime}(t)}{f^{\prime}(t)}\geq -C$ for any $t\in (t^*,t_0)$. Integrating this inequality yields
\begin{align*}
  \log{f^{\prime}(t)}\geq \log{f^{\prime}(t^*)}-C(t-t^*),\;\;\forall\;t\in (t^*,t_0).
\end{align*}
Taking $t\nearrow t_0$, we see that $\log{f^{\prime}(t)}$ remains bounded from below, thus $f'(t_0)>0$, which is a contradiction. One key consequence of \eqref{eqn:gLpredux f' positivity propagation} that we exploit is that $f$ cannot have any \emph{strict} local maxima in $(0, 1)$. 

We now return to the main inequality, fixing $t\in [0, 1]$ and $t'\in[0, 1)$ we consider two cases, according to whether $f(1)> f(t')$ or $f(1)\leq f(t')$.

First suppose $f(1)\leq f(t')$. We must then have $f(0)\geq f(t')$, otherwise this will contradict \eqref{eqn:gLpredux f' positivity propagation}. Now suppose that $f(t)>f(0)$ as otherwise there is nothing to prove, then this easily would imply the existence of some strict local maximum of $f$ in $(0, 1)$, which is a contradiction. Thus we obtain the inequality in this case.

It remains to consider the main case when $f(1)>f(t')$, to handle this case we follow a refined version of the argument in \cite[Lemma 2.23]{GK14}. Again assume $f(t)>f(0)$, and temporarily assume $f'(t')>0$. Consider the functions
\begin{align*}
\ftil(\ttil):&=f(t\ttil),\\
\fhat(\ttil):&=f(t'+\ttil(1-t')),
\end{align*}
then by Cauchy's mean value theorem, for some $\ttil\in[0, 1]$ we have
\begin{align}
 \frac{f(t)-f(0)}{f(1)-f(t')}&=\frac{\ftil(1)-\ftil(0)}{\fhat(1)-\fhat(0)}\notag\\
 &=\frac{\ftil^\prime(\ttil)}{\fhat^\prime(\ttil)}\notag\\
 &=\frac{tf^\prime(t\ttil)}{(1-t')f^\prime(t'+\ttil(1-t'))}.\label{eqn: cauchy mean value}
\end{align}
Since $t'+\ttil(1-t')>t'$ and we have assumed $f'(t')>0$, \eqref{eqn:gLpredux f' positivity propagation}  guarantees that $f^\prime(t'+\ttil(1-t'))>0$. This and \eqref{eqn: cauchy mean value} in turn imply that $f^\prime(t\ttil)>0$ as well. On the other hand, since $\ttil$, $t'\leq 1$ it is clear that $0\leq t\ttil\leq t'+\ttil(1-t')\leq 1$. Thus by \eqref{eqn:gLpredux f' positivity propagation} again we see $f'>0$ on $[t\ttil, t'+\ttil(1-t')]$ and we can integrate the inequality in Lemma~\ref{lem: cataclysm} over this interval. This yields
\begin{align*}
 \frac{f^\prime(t\ttil)}{f^\prime(t'+\ttil(1-t'))}\leq e^{C(t'+\ttil(1-t')-t\ttil)}\leq e^C=: M
\end{align*}
(note that $M\geq 1$). Combined with \eqref{eqn: cauchy mean value} we obtain the desired inequality when $f'(t')>0$.

Finally suppose that $f'(t')\leq 0$; by \eqref{eqn:gLpredux f' positivity propagation} we must have $f(t')\leq f(0)$. Let 
\begin{align*}
t_0=\sup\curly{t''\in [0, 1]\mid f(t'')=f(t')},
\end{align*}
 clearly $t'\leq t_0<1$ and for small enough $\epsilon>0$ we have $f'(t_0+\epsilon)>0$. We cannot have $t\leq t_0$, as  $f(t_0)=f(t')\leq f(0)<f(t)$ and this would imply the existence of a strict local maximum of $f$ on $(0, t_0)$. Thus $t\in (t_0, 1]$ and we redo the calculation leading to \eqref{eqn: cauchy mean value} with the function $f_\epsilon(\ttil):=f(t_0+\epsilon+\ttil(1-t_0-\epsilon))$ in place of $f$, $t_\epsilon:=\frac{t-t_0-\epsilon}{1-t_0-\epsilon}$ replacing $t$, and $0$ in place of $t'$ to obtain 
\begin{align*}
 \frac{f(t)-f(t_0+\epsilon)}{f(1)-f(t_0+\epsilon)}= \frac{f_\epsilon(t_\epsilon)-f_\epsilon(0)}{f_\epsilon(1)-f_\epsilon(0)}\leq \frac{t_\epsilon f_\epsilon^\prime(t_\epsilon\ttil)}{f_\epsilon^\prime(\ttil)}=\frac{t_\epsilon f^\prime(t_0+\epsilon+\ttil(t-t_0-\epsilon))}{f^\prime(t_0+\epsilon+\ttil(1-t_0-\epsilon))}\leq Mt_\epsilon;
\end{align*}
we are able to obtain the final inequality as $f^\prime(t_0+\epsilon)>0$, hence we may integrate as in the previous case of \eqref{eqn: cauchy mean value} for the bound. Finally taking $\epsilon$ to $0$ and using that
\begin{align*}
 f(t)-f(0)&\leq f(t)-f(t')=f(t)-f(t_0),\\
 \lim_{\epsilon \to 0}t_\epsilon &=\frac{t-t_0}{1-t_0}\leq t\leq \frac{t}{1-t'}
\end{align*}
we obtain the inequality in this case.

\end{proof}

\begin{rem}\label{rem: twists are symmetric}
 Note that \eqref{DualTwist} implies for fixed $\xbar$, the mapping $(\x, \u)\mapsto (\Dbar \H, \H)(\x, \xbar, \u)$ is injective on the set $\curly{(\x, u)\mid (\x, \xbar, \u)\in\gendomdual}$: suppose for $\xbar$ fixed, and $(\x_1, \u_1)$ and $(\x_2, \u_2)$ in this set we have 
\begin{align*}
 \Dbar\H(\x_1, \xbar, \u_1)&=\Dbar\H(\x_2, \xbar, \u_2),\\
 \H(\x_1, \xbar, \u_1)&=\H(\x_2, \xbar, \u_2)=:\z.
\end{align*}
Then, $\u_i=\G(\x_i, \xbar, \z)$ and $\Dbar\H(\x_i, \xbar, \G(\x_i, \xbar, \z))=-\frac{\Dbar \G}{\Gz}(\x_i, \xbar, \z)$ for $i=1$, $2$, hence the first line above is equivalent to 
\begin{align*}
-\frac{\Dbar \G}{\Gz}(\x_1, \xbar, \z)=-\frac{\Dbar \G}{\Gz}(\x_2, \xbar, \z).
\end{align*}
Since $(\x_i, \xbar, \z)=(\x_i, \xbar, \H(\x_i, \xbar, \u_i))\in\gendom$, by \eqref{DualTwist} we must have $x_1=x_2$. But then clearly also $\u_1=\u_2$, hence we have injectivity. Moreover, if $\U{\xbar}{\p}{\z}:=\G(\X{\xbar}{\z}{\p}, \xbar, z)$, then we have 
\begin{align*}
 \Dbar\H(\X{\xbar}{\z}{\p}, \xbar, \U{\xbar}{\p}{\z})&=\p,\\
 \H(\X{\xbar}{\z}{\p}, \xbar, \U{\xbar}{\p}{\z})&=z.
\end{align*}

Also note that \eqref{Twist} implies that for fixed $(\x, \u)$, the mapping $\xbar\mapsto-\frac{\Dx \H}{\Hu}(\x, \xbar, \u)$ is injective on the set $\curly{\xbar\mid (\x, \xbar, \u)\in\gendomdual}$: fix $\x$, $\u$ and say $\xbar_1\neq\xbar_2$ with $(\x, \xbar_i, \u)\in\gendomdual$ (thus $(\x, \xbar_i, \H(\x, \xbar_i, \u))\in\gendom$) for $i=1$, $2$. Now $\G(\x, \xbar_i, \H(\x, \xbar_i, \u))=\u$ for $i=0$, $1$,  thus by \eqref{Twist}  it must be that $\Dx\G(\x, \xbar_1, \H(\x, \xbar_1, \u))\neq \Dx\G(\x, \xbar_2, \H(\x, \xbar_2, \u))$. Since $-\frac{\Dx \H}{\Hu}(\x, \xbar_i, \u)=\Dx\G(\x, \xbar_i, \H(\x, \xbar_i, \u))$, this gives the claim. Moreover,
\begin{align*}
 -\frac{\Dx \H}{\Hu}(\x, \Xbar{\x}{\u}{\pbar}, \u)&=\Dx\G(\x,\Xbar{\x}{\u}{\pbar}, \H(\x, \Xbar{\x}{\u}{\pbar}, \u) )\\
 &=\Dx\G(\x, \Xbar{\x}{\u}{\pbar}, \Z{\x}{\pbar}{\u})=\pbar.
\end{align*}
\end{rem}

With the above Corollary~\ref{cor: gLp redux} in hand, \eqref{DualQQConv} will be immediate. To obtain \eqref{QQConv}, we first show a ``dual'' version of \eqref{DualQQConv} for the function $\H$ (\eqref{HQQConv} below). We then exploit the relation between $\G$ and $\H$, along with monotonicity properties in the scalar variables to translate this to \eqref{QQConv}.

\begin{proof}[Proof of Theorem~\ref{thm: G3w implies QQconv}]
 Corollary~\ref{cor: gLp redux} immediately implies \eqref{DualQQConv}.
 
 To obtain \eqref{QQConv}, first fix a subinterval $[\QQConvlower, \QQConvupper]\subset (\Gfivelower, \Gfiveupper)$, and let $\x_0$, $\x_1\in\outerdom$, $\xbar_1$, $\xbar_0\in\outertarget$, $\z_0\in\real$ with $\G(\x(s), \xbar_0, \z_0)\in[\QQConvlower, \QQConvupper]$ for $s\in[0, 1]$, and $\x(s):=\Gseg{\x_0}{\x_1}{\xbar_0, \z_0}$. $\x(s)$ is well-defined and remains in $\outerdom^{\cl}$ for all $s\in[0, 1]$ by \eqref{G5} and \eqref{DomConv}. Keeping in mind Remark~\ref{rem: twists are symmetric}, we can follow the proofs of Lemma~\ref{lem: cataclysm} and Corollary~\ref{cor: gLp redux} with the roles of $\H$ and $\G$, $\outerdom$ and $\outertarget$, and $T$ and $T^*$ switched to obtain an analogous version of \eqref{DualQQConv}, i.e. there exists a constant $M\geq 1$ depending on $[\QQConvlower, \QQConvupper]$ but not on $\x_0$, $\x_1$, $\xbar_0$, and $\xbar_1$ such that
\begin{align*}
  &\H(\x(s), \xbar_1, \G(\x(s), \xbar_0, \z_0))-\H(\x_0, \xbar_1, \G(\x_0, \xbar_0, \z_0))\label{HQQConv}\tag{$\H^*$-QQConv}\\
  &\leq \frac{Ms}{1-s'}\brackets{\H(\x_1, \xbar_1, \G(\x_1, \xbar_0, \z_0))-\H(\x(s'), \xbar_1, \G(\x(s'), \xbar_0, \z_0))}_+,\qquad \forall\; s, s'\in[0, 1).
\end{align*}

Indeed, we can redo the proof of Lemma~\ref{lem: cataclysm} with the functions
\begin{align*}
 f(s, t; s_0):&=-\frac{\H(\x(s), \xbar(t; s_0), \u(s))}{\Hu(\x(s_0), \xbar(t; s_0), \u(s_0))}\\
 g(s):&=-\Hu(\x(s_0), \xbar_1, \u(s_0))f(s, 1; s_0),
\end{align*}
with
\begin{align*}
\p_0:&=\Dbar\H(\x_0, \xbar_0, \G(\x_0, \xbar_0, \z_0)),\qquad
\p_1=\Dbar\H(\x_1, \xbar_0, \G(\x_1, \xbar_0, \z_0)),\\
\u(s)&=\U{\xbar_0}{(1-s)\p_0+s\p_1}{\z_0}=\G(\x(s), \xbar_0, \z_0),
\end{align*}
\begin{align*}
\pbar_0(s_0):&=-\frac{\Dx \H}{\Hu}(\x(s_0), \xbar_0, \u(s_0)),\qquad
\pbar_1(s_0):=-\frac{\Dx \H}{\Hu}(\x(s_0), \xbar_1, \u(s_0)),\\
 \xbar(t; s_0):&=\Xbar{\x(s_0)}{\u(s_0)}{(1-t)\pbar_0(s_0)+t\pbar_1(s_0)}.
\end{align*}
Since
$\u(s_0)\in[\QQConvlower, \QQConvupper]\subset(\Gfivelower, \Gfiveupper)$ for any $s_0\in[0, 1]$, \eqref{DualDomConv} implies that $\xbar(t; s_0)$ is well-defined and remains in $\outertarget^{\cl}$ for all $t$, $s_0\in[0, 1]$.
As before, (using \eqref{G3w} in place of \eqref{dual-G3w}) we eventually obtain the existence of a constant $C>0$ with the correct dependencies for which $g''(s)\geq -C\norm{g'(s)}$ for all $s\in[0, 1]$, and the same proof as Corollary~\ref{cor: gLp redux} yields \eqref{HQQConv}.

Let us write $\z_1(s):=\H(\x(s), \xbar_1, \G(\x(s), \xbar_0, \z_0))$, and fix some $s\in[0, 1]$, $s'\in [0, 1)$. Rearranging \eqref{HQQConv}, taking $\G(\x(s), \xbar_1, \cdot)$ of both sides, and using that $\Gz<0$ we obtain
\begin{align*}
 &\G(\x(s), \xbar_1, \z_1(0))\notag\\
 &\leq \G(\x(s), \xbar_1, \z_1(s)-\frac{Ms}{1-s'}\;\brackets{\z_1(1)-\z_1(s')}_{+})\notag\\
&\leq \G(\x(s), \xbar_1, \H(\x(s), \xbar_1, \G(\x(s), \xbar_0, \z_0)))+\sup\norm{\Gz}\frac{Ms}{1-s'}\;\brackets{\z_1(1)-\z_1(s')}_{+}\notag\\
&=\G(\x(s), \xbar_0, \z_0)+\sup\norm{\Gz}\frac{Ms}{1-s'}\;\brackets{\z_1(1)-\z_1(s')}_{+},
\end{align*}
where here the supremum of $\norm{\Gz}$ is over $(\x, \xbar)\in\outerdom^{\cl}\times\outertarget^{\cl}$ and 
\begin{align*}
\z_1(s)-\frac{Ms}{1-s'}\;\brackets{\z_1(1)-\z_1(s')}_{+}\leq z\leq \z_1(s).
\end{align*}
Now since $\G(\x(s), \xbar_0, \z_0)$ remains in the interval $[\QQConvlower, \QQConvupper]$, by using \eqref{Nondeg} and \eqref{eqn: xdot representation}, we see there is a finite upper bound $C_1>0$ depending only on $\G$ (through $\H$), $\outerdom$, $\outertarget$, and $[\QQConvlower, \QQConvupper]$ such that
\begin{align*}
 \z_1(1)-\z_1(s')&=\int_{s'}^1 \diff{s}\z_1(s)ds\leq C_1(1-s').
\end{align*}
Thus the supremum of $\norm{\Gz}$ can be taken over the interval $[\z_1(s)-C_1M, \z_1(s)]$, and in turn for a $C_2>0$ with the same dependencies as $C_1$ we find
\begin{align}
  \G(\x(s), \xbar_1, \z_1(0))-\G(\x(s), \xbar_0, \z_0)\leq C_2\frac{Ms}{1-s'}\brackets{\H(\x_1, \xbar_1, \G(\x_1, \xbar_0, \z_0))-\z_1(s')}_{+}.\label{eqn: intermediate QQConv}
\end{align}
First suppose $\G(\x_1, \xbar_1, \z_1(s'))\leq \G(\x_1, \xbar_0, \z_0)$, we then calculate that 
\begin{align*}
 \z_1(s')&=\H(\x_1, \xbar_1, \G(\x_1, \xbar_1, \z_1(s')))\\
 &\geq \H(\x_1, \xbar_1, \G(\x_1, \xbar_0, \z_0)),
\end{align*}
hence by \eqref{eqn: intermediate QQConv} we obtain \eqref{QQConv} in this case. 

Otherwise,
\begin{align*}
&\brackets{\H(\x_1, \xbar_1, \G(\x_1, \xbar_0, \z_0))-\z_1(s')}_{+}\\
&=\H(\x_1, \xbar_1, \G(\x_1, \xbar_0, \z_0))-\H(\x_1, \xbar_1, \G(\x_1, \xbar_1, \z_1(s')))\\
&\leq \sup\norm{\Hu}\norm{\G(\x_1, \xbar_1, \z_1(s'))-\G(\x_1, \xbar_0, \z_0)}\\  
&=\sup\norm{\Hu}\brackets{\G(\x_1, \xbar_1, \z_1(s'))-\G(\x_1, \xbar_0, \z_0)}_{+}
\end{align*}
where the supremum above is over $\u\in[\G(\x_1, \xbar_0, \z_0), \G(\x_1, \xbar_1, \z_1(s'))]$. Again, this supremum then has a finite upper bound $C_3>0$ depending only on $\G$, $\outerdom$, $\outertarget$, and $[\QQConvlower, \QQConvupper]$, and by \eqref{eqn: intermediate QQConv} we obtain \eqref{QQConv} with the constant $\max\curly{1, C_2C_3M}$. 

\end{proof}

\bibliography{PJEreg}
\bibliographystyle{alpha}
\end{document}